\RequirePackage[l2tabu, orthodox]{nag} 
\documentclass[11pt, a4paper]{article}

\title{Complex moment-based methods\\for differential eigenvalue problems\thanks{This work was supported in part by the Japan Society for the Promotion of Science (JSPS), Grants-in-Aid for Scientific Research (Nos.~JP18K13453, JP19KK0255, JP20K14356, and JP21H03451).}}
\author{Akira Imakura\thanks{Faculty of Engineering, Information and Systems, University of Tsukuba, 1-1-1 Tennodai, Tsukuba, Ibaraki 305-8573, Japan} \thanks{\texttt{imakura@cs.tsukuba.ac.jp}} \and Keiichi Morikuni\footnotemark[2] \thanks{\texttt{morikuni@cs.tsukuba.ac.jp}} \and Akitoshi Takayasu\footnotemark[2] \thanks{\texttt{takitoshi@risk.tsukuba.ac.jp}}}
\date{}

\usepackage{amsmath, amsfonts, amssymb, amsthm, latexsym, bm}
\usepackage[all, warning]{onlyamsmath} 
\usepackage[bookmarks=true, bookmarksnumbered=true, bookmarkstype=toc, colorlinks={true}, pdfauthor={Akira Imakura, Keiichi Morikuni, Akitoshi Takayasu}, pdfdisplaydoctitle=true, pdfkeywords={differential eigenvalue problem solve-then-discretize paradigm higher-order complex moments ordinary differential equations error bounds}, pdfpagemode=UseNone, pdfstartview=FitH, pdfsubject={}, pdftitle={Complex moment-based methods for differential eigenvalue problems}, setpagesize={false}, urlcolor={red}]{hyperref}
\usepackage[T1]{fontenc}
\usepackage[top=1in, bottom=1in, left=1in, right=1in]{geometry}
\usepackage{exscale, fix-cm, lmodern, textcomp}
\usepackage{algorithm,algorithmicx}
\usepackage{algpseudocode}
\usepackage{empheq}
\usepackage{booktabs}
\usepackage[subrefformat=parens]{subcaption}
\usepackage{mathrsfs}
\usepackage{capt-of}

\usepackage{mathtools}
\mathtoolsset{showonlyrefs=true}

\hypersetup{pdfpagemode=UseNone}

\theoremstyle{plain}
\newtheorem{theorem}{Theorem}[section]
\newtheorem{remark}{Remark}[section]

\numberwithin{equation}{section}
\makeatletter
	
	\@addtoreset{equation}{section}
\makeatother

\makeatletter
  
  \@addtoreset{algorithm}{section}
\makeatother

\makeatletter
  
  \@addtoreset{figure}{section}
\makeatother

\makeatletter
  
  \@addtoreset{table}{section}
\makeatother

\allowdisplaybreaks

\begin{document}
\maketitle
\begin{abstract}
This paper considers computing partial eigenpairs of differential eigenvalue problems (DEPs) such that eigenvalues are in a certain region on the complex plane.
Recently, based on a ``solve-then-discretize'' paradigm, an operator analogue of the FEAST method has been proposed for DEPs without discretization of the coefficient operators.
Compared to conventional ``discretize-then-solve'' approaches that discretize the operators and solve the resulting matrix problem, the operator analogue of FEAST exhibits much higher accuracy; however, it involves solving a large number of ordinary differential equations (ODEs).
In this paper, to reduce the computational costs, we propose operation analogues of Sakurai--Sugiura-type complex moment-based eigensolvers for DEPs using higher-order complex moments and analyze the error bound of the proposed methods.
We show that the number of ODEs to be solved can be reduced by a factor of the degree of complex moments without degrading accuracy, which is verified by numerical results.
Numerical results demonstrate that the proposed methods are over five times faster compared with the operator analogue of FEAST for several DEPs while maintaining almost the same high accuracy.
This study is expected to promote the ``solve-then-discretize'' paradigm for solving DEPs and contribute to faster and more accurate solutions in real-world applications.
\end{abstract}
\section{Introduction}
This paper considers solving differential eigenvalue problems (DEPs)
\begin{equation}
	\mathcal{A} u_i = \lambda_i \mathcal{B} u_i, \quad \lambda_i \in \Omega \subset \mathbb{C}
	\label{eq:gep}
\end{equation}
with boundary conditions, where $\mathcal{A}$ and $\mathcal{B}$ are linear, ordinary differential operators acting on functions from a Hilbert space $\mathcal{H}$ and $\Omega$ is a prescribed simply connected open set.
This type of problems appears in various fields such as physics~\cite{polizzi2009density,kestyn2016pfeast} and materials science~\cite{iwase2017efficient,huang2021efficient, kurz2020solving}.
Here, $\lambda_i$ and $u_i$ is an eigenvalue and the corresponding eigenfunction, respectively.
We assume that the boundary $\Gamma$ of $\Omega$ is a rectifiable, simple closed curve and that the spectrum of \eqref{eq:gep} is discrete and does not intersect $\Gamma$, while only $m$ finite eigenvalues counting multiplicities are in $\Omega$.
We also assume that there are eigenfunctions of \eqref{eq:gep} that form a basis for the invariant subspace associated with $\lambda_i \in \Omega$.
\par
A conventional way to solve \eqref{eq:gep} is to discretize the operators $\mathcal{A}$ and $\mathcal{B}$ and solve the resulting matrix eigenvalue problem using some matrix eigensolver, e.g., the QZ and Krylov subspace methods \cite{chatelin2012eigenvalues}.
Fine discretization can reduce discretization error but lead to the formation of large matrix eigenvalue problems.
Owing to parallel efficiency, complex moment-based eigensolvers are practical choices for large eigenvalue problems such as Sakurai--Sugiura' s approach~\cite{sakurai2003projection} and FEAST eigensolvers~\cite{polizzi2009density}.
This class of eigensolvers constructs an approximation of the target invariant subspace using a contour integral and computes an approximation of the target eigenpairs using a projection onto the subspace.
Because of the high efficiency of parallel computation of the contour integral \cite{kestyn2016pfeast,iwase2017efficient}, which is the most time-consuming part, complex moment-based eigensolvers have attracted considerable attention.
\par
In contrast to the above ``discretize-then-solve'' paradigm, a ``solve-then-discretize'' paradigm emerged, motivated by mathematical software Chebfun~\cite{driscoll2014chebfun}.
Chebfun enables highly adaptive computation with operators and functions in the same manner as matrices and functions.
This paradigm has extended numerical linear algebra techniques in finite dimensional spaces to infinite-dimensional spaces~\cite{battles2004extension,trefethen2010householder,olver2014practical,townsend2015continuous,gilles2019continuous,mohr2021full}.
Under the circumstances, an operator analogue of FEAST was recently developed~\cite{horning2020feast} for solving \eqref{eq:gep} and dealing with operators~$\mathcal{A}$ and $\mathcal{B}$ without their discretization\footnote{An algorithm of a FEAST-like eigensolver for solving DEPs \eqref{eq:gep} without the discretization of the operators $\mathcal{A}$ and $\mathcal{B}$ was demonstrated in 2013 in the online document of Chebfun \cite{driscoll2014chebfun}.}. 
On one hand, the operator analogue of FEAST exhibits much higher accuracy than methods based on the traditional ``discretize-then-solve'' paradigm.
On the other hand, a large number of ordinary differential equations (ODEs) must be solved for the construction of invariant subspaces, which is computationally expensive, although the method can be efficiently parallelized.
\par
In this paper, we propose operation analogues of Sakurai--Sugiura's approach for DEPs~\eqref{eq:gep} in the ``solve-then-discretize'' paradigm.
The difference between the operator analogue of FEAST and the proposed methods lies in the order of complex moments used: the operator analogue of FEAST used only complex moments of order zero, whereas the proposed methods use complex moments of higher order.
The difference enables the proposed methods to reduce the number of ODEs to be solved by a factor of the degree of complex moments without degrading accuracy.
The proposed methods can be extended to higher dimensions in a straightforward manner for simple geometries.
\par
The remainder of this paper is organized as follows.
Section~\ref{sec:matrix_solvers} briefly introduces the complex moment-based matrix eigensolvers.
In Section~\ref{sec:propose}, we propose operation analogues of Sakurai--Sugiura's approach for DEPs~\eqref{eq:gep}.
We also introduce a subspace iteration technique and analyze an error bound.
Numerical experiments are reported in Section~\ref{sec:experiment}.
The paper concludes with Section~\ref{sec:conclusion}.
\par
We use the following notations for quasi-matrices. 
Let $V = [v_1, v_2, \dots, v_L]$, $W = \break [w_1, w_2, \dots ,w_L]$: $\mathbb{C}^L \rightarrow \mathcal{H}$ be quasi-matrices, whose columns are functions defined on an interval $[a, b]$, $a, b \in \mathbb{R}$.
Then, we define the range of $V$ by $\mathscr{R}(V) = \{ y \in \mathcal{H} \mid y = V {\bm x}, {\bm x} \in \mathbb{C}^L \}$.
	In addition, the $L \times L$ matrix~$X$, whose $(i, j)$ element is $X_{ij} = (v_i, w_j)_{\mathcal{H}}$, is expressed as $X = V^\mathsf{H} W$.
Here, $V^\mathsf{H}$ is the conjugate transpose of a quasi-matrix~$V$ such that its rows are the complex conjugates of functions $v_1, v_2, \dots , v_L$.
\section{Complex moment-based matrix eigensolvers}
\label{sec:matrix_solvers}
The complex moment-based eigensolvers proposed by Sakurai and Sugiura \cite{sakurai2003projection} are intended for solving matrix generalized eigenvalue problems:
\begin{align*}
	A {\bm x}_i = \lambda_i B {\bm x}_i, \quad
	A, B \in \mathbb{C}^{n \times n}, \quad
	{\bm x}_i \in \mathbb{C}^n \setminus \{ {\bm 0} \}, \quad
	\lambda_i \in \Omega \subset \mathbb{C},
	\label{eq:gep}
\end{align*}
where $zB-A$ is nonsingular in a boundary $\Gamma$ of the target region $\Omega$.
These eigensolvers use Cauchy's integral formula to form complex moments.
Complex moments can extract the target eigenpairs from random vectors or matrices.
\par
We denote the $k$th order complex moment by
\begin{equation*}
	M_k = \frac{1}{2 \pi \mathrm{i}} \oint_\Gamma z^k (zB-A)^{-1} B \mathrm{d} z,
\end{equation*}
where $\pi$ is the circular constant, $\mathrm{i}$ is the imaginary unit, and $\Gamma$ is a positively oriented closed Jordan curve of which $\Omega$ is the interior.
Then, the complex moment~$M_k$ applied to a matrix~$V \in \mathbb{C}^{n \times L}$ serves as a filter that stops undesired eigencomponents in the column vectors of $V$ from passing through.
To achieve this role of a complex moment, we introduce a transformation matrix~$S \in \mathbb{C}^{n \times LM}$
\begin{equation}
	S = [S_0, S_1, \dots, S_{M-1}] , \quad S_k = M_k V,
	\label{eq:set_S}
\end{equation}
where $V \in \mathbb{C}^{n \times L}$ and $M-1$ is the largest order of complex moments.
Note that $L$ and $M$ are regarded as parameters.
The special case~$M = 1$ in $S$ reduces to FEAST~\cite[equation~(3)]{polizzi2009density}.
Thus, the range $\mathscr{R}(S)$ of $S$ forms the eigenspace of interest (see e.g., \cite[Theorem~1]{imakura2016error}).
\par
Practical algorithms of the complex moment-based eigensolvers approximate the contour integral of the transformation matrix~$\widehat{S}_k \simeq S_k$ of \eqref{eq:set_S} using a quadrature rule
\begin{equation*}
	\widehat{S}_k = \sum_{j=1}^N \omega_j z_j^k (z_j B - A)^{-1} BV,
\end{equation*}
%
where $z_j, \omega_j \in \mathbb{C}$ $(j = 1, 2, \dots, N)$ are quadrature points and the corresponding weights, respectively.
\par
The most time-consuming part of complex moment-based eigensolvers involves solving linear systems at each quadrature point.
These linear systems can be independently solved so that the eigensolvers have good scalability, as demonstrated in \cite{kestyn2016pfeast,iwase2017efficient}.
For this reason, complex moment-based eigensolvers have attracted considerable attention, particularly in physics~\cite{polizzi2009density,kestyn2016pfeast}, materials science~\cite{iwase2017efficient,huang2021efficient,kurz2020solving}, power systems \cite{georgios2020comparison}, data science~\cite{imakura2019complex} and so on.
Currently, there are several methods, including direct extensions of Sakurai and Sugiura's approach \cite{sakurai2007cirr,ikegami2010filter,ikegami2010contour,imakura2014block,imakura2016relationships,imakura2017block,imakura2017structure}, the FEAST eigensolver \cite{polizzi2009density} developed by Polizzi, and its improvements \cite{tang2014feast,guttel2015zolotarev,kestyn2016pfeast}.
We refer to the study by \cite{imakura2016relationships} and the references therein, for relationship among typical complex moment-based methods: the methods using the Rayleigh--Ritz procedure \cite{sakurai2007cirr,ikegami2010contour}, the methods using Hankel matrices \cite{sakurai2003projection,ikegami2010filter}, the method using the communication avoiding Arnoldi procedure \cite{imakura2017block}, FEAST eigensolver \cite{polizzi2009density}, and so on.
\section{Complex moment-based methods}
\label{sec:propose}
In the ``solve-then-discretize'' paradigm, an operator analogue of the FEAST method was proposed \cite{horning2020feast} for solving \eqref{eq:gep} without requiring discretization of the operators $\mathcal{A}$ and $\mathcal{B}$.
The operator analogue of FEAST (contFEAST) is a simple extension of the matrix FEAST eigensolver and is based on an accelerated subspace iteration only with complex moments of order zero; see Algorithm~\ref{alg:feast}.
In each iteration, contFEAST requires solving a large number of ODEs to construct a subspace.
In this study, to reduce computational costs, we propose operator analogues of Sakurai--Sugiura-type complex moment-based eigensolvers: contSS-RR, contSS-Hankel, and contSS-CAA using complex moments of higher order.
\begin{algorithm}[t]
	\small
	\caption{contFEAST method}
	\label{alg:feast}
	\begin{algorithmic}[1]
		\Require $L, N \in \mathbb{N}, \delta \in \mathbb{R}, V : \mathbb{C}^{L} \rightarrow \mathcal{H}, (z_j, \omega_j)$ for $j = 1, 2, \dots, N$
		\Ensure Approximate eigenpairs $(\widehat\lambda_i, \widehat{u}_i)$ for $i = 1, 2, \ldots, L$
		\For{$\ell = 1, 2, \dots$}
		\State Compute $\widehat{S}_0 = \sum_{j=1}^{N} \omega_j (z_j \mathcal{B} -\mathcal{A} )^{-1} \mathcal{B} V$
		\State Compute QR factorization of $\widehat{S}_0$: $\widehat{S}_0 = \widehat{Q} \widehat{R}$
		\State Compute eigenpairs $(\theta_i, {\bm t}_i)$ of $\widehat{Q}^\mathsf{H} \mathcal{A} \widehat{Q} {\bm t}_i = \theta_i \widehat{Q}^\mathsf{H} \mathcal{B} \widehat{Q} {\bm t}_i$
		\Statex \qquad and compute $(\widehat\lambda_i, \widehat{u}_i) = (\theta_i, \widehat{Q} {\bm t}_i)$ for $i = 1, 2, \ldots, L$
		\State Set $V = [\widehat{u}_1, \widehat{u}_2, \dots, \widehat{u}_L]$
		\EndFor
	\end{algorithmic}
\end{algorithm}
\subsection{Complex moment subspace and its properties}
For the differential eigenvalue problem \eqref{eq:gep}, spectral projectors $\mathcal{P}_i$ and $\mathcal{P}_\Omega$ associated with a finite eigenvalue $\lambda_i$ and the target eigenvalues $\lambda_i \in \Omega$ are defined as
\begin{equation}
	\mathcal{P}_i = \frac{1}{2 \pi \textrm{i}} \oint_{\Gamma_i} (z \mathcal{B} - \mathcal{A})^{-1} \mathcal{B} \textrm{d}z, \quad
	\mathcal{P}_\Omega = \sum_{\lambda_i \in \Omega} \mathcal{P}_i =  \frac{1}{2 \pi \textrm{i}} \oint_{\Gamma} (z \mathcal{B} - \mathcal{A})^{-1} \mathcal{B} \textrm{d}z,
	\label{eq:defP}
\end{equation}
respectively, where $\Gamma_i$ is a positively oriented closed Jordan curve in which $\lambda_i$ lies and contour paths $\Gamma_i$ and $\Gamma_j$ do not intersect each other for $i \neq j$; see \cite[pp.178--179]{kato1995perturbation} for the case of $\mathcal{B}=\mathcal{I}$.
Here, spectral projectors $\mathcal{P}_i$ satisfy
\begin{equation*}
	\mathcal{P}_i \mathcal{P}_j = \delta_{ij} \mathcal{P}_i
\end{equation*}
where $\delta_{ij}$ is the Kronecker delta.
\par
Analogously to the complex moment-based eigensolvers for matrix eigenvalue problems, we define the $k$th order complex moment as
\begin{equation}
	\mathcal{M}_k = \frac{1}{2 \pi \textrm{i}} \oint_\Gamma z^k (z \mathcal{B} - \mathcal{A})^{-1} \mathcal{B} \textrm{d}z, \quad
	k = 1, 2, \ldots, M-1
	\label{eq:sk}
\end{equation}
and the transformation quasi-matrix as
\begin{equation}
	S = [S_0, S_1, \dots, S_{M-1}], \quad
	S_k = \mathcal{M}_k V
	\label{eq:defS}
\end{equation}
for $k = 0, 1, \dots, M-1$, where $M-1$ is the highest order of complex moments and $V: \mathbb{C}^{L} \rightarrow \mathcal{H}$ is a quasi-matrix.
Here, $L$ is a parameter.
Note the identity $\mathcal{P}_i = \mathcal{M}_0$ for $\Gamma = \Gamma_i$.
Then, the range $\mathscr{R}(S)$ has the following properties.
\begin{theorem}
	\label{thm:s}
	The columns of $S$ defined in \eqref{eq:defS} form a basis of the target eigenspace $\mathcal{X}_\Omega$ corresponding to $\Omega$, i.e.,
	\begin{equation}
		\mathscr{R}(S) = \mathcal{X}_\Omega = \mathscr{R} \left( \sum_{\lambda_i \in \Omega} \mathcal{P}_i \right),
		\label{eq:s=u_omega}
	\end{equation}
	if $\textrm{rank}(S) = m$, where $m$ is the number of eigenvalues, counting multiplicity, in $\Omega$ of \eqref{eq:gep}.
\end{theorem}
\begin{proof}
	Cauchy's integral formula shows
	\begin{equation*}
		\mathcal{M}_k = \sum_{\lambda_i \in \Omega} \lambda_i^k \mathcal{P}_i.
	\end{equation*}
	Therefore, from the definitions of $S$ and $S_k$, the quasi-matrix $S$ can be written as
	\begin{align*}
		S 
		&= \left[ \sum_{\lambda_i \in \Omega} \mathcal{P}_i V, \sum_{\lambda_i \in \Omega} \lambda_i \mathcal{P}_i V, \dots, \sum_{\lambda_i \in \Omega} \lambda_i^{M-1} \mathcal{P}_i V \right] \\
		&= \mathcal{P}_\Omega \sum_{\lambda_i \in \Omega} \left[ \mathcal{P}_i V, \lambda_i \mathcal{P}_i V, \dots, \lambda_i^{M-1} \mathcal{P}_i V \right],
	\end{align*}
	which provides \eqref{eq:s=u_omega} if $\textrm{rank}(S) = m$.
\end{proof}
	\begin{remark}
		Theorem~\ref{thm:s} shows that the target eigenpairs of \eqref{eq:gep} can be obtained by using a projection method onto $\mathscr{R}(S)$.
	\end{remark}
\begin{theorem}
	\label{thm:krylov}
	Let $S_0$ and $S$ be defined as in \eqref{eq:defS}.
	Then, the range $\mathscr{R}(S)$ and the block Krylov subspace
	\begin{equation*}
		\mathscr{K}_M(\mathcal{C},S_0) = \mathscr{R}([S_0, \mathcal{C}S_0, \dots, \mathcal{C}^{M-1}S_0])
	\end{equation*}
	are the same, i.e.,
	\begin{equation}
		\mathscr{R}(S) = \mathscr{K}_M(\mathcal{C},S_0),
		\label{eq:krylov}
	\end{equation}
	where
	\begin{equation*}
		\mathcal{C} = \sum_{\lvert \lambda_i \rvert < \infty} \lambda_i \mathcal{P}_i.
	\end{equation*}
	Here, $\mathcal{P}_i$ is defined in \eqref{eq:defP}.
	Moreover, the eigenvalue problem of linear operator $\mathcal{C}$ 
	\begin{equation}
		\mathcal{C} u_i = \lambda_i u_i
		\label{eq:sep}
	\end{equation}
	has the same finite eigenpairs as $\mathcal{A}u_i = \lambda_i\mathcal{B}u_i$.
\end{theorem}
\begin{proof}
	The quasi-matrix $S_k$ is written as
	\begin{align*}
		S_k 
		&= \sum_{\lambda_i \in \Omega} \lambda_i^k \mathcal{P}_i V 
		= \left( \sum_{\lvert \lambda_i \rvert < \infty} \lambda_i \mathcal{P}_i \right) \sum_{\lambda_i \in \Omega} \lambda_i^{k-1} \mathcal{P}_i V  \\
		&= \left( \sum_{\lvert \lambda_i \rvert < \infty} \lambda_i \mathcal{P}_i \right)^k \sum_{\lambda_i \in \Omega} \mathcal{P}_i V 
		= \mathcal{C}^k S_0,
	\end{align*}
	which provides \eqref{eq:krylov}.
	Hence, the eigenspace of $\mathcal{C}$ and that of \eqref{eq:gep} are the same.
\end{proof}
	\begin{remark}
		Theorem~\ref{thm:krylov} shows that several techniques for block Krylov subspace can be used to form $\mathscr{R}(S)$ and the target eigenpairs of \eqref{eq:gep} can be obtained by solving \eqref{eq:sep}.
	\end{remark}
	\par
	Theorems~\ref{thm:s} and \ref{thm:krylov} are used to derive methods in Section~\ref{sec:derivation} and provide an error bound in Section~\ref{sec:errorbound}.
\subsection{Derivations of methods}
\label{sec:derivation}
Using Theorems \ref{thm:s} and \ref{thm:krylov}, based on the complex moment-based eigensolvers, SS-RR, SS-Hankel, and SS-CAA, we develop complex moment-based differential eigensolvers for solving \eqref{eq:gep} without the discretization of operators $\mathcal{A}$ and $\mathcal{B}$.
The proposed methods are projection methods based on $\mathscr{R}(S)$, which is a larger subspace than $\mathscr{R}(S_0)$ used in contFEAST (Algorithm~\ref{alg:feast}).
\par
In practice, we numerically deal with operators, functions, and the contour integrals.
The contour integral in \eqref{eq:sk} is approximated using the quadrature rule:
\begin{equation}
	\widehat{S} = [\widehat{S}_0, \widehat{S}_1, \dots, \widehat{S}_{M-1}], \quad
	\widehat{S}_k = \sum_{j=1}^N \omega_j z_j^k (z_j \mathcal{B} - \mathcal{A})^{-1} \mathcal{B} V,
	\label{eq:quadrature}
\end{equation}
where $z_j, \omega_j \in \mathbb{C}$ $(j = 1, 2, \dots, N)$ are quadrature points and the corresponding weights, respectively.
As well as contFEAST, we avoid discretizing the operators, but we construct polynomial approximations on the basis of the invariant subspace by approximately solving ODEs of the form
\begin{equation}
	(z_j \mathcal{B} - \mathcal{A}) y_{i,j} =  \mathcal{B} v_i, \quad
	i = 1, 2, \dots, L, \quad j = 1, 2, \dots, N
	\label{eq:ode}
\end{equation}
with boundary conditions.
Note that the number of ODEs to be solved does not depend on the degree of complex moments $M$.
\par
For real operators $\mathcal{A}$ and $\mathcal{B}$, if quadrature points and the corresponding weights are set symmetric about the real axis, $(z_j, \omega_j) = (\overline{z}_{j+N/2}, \overline{\omega}_{j+N/2}), j = 1,2,\dots, N/2$, we can halve the number of ODEs to be solved as follows:
\begin{equation}
	\widehat{S}_k = 2 \sum_{j=1}^{N/2} \textrm{Re}\left( \omega_j z_j^k (z_j \mathcal{B} - \mathcal{A})^{-1} \mathcal{B} V \right).
	\label{eq:sym}
\end{equation}
\par
As another efficient computation technique for real self-adjoint problems, we can avoid complex ODEs using the real rational filtering technique \cite{austin2015computing} for matrix eigenvalue problems.
Using the real rational filtering technique, quasi-matrix $S$ is approximated by \eqref{eq:quadrature} with the $N$ Chebyshev points of the first kind and the corresponding barycentric weights,
\begin{equation}
	z_j = \gamma + \rho \cos \left( \frac{(2j-1)\pi}{2N} \right), \quad
	\omega_j = (-1)^j \sin \left( \frac{(2j-1)\pi}{2N} \right),
	\label{eq:Chebyshev}
\end{equation}
for $j = 1,2, \dots, N$, where $\gamma$ and $\rho$ are the center and radius of the target interval.
Note that $z_j, \omega_j \in \mathbb{R}$ for $j = 1,2, \dots, N$.
\subsubsection{ContSS-RR method}
An operator analogue of the complex moment-based method using the Rayleigh--Ritz procedure for matrix eigenvalue problems \cite{sakurai2007cirr,ikegami2010contour} is presented.
Theorem~\ref{thm:s} shows that the target eigenpairs of \eqref{eq:gep} can be obtained by a Rayleigh--Ritz procedure based on $\mathscr{R}(S)$, i.e.,
\begin{equation*}
	S^\mathsf{H} \mathcal{A} S {\bm t}_i = \theta_i S^\mathsf{H} \mathcal{B} S {\bm t}_i,
\end{equation*}
where $(\lambda_i,u_i) = (\theta_i,S {\bm t}_i)$.
We approximate this Rayleigh--Ritz procedure using an $\mathcal{H}$-orthonormal basis of the approximated subspace $\mathscr{R}(\widehat{S})$.
Here, to reduce computational costs and improve numerical stability, we use a low-rank approximation of quasi-matrix $\widehat{S}$ based on its truncated singular value decomposition (TSVD) \cite{trefethen2010householder}, i.e.,
\begin{equation*}
	\widehat{S} = [U_\textrm{S1}, U_\textrm{S2}] \left[
	\begin{array}{cc}
		\Sigma_\textrm{S1} & O \\
		O & \Sigma_\textrm{S2}
	\end{array}
	\right] \left[
	\begin{array}{c}
		W_\textrm{S1}^\mathsf{H} \\
		W_\textrm{S2}^\mathsf{H}
	\end{array}
	\right] \approx
	U_\textrm{S1} \Sigma_\textrm{S1} W_\textrm{S1}^\mathsf{H},
\end{equation*}
where $\Sigma_\textrm{S1} \in \mathbb{R}^{d \times d}$ is a diagonal matrix whose diagonal entries are the $d$ largest singular values such that $\sigma_d / \sigma_1 \geq \delta \geq  \sigma_{d+1} / \sigma_1$ $(\sigma_i \geq \sigma_{i+1}, i = 1, 2, \dots, d)$ and $U_\textrm{S1}: \mathbb{C}^d \rightarrow \mathcal{H}$ and $W_\textrm{S1} \in \mathbb{C}^{LM \times d}$ are column-orthonormal (quasi-)matrices corresponding to the left and right singular vectors, respectively.
\par
Thus, the target problem \eqref{eq:gep} is reduced to a $d$-dimensional matrix generalized eigenvalue problem
\begin{equation*}
	U_\textrm{S1}^\mathsf{H} \mathcal{A} U_\textrm{S1} {\bm t}_i = \theta_i U_\textrm{S1}^\mathsf{H} \mathcal{B} U_\textrm{S1} {\bm t}_i,
\end{equation*}
where the approximated eigenpairs are computed as $(\widehat{\lambda}_i, \widehat{u}_i) = (\theta_i, U_\textrm{S1} {\bm t}_i)$.
The procedure of the contSS-RR method is summarized in Algorithm~\ref{alg:ss-rr}.
\begin{algorithm}[t]
	\small
	\caption{contSS-RR method}
	\label{alg:ss-rr}
	\begin{algorithmic}[1]
		\Require $L, M, N \in \mathbb{N}, \delta \in \mathbb{R}, V : \mathbb{C}^{L} \rightarrow \mathcal{H}, (z_j, \omega_j)$ for $j = 1, 2, \dots, N$
		\Ensure Approximate eigenpairs $(\widehat\lambda_i, \widehat{u}_i)$ for $i = 1, 2, \ldots, d$
		\State Compute $\widehat{S}_k = \sum_{j=1}^{N} \omega_j z_j^k (z_j \mathcal{B} -\mathcal{A} )^{-1} \mathcal{B} V$
		\State Set $\widehat{S} = [\widehat{S}_0, \widehat{S}_1, \dots, \widehat{S}_{M-1}]$
		\State Compute low-rank approximation of $\widehat{S}$ using the threshold $\delta$:
		\Statex $\widehat{S}= [U_\textrm{S1}, U_\textrm{S2}] [\Sigma_\textrm{S1}, O; O, \Sigma_\textrm{S2}] [W_\textrm{S1}, W_\textrm{S2}]^\mathsf{H} \approx U_\textrm{S1} \Sigma_\textrm{S1} W_\textrm{S1}^\mathsf{H}$
		\State Compute eigenpairs $(\theta_i, {\bm t}_i)$ of $U_\textrm{S1}^\mathsf{H} \mathcal{A} U_\textrm{S1} {\bm t}_i = \theta_i U_\textrm{S1}^\mathsf{H} \mathcal{B} U_\textrm{S1} {\bm t}_i$,
		\Statex and compute $(\widehat\lambda_i, \widehat{u}_i) = (\theta_i, U_\textrm{S1} {\bm t}_i)$ for $i = 1, 2, \ldots, d$
	\end{algorithmic}
\end{algorithm}
\subsubsection{ContSS-Hankel method}
An operator analogue of the complex moment-based method using Hankel matrices for matrix eigenvalue problems \cite{sakurai2003projection,ikegami2010filter} is presented.
Let $\mu_k \in \mathbb{C}^{L \times L}$ be a reduced complex moment of order $k$ defined as
\begin{equation*}
	\mu_k 
	= \frac{1}{2 \pi \textrm{i}} \oint \widetilde{V}^\mathsf{H} z^k (z \mathcal{B} - \mathcal{A})^{-1} \mathcal{B} V \textrm{d} z
	= \widetilde{V}^\mathsf{H} S_k
\end{equation*}
with $\widetilde{V}: \mathbb{C}^L \rightarrow \mathcal{H}$.
We also define block Hankel matrices
\begin{equation*}
	H_M^< = \left[
	\begin{array}{cccc}
		\mu_1 & \mu_2 & \cdots & \mu_{M} \\
		\mu_2 & \mu_3 & \cdots & \mu_{M+1} \\
		\vdots & \vdots & \ddots & \vdots \\
		\mu_{M} & \mu_{M+1} & \cdots & \mu_{2M-1}
	\end{array}
	\right], \quad
	H_M = \left[
	\begin{array}{cccc}
		\mu_0 & \mu_1 & \cdots & \mu_{M-1} \\
		\mu_1 & \mu_2 & \cdots & \mu_{M} \\
		\vdots & \vdots & \ddots & \vdots \\
		\mu_{M-1} & \mu_M & \cdots & \mu_{2M-2}
	\end{array}
	\right],
\end{equation*}
which have the following property.
\begin{theorem}
	\label{thm:hankel}
	If $\textrm{rank}(H_M) = \textrm{rank}(H_M^<) = m$, where $m$ is the number of eigenvalues in $\Omega$ of \eqref{eq:gep}, then the nonsingular part of a matrix pencil $zH_M - H_M^<$ and $z - \mathcal{P}_\Omega$ have the same spectrum, where $\mathcal{P}_\Omega$ is defined in \eqref{eq:defP}.
\end{theorem}
\begin{proof}
	The complex moment $\mu_k$ can be written as
	\begin{equation*}
		\mu_k  
		= \widetilde{V}^\mathsf{H} \mathcal{P}_\Omega S_{k}
		= \widetilde{V}^\mathsf{H} \mathcal{P}_\Omega \mathcal{C}^1 \mathcal{P}_\Omega S_{k-1}
		= \dots 
		= \widetilde{V}^\mathsf{H} \mathcal{P}_\Omega \mathcal{C}^k \mathcal{P}_\Omega S_0,
	\end{equation*}
	where $\mathcal{C}$ is defined in Theorem~\ref{thm:krylov}.
	Letting
	\begin{equation*}
		\widetilde{S} = \left[ \widetilde{V}, \mathcal{C}^\mathsf{H} \widetilde{V}, \dots, (\mathcal{C}^\mathsf{H})^{M-1} \widetilde{V} \right],
	\end{equation*}
	the block Hankel matrices are written as
	\begin{equation*}
		H_M^< = \widetilde{S}^\mathsf{H} \mathcal{P}_\Omega \mathcal{C} \mathcal{P}_\Omega S, \quad
		H_M = \widetilde{S}^\mathsf{H} \mathcal{P}_\Omega S,
	\end{equation*}
	which proves Theorem~\ref{thm:hankel}.
\end{proof}
Theorem~\ref{thm:hankel} shows that the target eigenpairs of \eqref{eq:gep} can be computed via a matrix eigenvalue problem:
\begin{equation*}
	H_M^< {\bm y}_i = \theta_i H_M {\bm y}_i.
\end{equation*}
Note that from the equivalence
\begin{equation*}
	H_M^< {\bm y}_i = \theta_i H_M {\bm y}_i
	\quad \Leftrightarrow \quad
	( \mathcal{P}_\Omega \widetilde{S} )^\mathsf{H} \mathcal{C} ( \mathcal{P}_\Omega S ) {\bm y}_i = 
	\theta_i ( \mathcal{P}_\Omega \widetilde{S} )^\mathsf{H} ( \mathcal{P}_\Omega S ) {\bm y}_i,
\end{equation*}
this approach can be regarded as a Petrov--Galerkin-type projection for \eqref{eq:sep}, which has the same finite eigenpairs as $\mathcal{A}u_i = \lambda_i\mathcal{B}u_i$.
In practice, block Hankel matrices $H^<_M$ and $H_M$ are approximated by block Hankel matrices $\widehat{H}_M^<$ and $\widehat{H}_M$ whose block $(i,j)$ entries are $\widehat{\mu}_{i+j+1}$ and $\widehat{\mu}_{i+j}$, respectively, where
\begin{equation*}
	\widehat{\mu}_k
	= \sum_{j=1}^N \widetilde{V}^\mathsf{H} z_j^k \omega_j (z_j \mathcal{B} - \mathcal{A})^{-1} \mathcal{B} V
	= \widetilde{V}^\mathsf{H} \widehat{S}_k,
\end{equation*}
for $k = 1, 2, \dots, 2M-1$.
To reduce computational costs and improve numerical stability, we use a low-rank approximation of $\widehat{H}_M$ based on TSVD, i.e.,
\begin{equation*}
	\widehat{H}_M = [U_\textrm{H1}, U_\textrm{H2}] \left[
	\begin{array}{cc}
		\Sigma_\textrm{H1} & O \\
		O & \Sigma_\textrm{H2}
	\end{array}
	\right] \left[
	\begin{array}{c}
		W_\textrm{H1}^\mathsf{H} \\
		W_\textrm{H2}^\mathsf{H}
	\end{array}
	\right] \approx
	U_\textrm{H1} \Sigma_\textrm{H1} W_\textrm{H1}^\mathsf{H},
\end{equation*}
where $\Sigma_\textrm{H1} \in \mathbb{R}^{d \times d}$ is a diagonal matrix whose diagonal entries are the $d$ largest singular values such that $\sigma_d / \sigma_1 \geq \delta \geq  \sigma_{d+1} / \sigma_1$ $(\sigma_i \geq \sigma_{i+1}, i = 1, 2, \dots, d)$ and $U_\textrm{H1}, W_\textrm{H1} \in \mathbb{C}^{LM \times d}$ are column-orthonormal matrices corresponding to the left and right singular vectors, respectively.
\par
Then, the target problem \eqref{eq:gep} is reduced to a $d$-dimensional standard matrix eigenvalue problem of the form
\begin{equation*}
	U_\textrm{H1}^\mathsf{H} \widehat{H}_M^< W_\textrm{H1} \Sigma_\textrm{H1}^{-1} {\bm t}_i = \theta_i {\bm t}_i,
\end{equation*}
where the approximated eigenpairs can be computed as $(\widehat{\lambda}_i, \widehat{u}_i) = (\theta_i, \widehat{S} W_\textrm{H1} \Sigma_\textrm{H1}^{-1} {\bm t}_i)$.
The procedure of the contSS-Hankel method is summarized in Algorithm~\ref{alg:ss-hankel}.
\begin{algorithm}[t]
	\small
	\caption{contSS-Hankel method}
	\label{alg:ss-hankel}
	\begin{algorithmic}[1]
		\Require $L, M, N \in \mathbb{N}, \delta \in \mathbb{R}, V : \mathbb{C}^{L} \rightarrow \mathcal{H}, (z_j, \omega_j)$ for $j = 1, 2, \dots, N$
		\Ensure Approximate eigenpairs $(\widehat\lambda_i, \widehat{u}_i)$ for $i = 1, 2, \ldots, d$
		\State Compute $\widehat{S}_k = \sum_{j=1}^{N} \omega_j z_j^k (z_j \mathcal{B} -\mathcal{A} )^{-1} \mathcal{B} V$
		\State Set $\widehat{S} = [\widehat{S}_0, \widehat{S}_1, \dots, \widehat{S}_{M-1}]$ and $\widehat{\mu}_k = \widetilde{V}^\mathsf{H} \widehat{S}_k$
		\State Set block Hankel matrices $\widehat{H}_M^<$ and $\widehat{H}_M$
		\State Compute low-rank approximation of $\widehat{H}_M$ using the threshold $\delta$:
		\Statex $\widehat{H}_M= [U_\textrm{H1}, U_\textrm{H2}] [\Sigma_\textrm{H1}, O; O, \Sigma_\textrm{H2}] [W_\textrm{H1}, W_\textrm{H2}]^\mathsf{H} \approx U_\textrm{H1} \Sigma_\textrm{H1} W_\textrm{H1}^\mathsf{H}$
		\State Compute eigenpairs $(\theta_i, {\bm t}_i)$ of $U_\textrm{H1}^\mathsf{H} \widehat{H}_M^< W_\textrm{H1} \Sigma_\textrm{H1} {\bm t}_i = \theta_i {\bm t}_i$,
		\Statex and compute $(\widehat\lambda_i, \widehat{u}_i) = (\theta_i, \widehat{S} W_\textrm{H1} \Sigma_\textrm{H1}^{-1} {\bm t}_i)$ for $i = 1, 2, \ldots, d$
	\end{algorithmic}
\end{algorithm}
\subsubsection{ContSS-CAA method}
An operator analogue of the complex moment-based method using the communication avoiding Arnoldi procedure for matrix eigenvalue problems \cite{imakura2017block} is presented.
Theorem~\ref{thm:krylov} shows that the target eigenpairs of \eqref{eq:gep} can be obtained by using a block Arnoldi method with $\mathscr{R}(S) = \mathscr{K}_M(\mathcal{C},S_0)$ for \eqref{eq:sep}, which has the same finite eigenpairs as $\mathcal{A}u_i = \lambda_i\mathcal{B}u_i$.
In this algorithm, a quasi-matrix $Q : \mathbb{C}^{LM} \rightarrow \mathcal{H}$ whose columns form an orthonormal basis of $\mathscr{R}(S) = \mathscr{K}_M(\mathcal{C},S_0)$ and a block Hessenberg matrix $T_M = Q^\mathsf{H} \mathcal{C} Q$ are constructed, and the target eigenpairs are computed by solving the standard matrix eigenvalue problem
\begin{equation*}
	T_M {\bm t}_i = \theta_i {\bm t}_i.
\end{equation*}
Therefore, we have $(\lambda_i,u_i) = (\theta_i, Q {\bm t}_i)$.
\par
Further, we consider using a block version of the communication-avoiding Arnoldi procedure~\cite{hoemmen2010communication}.
Let $S_+= [S_0, S_1, \dots, S_M]: \mathbb{C}^{L(M+1)} \rightarrow \mathcal{H}$ be a quasi-matrix.
From Theorem~\ref{thm:krylov}, we have
\begin{equation*}
	\mathcal{C} S = S_+ D_1, \quad
	D_1 = \left[
	\begin{array}{c}
		O_{L,LM} \\
		I_{LM}
	\end{array}
	\right].
\end{equation*}
Here, based on the concept of a block version of the communication-avoiding Arnoldi procedure, using the QR factorizations
\begin{equation*}
	S_+ = Q_+ R_+, \quad
	S=Q R,
\end{equation*}
where $Q = Q_+(:,1\!:\!LM),  R = R_+(1\!:\!LM,1\!:\!LM)$, the block Hessenberg matrix $T_M$ is obtained by
\begin{align*}
	T_M 
	&= Q^\mathsf{H} S_+ D_1 R^{-1} \\
	&= Q^\mathsf{H} Q_+ R_+ D_1 R^{-1} \\
	&= [I_{LM}, O_{LM,L}] R_+ D_1 R^{-1} \\
	&= R_+(1\!:\!LM,L+1\!:LM+L) R^{-1}.
\end{align*}
\par
In practice, we approximate the block Hessenberg matrix $T_M$ by
\begin{equation*}
	\widehat{T}_M = \widehat{R}_+(1\!:\!LM,L+1\!:LM+L) \widehat{R}^{-1},
\end{equation*}
where
\begin{equation*}
	\widehat{S}_+ = [S_0, S_1, \dots, S_M]= \widehat{Q}_+ \widehat{R}_+, \quad
	\widehat{S} = \widehat{Q} \widehat{R},
\end{equation*}
are the QR factorizations of $\widehat{S}_+$ and $\widehat{S}$, respectively, $\widehat{Q} = \widehat{Q}_+(:,1\!:\!LM)$, and $\widehat{R}=\widehat{R}_+(1\!:\!LM,1\!:\!LM)$, and use a low-rank approximation of $\widehat{S}$ based on TSVD, i.e.,
\begin{equation*}
	\widehat{S} = \widehat{Q} \widehat{R}
	= \widehat{Q} [U_\textrm{R1}, U_\textrm{R2}] \left[
	\begin{array}{ll}
		\Sigma_\textrm{R1} & O \\
		O & \Sigma_\textrm{R2}
	\end{array}
	\right] \left[
	\begin{array}{ll}
		W_\textrm{R1}^\mathsf{H} \\
		W_\textrm{R2}^\mathsf{H}
	\end{array}
	\right] \approx \widehat{Q} U_\textrm{R1} \Sigma_\textrm{R1} W_\textrm{R1}^\mathsf{H},
\end{equation*}
where $\Sigma_\textrm{R1} \in \mathbb{R}^{d \times d}$ is a diagonal matrix whose diagonal entries are the $d$ largest singular values $\sigma_d / \sigma_1 \geq \delta \geq  \sigma_{d+1} / \sigma_1$ $(\sigma_i \geq \sigma_{i+1}, i = 1,2, \dots, d)$ and $U_\textrm{R1}, W_\textrm{R1} \in \mathbb{C}^{LM \times d}$ are column-orthonormal matrices corresponding to the left and right singular vectors of $\widehat{R}$, respectively.
\par
Then, the target problem \eqref{eq:gep} is reduced to a $d$-dimensional matrix standard eigenvalue problem of the form
\begin{equation*}
	U_\textrm{R1}^\mathsf{H} \widehat{T}_M U_\textrm{R1} {\bm t}_i = \theta_i {\bm t}_i.
\end{equation*}
The approximate eigenpairs are obtained as $(\widehat{\lambda}_i,\widehat{u}_i) = (\theta_i, \widehat{Q} U_\textrm{R1} {\bm t}_i)$.
The coefficient\break matrix~$U_\textrm{R1}^\mathsf{H} \widehat{T}_M U_\textrm{R1}$ is efficiently computed by
\begin{equation*}
	U_\textrm{R1}^\mathsf{H} \widehat{T}_M U_\textrm{R1} = U_\textrm{R1}^\mathsf{H} \widehat{R}_{+}(1\!:\!LM,L\!+\!1\!:\!LM\!+\!L) W_\textrm{R1} \Sigma_\textrm{R1}^{-1}.
\end{equation*}
The procedure of the contSS-CAA method is summarized in Algorithm~\ref{alg:ss-caa}.
\begin{algorithm}[t]
	\small
	\caption{contSS-CAA method}
	\label{alg:ss-caa}
	\begin{algorithmic}[1]
		\Require $L, M, N \in \mathbb{N}, \delta \in \mathbb{R}, V : \mathbb{C}^{L} \rightarrow \mathcal{H}, (z_j, \omega_j)$ for $j = 1, 2, \dots, N$
		\Ensure Approximate eigenpairs $(\widehat\lambda_i, \widehat{u}_i)$ for $i = 1, 2, \ldots, d$
		\State Compute $\widehat{S}_k = \sum_{j=1}^{N} \omega_j z_j^k (z_j \mathcal{B} -\mathcal{A} )^{-1} \mathcal{B} V$
		\State Set $\widehat{S} = [\widehat{S}_0, \widehat{S}_1, \dots, \widehat{S}_{M-1}]$ and $\widehat{S}_+ = [\widehat{S}_0, \widehat{S}_1, \dots, \widehat{S}_{M}]$
		\State Compute QR factorization of $\widehat{S}_+$: $\widehat{S}_+ = \widehat{Q}_+ \widehat{R}_+$
		\State Set $\widehat{R} = \widehat{R}_+(1\!:\!LM,1\!:\!LM)$, $\widehat{Q} = \widehat{Q}_+(:,1\!:\!LM)$
		\State Compute low-rank approximation of $R$ using the threshold $\delta$:
		\Statex $\widehat{R}= [U_\textrm{R1}, U_\textrm{R2}] [\Sigma_\textrm{R1}, O; O, \Sigma_\textrm{R2}] [W_\textrm{R1}, W_\textrm{R2}]^\mathsf{H} \approx U_\textrm{R1} \Sigma_\textrm{R1} W_\textrm{R1}^\mathsf{H}$
		\State Compute eigenpairs $(\theta_i, {\bm t}_i)$ of $U_\textrm{R1}^\mathsf{H} \widehat{R}_{+}(1\!:\!LM,L\!+\!1\!:\!LM\!+\!L) W_\textrm{R1} \Sigma_\textrm{R1}^{-1} {\bm t}_i = \theta_i {\bm t}_i$,
		\Statex and compute $(\widehat\lambda_i, \widehat{u}_i) = (\theta_i, \widehat{Q} U_\textrm{R1} {\bm t}_i)$ for $i = 1, 2, \ldots, d$
	\end{algorithmic}
\end{algorithm}
\subsection{Subspace iteration and error bound}
\label{sec:errorbound}
We consider improving the accuracy of the eigenpairs via a subspace iteration technique, as in the matrix version of complex moment-based eigensolvers.
We construct $\widehat{S}_0^{(\ell-1)}$ via the following iteration step:
\begin{equation}
	\widehat{S}^{(\nu)}_0 = \sum_{j=1}^{N} \omega_j (z_j \mathcal{B} - \mathcal{A})^{-1} \mathcal{B} \widehat{S}_0^{(\nu-1)}, \quad
	\nu = 1, 2, \ldots, \ell-1
	\label{eq:iter1}
\end{equation}
with the initial quasi-matrix $\widehat{S}_0^{(0)} = V$.
Then, instead of $\widehat{S}$ in each method, we use $\widehat{S}^{(\ell)}$ constructed from $\widehat{S}_0^{(\ell-1)}$ by
\begin{equation}
	\widehat{S}^{(\ell)} = [\widehat{S}_0^{(\ell)}, \widehat{S}_1^{(\ell)}, \ldots, \widehat{S}_{M-1}^{(\ell)} ], \quad
	\widehat{S}^{(\ell)}_k = \sum_{j=1}^{N} \omega_j z_j^k (z_j \mathcal{B} - \mathcal{A})^{-1} \mathcal{B} \widehat{S}_0^{(\ell-1)}.
	\label{eq:iter2}
\end{equation}
The orthonormalization of the columns of $\widehat{S}_0^{(\nu)}$ in each iteration may improve the numerical stability.
\par
Now, we analyze the error bound of the proposed methods with the subspace iteration technique as introduced in \eqref{eq:iter1} and \eqref{eq:iter2}.
We assume that $\mathcal{B}$ is invertible and all the eigenvalues are isolated.
Then, we have
\begin{equation*}
	(z \mathcal{B} - \mathcal{A})^{-1} \mathcal{B} = \sum_{i=1}^\infty \frac{1}{z - \lambda_i} \mathcal{P}_i,
\end{equation*}
where $\mathcal{P}_i$ is a spectral projector associated with $\lambda_i$ s.t. $\mathcal{P}_i \mathcal{P}_j = \delta_{ij} \mathcal{P}_i$ \cite[VII Section 6]{kato1995perturbation}.
We also assume that the numerical quadrature satisfies
\begin{equation*}
	\sum_{j=1}^N \omega_j z_j^k \left\{
	\begin{array}{ll}
		\neq 0, & (k = -1) \\
		= 0, & (k = 0, 1, \ldots, N-2)
	\end{array}
	\right..
\end{equation*}
\par
Under the above assumptions, for $k = 0, 1, \dots, M-1$, the quasi-matrix $\widehat{S}_k^{(\ell)}$ can be written as
\begin{align*}
	\widehat{S}_k^{(\ell)} 
	& = \sum_{j=1}^N z_j^k (z_j \mathcal{B} - \mathcal{A})^{-1} \mathcal{B} \widehat{S}^{(\ell-1)} \\
	&= \sum_{i=1}^\infty \left(\sum_{j=1}^N \frac{\omega_jz_j^k}{z_j - \lambda_i} \right) \left(\sum_{j=1}^N \frac{\omega_j}{z_j - \lambda_i} \right)^{\ell-1} \mathcal{P}_i V \\
	&= \sum_{i=1}^\infty \lambda_i^k \left(\sum_{j=1}^N \frac{\omega_j}{z_j - \lambda_i} \right)^\ell \mathcal{P}_i V \\
	&= \mathcal{C}^k \mathcal{F}^\ell V,
\end{align*}
where
\begin{equation*}
	\mathcal{F} = \sum_{i=1}^\infty f_N(\lambda_i) \mathcal{P}_i, \quad
	f_N(\lambda_i) = \sum_{j=1}^N \frac{\omega_j}{z - \lambda_i}.
\end{equation*}
From the definitions of $\mathcal{C}$ and $\mathcal{F}$, these operators are commutative, $\mathcal{C}\mathcal{F} = \mathcal{F}\mathcal{C}$.
Therefore, we have
\begin{equation}
	\widehat{S}^{(\ell)} = \mathcal{F}^\ell [V, \mathcal{C} V, \dots, \mathcal{C}^{M-1}V].
	\label{eq:iteration}
\end{equation}
\par
Here, $f_N(\lambda_i)$ is called the filter function; it is used for error analyses of complex moment-based matrix eigensolvers \cite{tang2014feast,guttel2015zolotarev,imakura2016error,imakura2016relationships}.
Fig.~\ref{fig:filter} shows the magnitude of the filter function~$\lvert f_N(\lambda_i) \rvert$ for the $N$-point trapezoidal rule with $N = 16, 32$, and $64$ for the unit circle region~$\Omega$.
Here, note that the oscillations at $\lvert f_N(\lambda) \rvert \approx  10^{-16}$ are due to roundoff errors.
The filter function has $\lvert f_N(\lambda) \rvert \approx 1$ inside $\Omega$, $\lvert f_N(\lambda) \rvert \approx 0$ far from $\Omega$, and $0 < \lvert f_N(\lambda) \rvert < 1$ outside but near the region.
Therefore, $\mathcal{F}$ is a bounded linear operator.
\begin{figure}[t]
	\centering
	\begin{minipage}{0.49\hsize}
		\centering
		\includegraphics[width=0.8\textwidth]{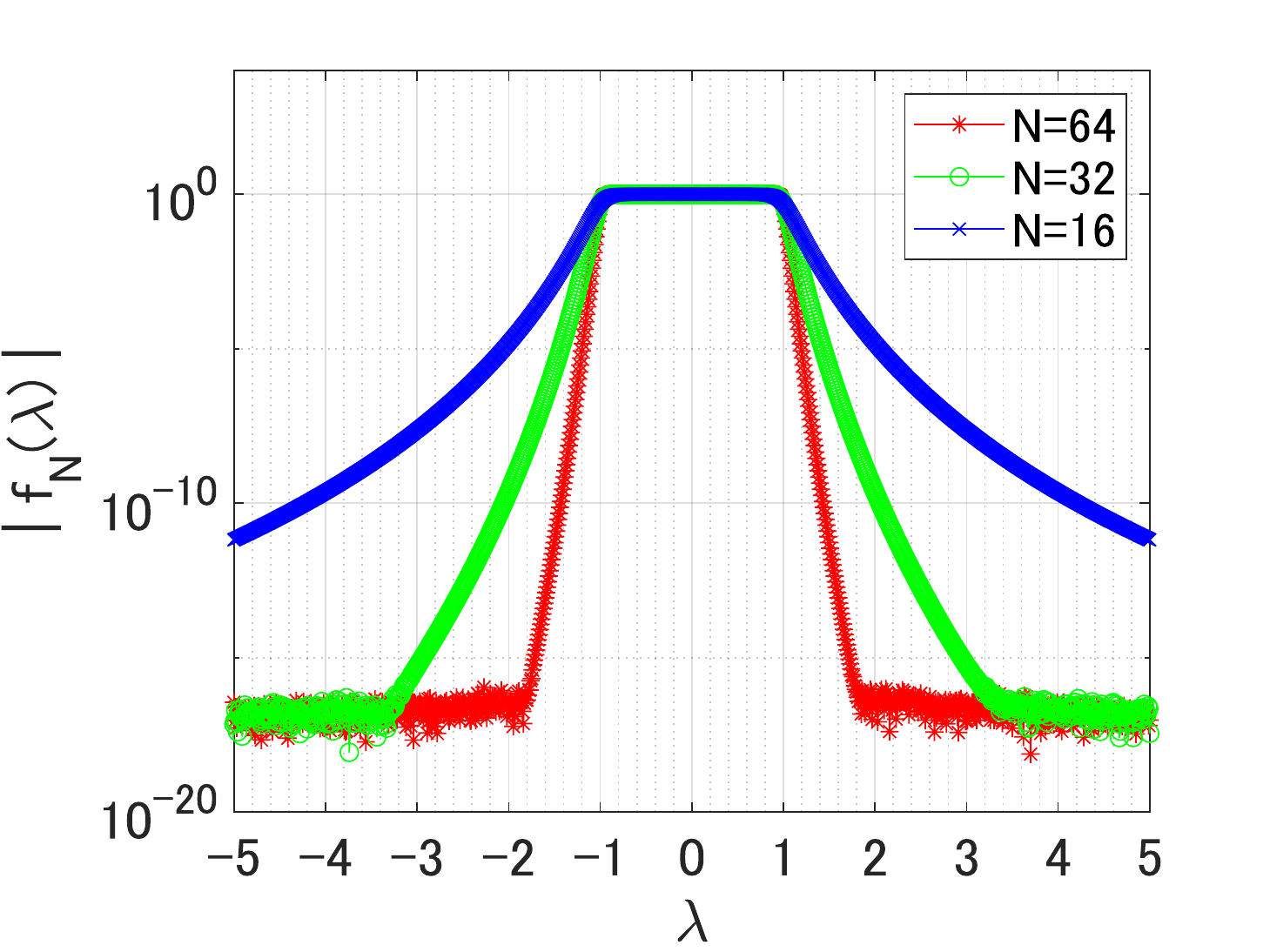}
		\subcaption{On the real axis}
	\end{minipage}
	\begin{minipage}{0.49\hsize}
		\centering
		\includegraphics[width=0.8\textwidth]{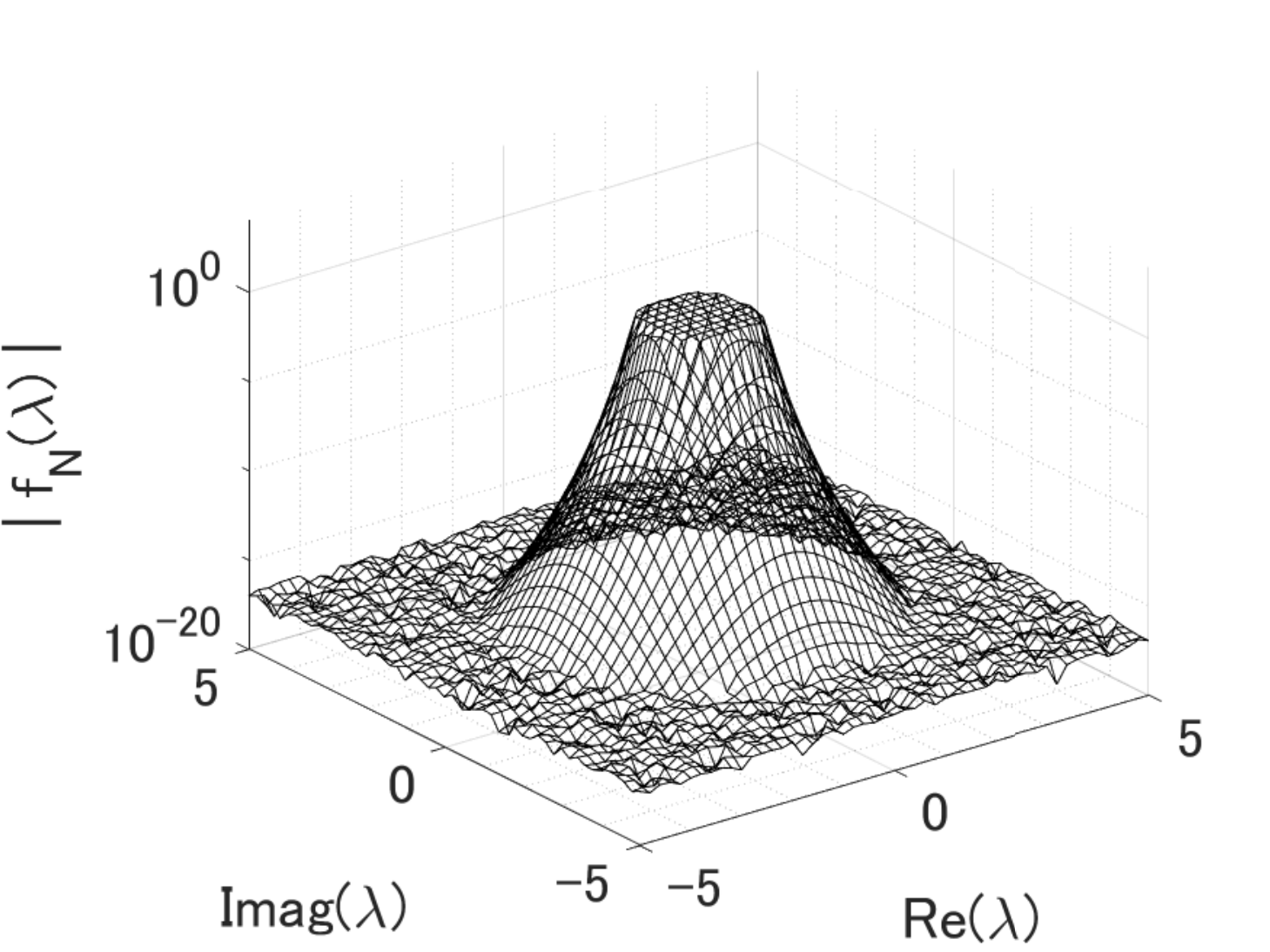}
		\subcaption{On the complex plane}
	\end{minipage}
	\caption{Magnitude of the filter functions for the $N$-point trapezoidal rule for the unit circle region $\Omega$.}
	\label{fig:filter}
\end{figure}
\par
Applying \cite[Theorem~5.1]{horning2020feast} to \eqref{eq:iteration} under the above assumptions, we have the following theorem for an error bound of the proposed methods.
\begin{theorem}
	\label{thm:error}
	Let $(\lambda_i,u_i)$ be exact finite eigenpairs of the differential eigenvalue problem $\mathcal{A} u_i = \lambda_i \mathcal{B} u_i, i = 1, 2, \dots, LM$.
	Assume that the filter function $f_N(\lambda_i)$ is ordered by decreasing magnitude $\lvert f_N(\lambda_i) \rvert \geq \lvert f_N(\lambda_{i+1}) \rvert$.
	We define as an orthogonal projector onto the subspaces~$\mathscr{R}( \widehat{S}^{(\ell)} )$ and the spectral projector with an invariant subspace $\textrm{span}\{ u_1, u_2, \ldots, u_{LM} \}$ by $\mathcal{P}^{(\ell)}$ and $\mathcal{P}_{LM}$, respectively.
	Assume that $\mathcal{P}_{LM} \widehat{S}^{(0)}$ has full rank, where $\widehat{S}_0$ is defined in \eqref{eq:iter2}.
	Then, for each eigenfunction $u_i, i = 1, 2, \dots$, $LM$, there exists a unique function $s_i \in \mathcal{K}_M^\square(\mathcal{C},V)$ such that $\mathcal{P}_{LM} s_i = u_i$.
	Thus, the following inequality is satisfied: 
	\begin{equation*}
		\| (I - \mathcal{P}^{(\ell)} ) u_i \|_\mathcal{H}
		\leq \alpha \beta_i \left \lvert \frac{f_N(\lambda_{LM+1})}{f_N(\lambda_i)} \right \rvert^\ell, \quad
		i = 1, 2, \ldots, LM, \quad
		\ell = 1, 2, \ldots,
		\label{eq:ineq1}
	\end{equation*}
	where $\alpha$ is a constant and $\beta_i = \| u_i - s_i \|_\mathcal{H}$.
\end{theorem}
\par
Theorem~\ref{thm:error} indicates that, using a sufficiently large number of columns $LM$ in the transformation quasi-matrix $\widehat{S}$ such that $\lvert f_N(\lambda_{LM+1}) \rvert^\ell \approx 0$, the proposed methods achieve high accuracy for the target eigenpairs even if $N$ is small and some eigenvalues exist outside but near the region.
\subsection{Summary and advantages over existing methods}
We summarize the proposed methods and present their advantages over existing methods.
\subsubsection{Summary of the proposed methods}
ContSS-RR is a Rayleigh--Ritz-type projection method that explicitly solves \eqref{eq:gep}; on the other hand, contSS-Hankel and contSS-CAA are a Petrov--Galerkin-type projection method and block Arnoldi method that implicitly solve \eqref{eq:sep}, respectively.
If the computational cost for explicit projection, i.e., $U_\textrm{S1}^\mathsf{H} \mathcal{A} U_\textrm{S1}$ and $U_\textrm{S1}^\mathsf{H} \mathcal{B} U_\textrm{S1}$ in contSS-RR, is large, contSS-Hankel and contSS-CAA can be more efficient than contSS-RR.
\par
Orthogonalization of basis functions is required for contSS-RR (step 3 of Algorithm~\ref{alg:ss-rr}) and contSS-CAA(step 5 of Algorithm~\ref{alg:ss-caa}) for accuracy but not performed in contSS-Hankel.
This is the advantage of contSS-Hankel regarding computational costs over other methods.
In addition, this is advantageous for contSS-Hankel when applied to DEPs over a domain for which it is difficult to construct accurate orthonormal bases in such as triangles and tetrahedra domains.
\par
As well as contFEAST, since solving $LN$ ODEs \eqref{eq:ode} is the most-time consuming part of the proposed methods and is fully parallelizable, the proposed methods can be efficiently parallelized.
	\subsubsection{Advantages over contFEAST}
	\label{sec:adv_contFEAST}
	ContFEAST is a subspace iteration method based on the $L$ dimensional subspace $\mathscr{R}(\widehat{S}_0)$ for \eqref{eq:gep}.
	Instead, the proposed methods are projection methods based on the $LM$ dimensional subspace $\mathscr{R}(\widehat{S})$. 
	From Theorem~\ref{thm:error}, we can also observe that the proposed methods using higher-order complex moments achieve higher accuracy than contFEAST, since $ \lvert f_N(\lambda_{LM+1}) \rvert < \lvert f_N(\lambda_{L+1})\rvert$.
	In other words, the proposed methods can use a smaller number of initial functions $L$ than contFEAST to achieve almost the same accuracy.
	Since the number of ODEs to solve is $LN$ in each iteration, the reduction of $L$ drastically reduces the computational costs.
	\par
	Therefore, the proposed methods exhibit smaller elapsed time than contFEAST, while maintaining almost the same high accuracy, as experimentally verified in Section~\ref{sec:experiment}.
	\subsubsection{Advantages over complex moment-based matrix eigensolvers}
	Methods using a ``solve-then-discretize'' approach, including the proposed methods and contFEAST, automatically preserves the normality or self-adjointness of the problems with respect to a relevant Hilbert space $\mathcal{H}$.
	In addition, the stability analysis in \cite{horning2020feast} shows that the sensitivity of the eigenvalues is preserved by Rayleigh--Ritz-type projection methods with an $\mathcal{H}$-orthonormal basis for self-adjoint DEPs, but can be increased by methods using a ``discretize-then-solve'' approach.
	As well as contFEAST, contSS-RR follows this result.
	\par
	Based on these properties, the proposed methods exhibit much higher accuracy than the complex moment-based matrix eigensolvers using a ``solve-then-discretize'' approach, as experimentally verified in Section~\ref{sec:experiment}.
\section{Numerical experiments}
\label{sec:experiment}
In this section, we evaluate the performances of the proposed methods, contSS-RR (Algorithm~\ref{alg:ss-rr}), contSS-Hankel (Algorithm~\ref{alg:ss-hankel}), and contSS-CAA (Algorithm~\ref{alg:ss-caa}), and compare them with that of contFEAST (Algorithm~\ref{alg:feast}) for solving DEPs \eqref{eq:gep}.
Although the target problem of this paper is DEPs with ordinary differential operators, here we apply the proposed methods to DEPs with partial differential operators and evaluate their effectiveness.
\par
The compared methods use the $N$-point trapezoidal rule to approximate the contour integrals.
In Sections \ref{sec:proof_of_concept}--\ref{sec:parallel} (Experiments I--IV) for ordinary differential operators, the quadrature points for the $N$-point trapezoidal are on an ellipse with center $\gamma$, major axis $\rho$, and aspect ratio~$\alpha$, i.e.,
\begin{equation*}
	z_j = \gamma + \rho \left(\cos(\theta_j) + \alpha\textrm{i} \sin(\theta_j)\right), \quad
	\theta_j = \frac{2 \pi}{N} \left(j-\frac{1}{2}\right), \quad
	j = 1, 2, \ldots, N.
\end{equation*}
The corresponding weights are set as
\begin{equation*}
	\omega_j = \frac{\rho}{N} \left(\alpha\cos(\theta_j) + \textrm{i} \sin(\theta_j)\right), \quad
	j = 1, 2, \ldots, N.
\end{equation*}
Here, for real problems, we used \eqref{eq:sym} to reduce the number of ODEs to be solved.
In Section~\ref{sec:partial_differential} (Experiment V) for partial differential operators, we used the real rational filtering technique~\eqref{eq:Chebyshev} to avoid complex partial differential equations (PDEs).
For the proposed methods, we set $\delta = 10^{-14}$ for the threshold of the low-rank approximation.
In all the methods, we set $V: \mathbb{C}^{L} \rightarrow \mathcal{H}$ to a random quasi-matrix, whose columns are randomly generated functions represented by using 32 Chebyshev points on the same domain with the target problem.
\par
Methods were implemented using MATLAB and Chebfun \cite{driscoll2014chebfun}.
ODEs and PDEs were solved by using the ``$\backslash$'' command of Chebfun.
All the numerical experiments were performed on a serial computer with the Microsoft (R) Windows (R) 10 Pro Operating System, an 11th Gen Intel(R) Core(TM) i7-1185G7 @ 3.00GHz CPU, and 32GB RAM.
\subsection{Experiment I: proof of concept}
\label{sec:proof_of_concept}
For a proof of concept of the proposed methods, i.e., to show an advantage of the proposed method in the ``solve-then-discretize'' paradigm over a ``discretize-then-solve'' approach in terms of accuracy, we tested on the one dimensional Laplace eigenvalue problem
\begin{equation}
	- \frac{ \textrm{d}^2 }{ \textrm{d}x^2 } u_i = \lambda_i u_i, \quad
	u_i(0) = u_i(\pi) = 0.
	\label{eq:laplace}
\end{equation}
Note that the true eigenpairs are $(\lambda_i, u_i) = (i^2, \sin (ix)), i \in \mathbb{Z}_+$.
We computed four eigenpairs such that $\lambda_i \in [0, 20]$.
\par
First, we apply standard ``discretize-then-solve'' approaches for solving \eqref{eq:laplace} in which the coefficient operator is discretized by a three-point central difference.
The obtained matrix eigenvalue problem of size $n$ is
\begin{figure}[t!]
\centering
\begin{minipage}{0.49\hsize}
	\centering
	\includegraphics[width=0.8\textwidth]{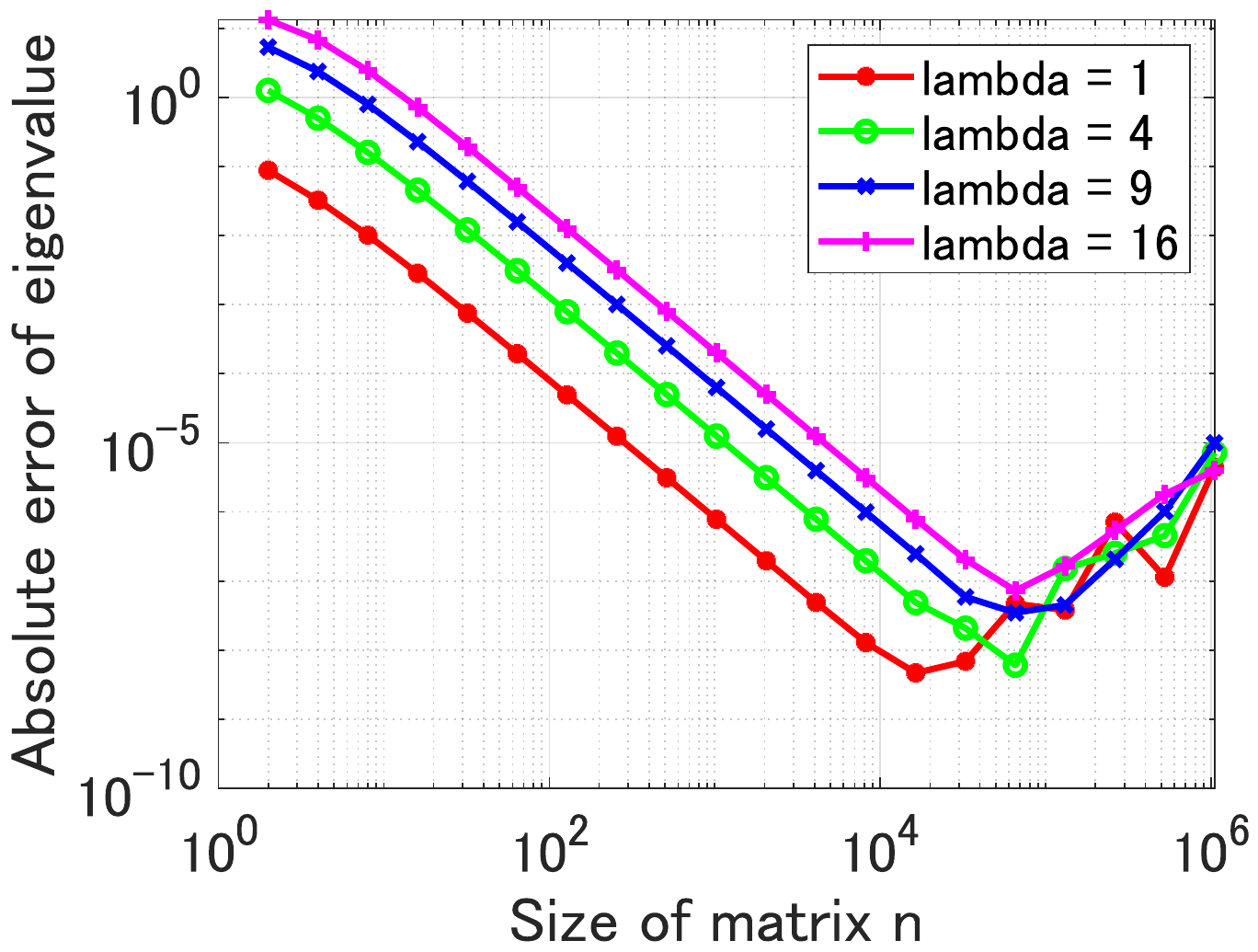}\\
	\subcaption{Using \eqref{eq:laplace_eig}.}
\end{minipage}
\begin{minipage}{0.49\hsize}
	\centering
	\includegraphics[width=0.8\textwidth]{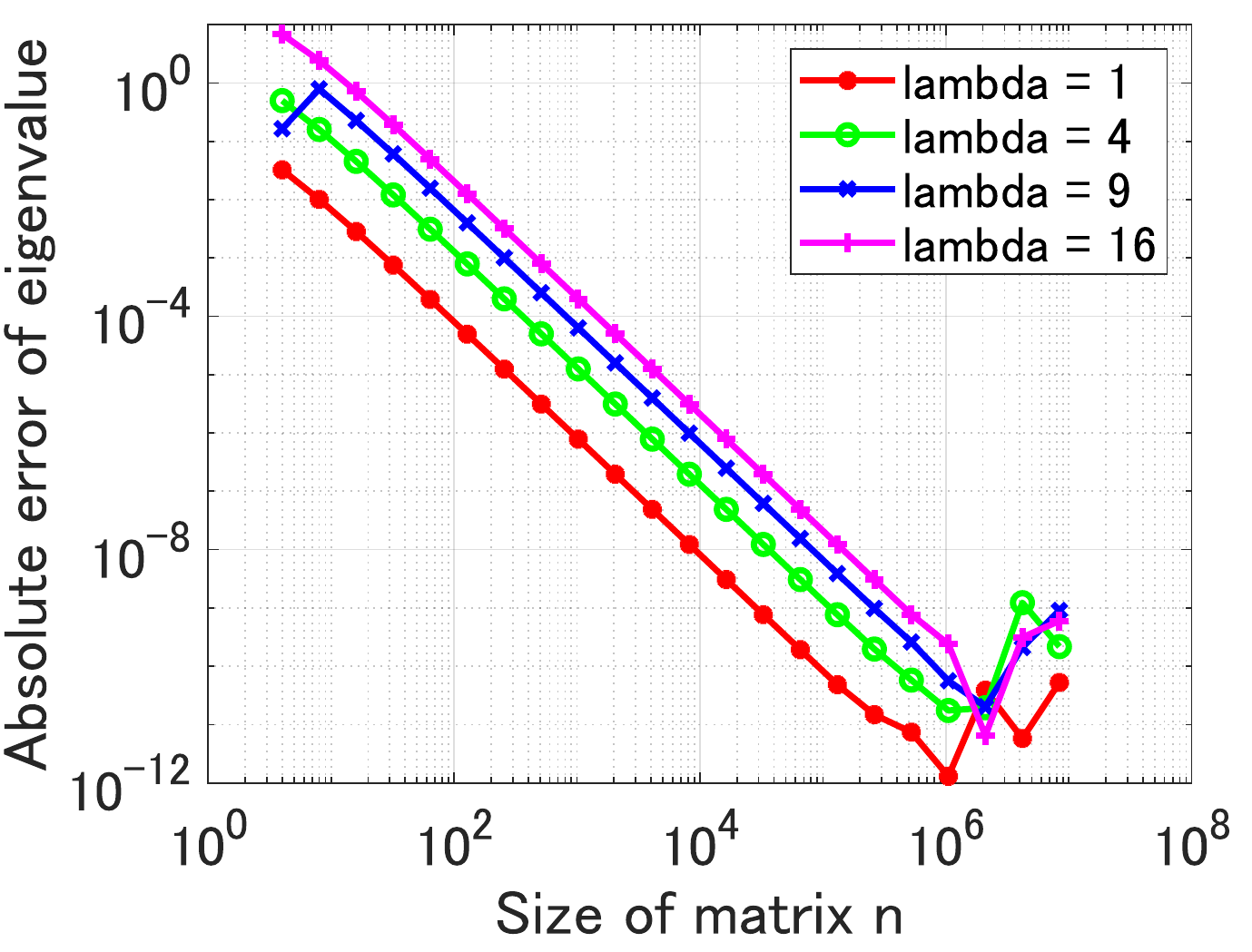}
	\subcaption{The SS-RR method for the discretized problem~\eqref{eq:laplace_problem}}
\end{minipage}
\caption{Absolute error of eigenvalue of ``discretize-then-solve'' approaches for the Laplace eigenvalue problem \eqref{eq:laplace}.}
\label{fig:Laplace}
\end{figure}
\begin{equation}
	\frac{1}{h^2} \left[
	\begin{array}{cccc}
		2 & -1 \\
		-1 & \ddots & \ddots \\
		& \ddots & \ddots & -1 \\
		&  & -1 & 2 
	\end{array}
	\right]
	{\bm u}^{(n)} = \lambda^{(n)} {\bm u}^{(n)}, \quad
	h = \frac{\pi}{n+1},
	\label{eq:laplace_problem}
\end{equation}
and its eigenvalues can be written as
\begin{equation}
	\lambda_i^{(n)} = \left(2 - 2 \cos \left(\frac{i \pi}{n+1} \right) \right) \left( \frac{n+1}{\pi}\right)^2.
	\label{eq:laplace_eig}
\end{equation}
Note that we have $\lim_{n \rightarrow \infty} \lambda_i^{(n)} = i^2$.
We computed the eigenvalues using \eqref{eq:laplace_eig} with increasing $n$.
We also applied SS-RR \cite{ikegami2010contour} with $(L,M,N) = (3,2,16)$ to the discretized problem \eqref{eq:laplace_problem} for each $n$.
\par
The absolute errors of approximate eigenvalues computed by the ``discretize-then-solve'' approaches are shown in Fig.~\ref{fig:Laplace}.
The errors decrease with increasing $n$; however, the errors turn to increase at $n \approx 10^5$ when using \eqref{eq:laplace_eig} due to rounding error and $n \approx 10^6$ for SS-RR due to quadrature and rounding errors.
The error reaches a minimum approximately $10^{-8}$ and $10^{-10}$ when using \eqref{eq:laplace_eig} and SS-RR, respectively.
\par
Next, we apply contSS-RR with $(L,M,N) = (3,2,16)$ to \eqref{eq:laplace}.
Here, we set $(\gamma,\rho,\alpha) = (10,10,1)$ for the contour path.
The obtained eigenvalues and eigenfunction are shown in Table~\ref{table:Laplace_eig} and Fig.~\ref{fig:Laplace_eigenvector}, respectively.
In contrast to the ``discretize-then-solve'' approach, the proposed method achieves much higher accuracy (absolute errors are approximately $10^{-14}$), which shows effectiveness of the ``solve-then-discretize'' approach over the ``discretize-then-solve'' approach.
This is one of the greatest advantages of the proposed method over complex moment-based matrix eigensolvers.
\begin{figure}[t]
	\centering
	\includegraphics[width=0.4\textwidth]{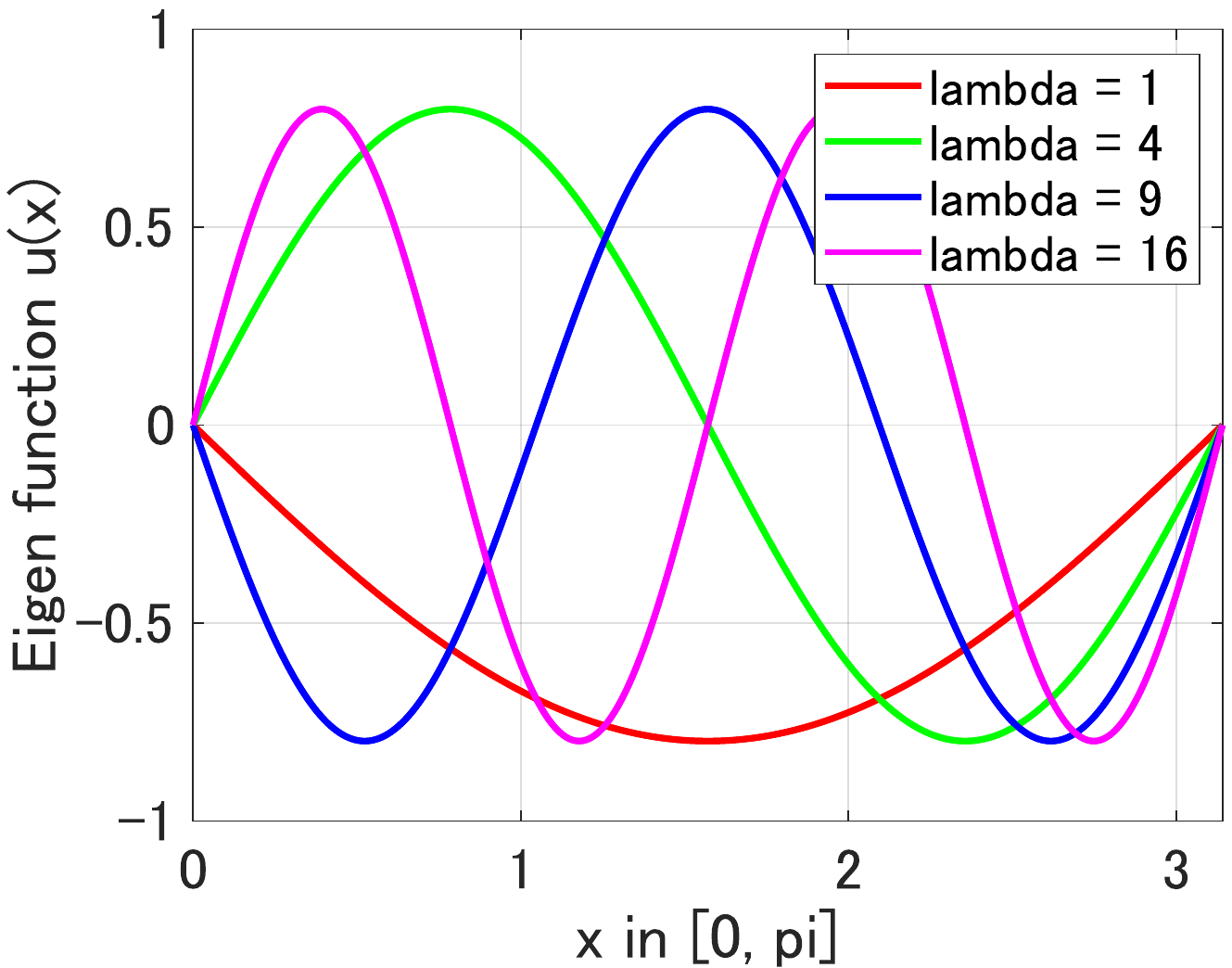}
	\caption{Obtained eigenfunction of the contSS-RR for the Laplace eigenvalue problem \eqref{eq:laplace}.}
	\label{fig:Laplace_eigenvector}
\end{figure}
\begin{table}[t]
	\centering
	\caption{True and obtained eigenvalues of the contSS-RR method for the Laplace eigenvalue problem \eqref{eq:laplace}.}
	\label{table:Laplace_eig}
	\begin{tabular}{crc} 
		\toprule
		True eigenvalue & \multicolumn{1}{c}{Obtained eigenvalue} & Absolute error \\ \cmidrule{1-3}
		1.0  &  0.999999999999997 & $3.00 \times 10^{-15}$ \\
		4.0  &  4.000000000000006 & $6.22 \times 10^{-15}$ \\
		9.0  &  9.000000000000020 & $1.95 \times 10^{-14}$ \\
		16.0 & 16.000000000000011 & $1.07 \times 10^{-14}$ \\
		\bottomrule
	\end{tabular}
\end{table}
\par
These results demonstrate that the proposed methods work well for solving DEPs without discretization of the coefficient operator.
\subsection{Experiment II: parameter dependence}
\begin{figure}[t]
	\centering
	\begin{minipage}{0.49\hsize}
		\centering
		\includegraphics[width=0.8\textwidth]{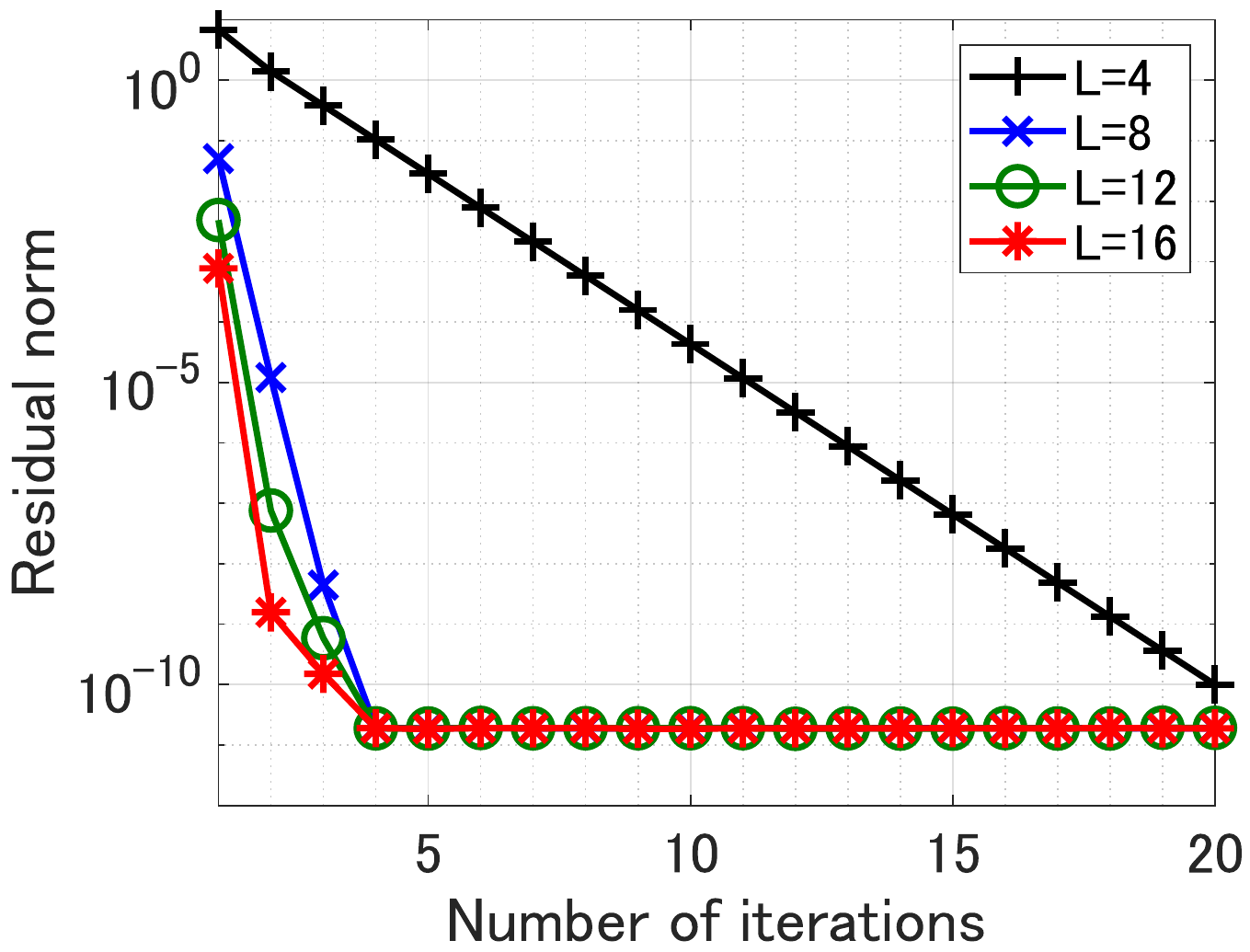}
		\subcaption{contFEAST varying $L$}
	\end{minipage}
	\begin{minipage}{0.49\hsize}
		\centering
		\includegraphics[width=0.8\textwidth]{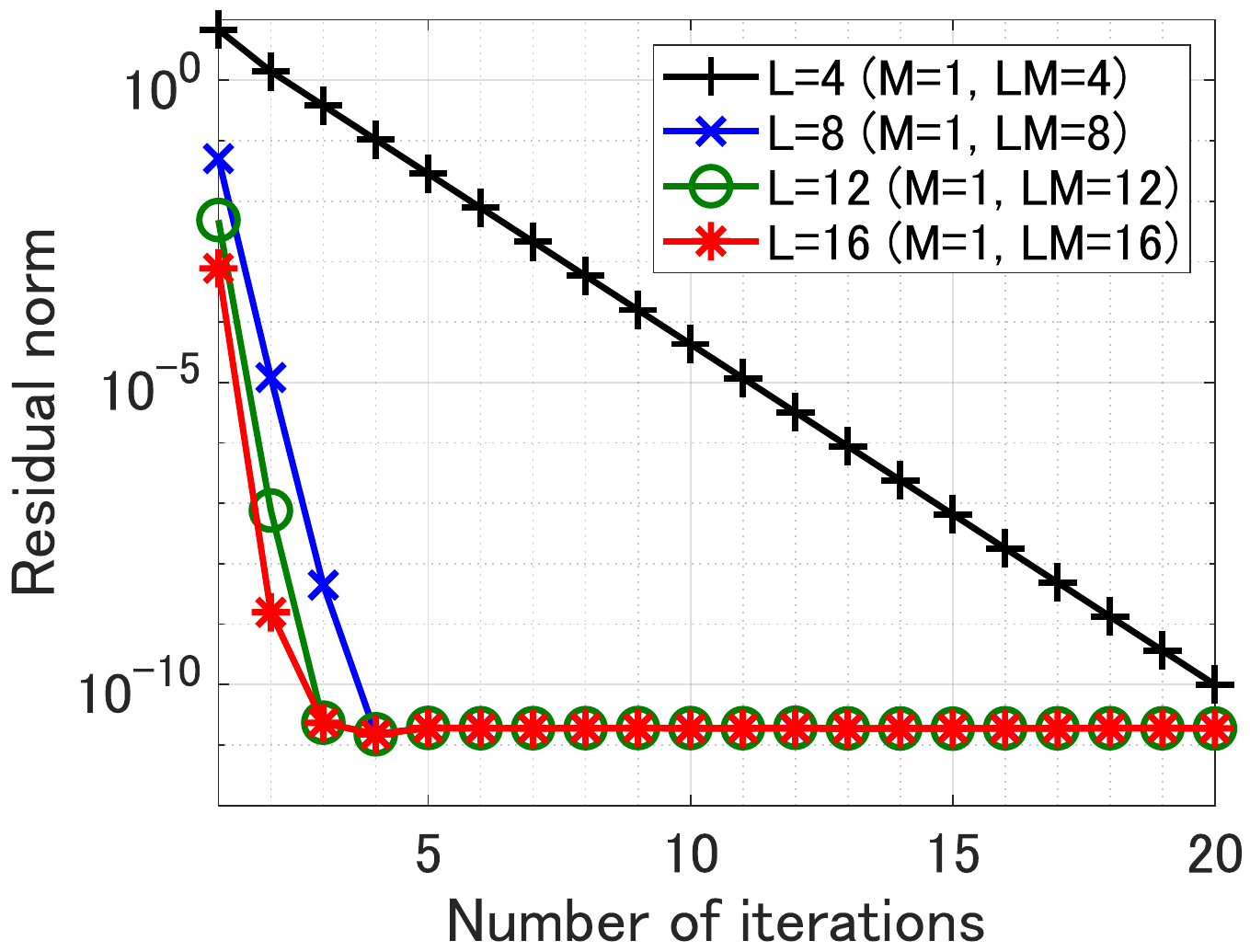}
		\subcaption{contSS-RR with $M=1$ varying $L$}
	\end{minipage}
	\begin{minipage}{0.49\hsize}
		\centering
		\includegraphics[width=0.8\textwidth]{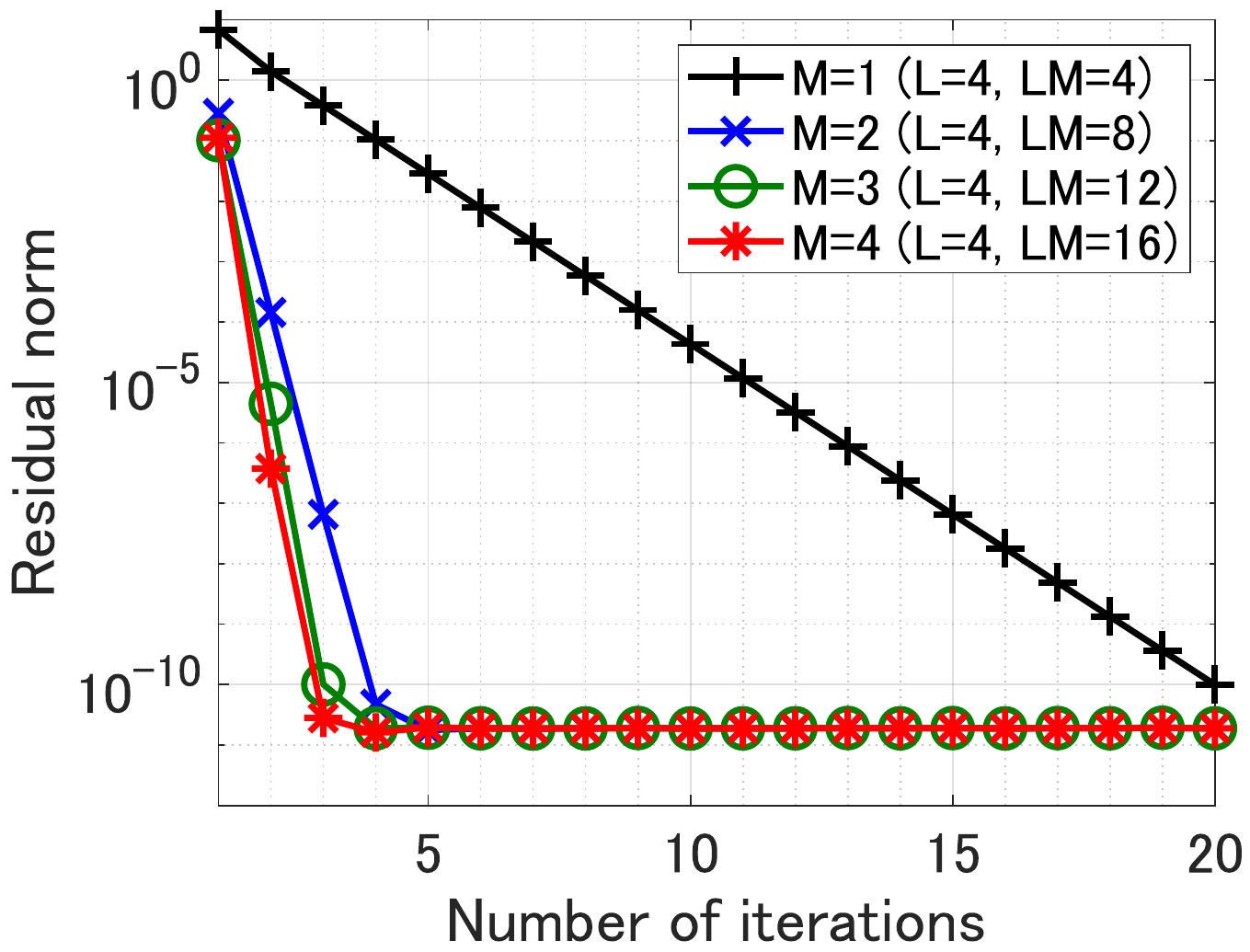}\\
		\subcaption{contSS-RR with $L=4$ varying $M$}
	\end{minipage}
	\begin{minipage}{0.49\hsize}
		\centering
		\includegraphics[width=0.8\textwidth]{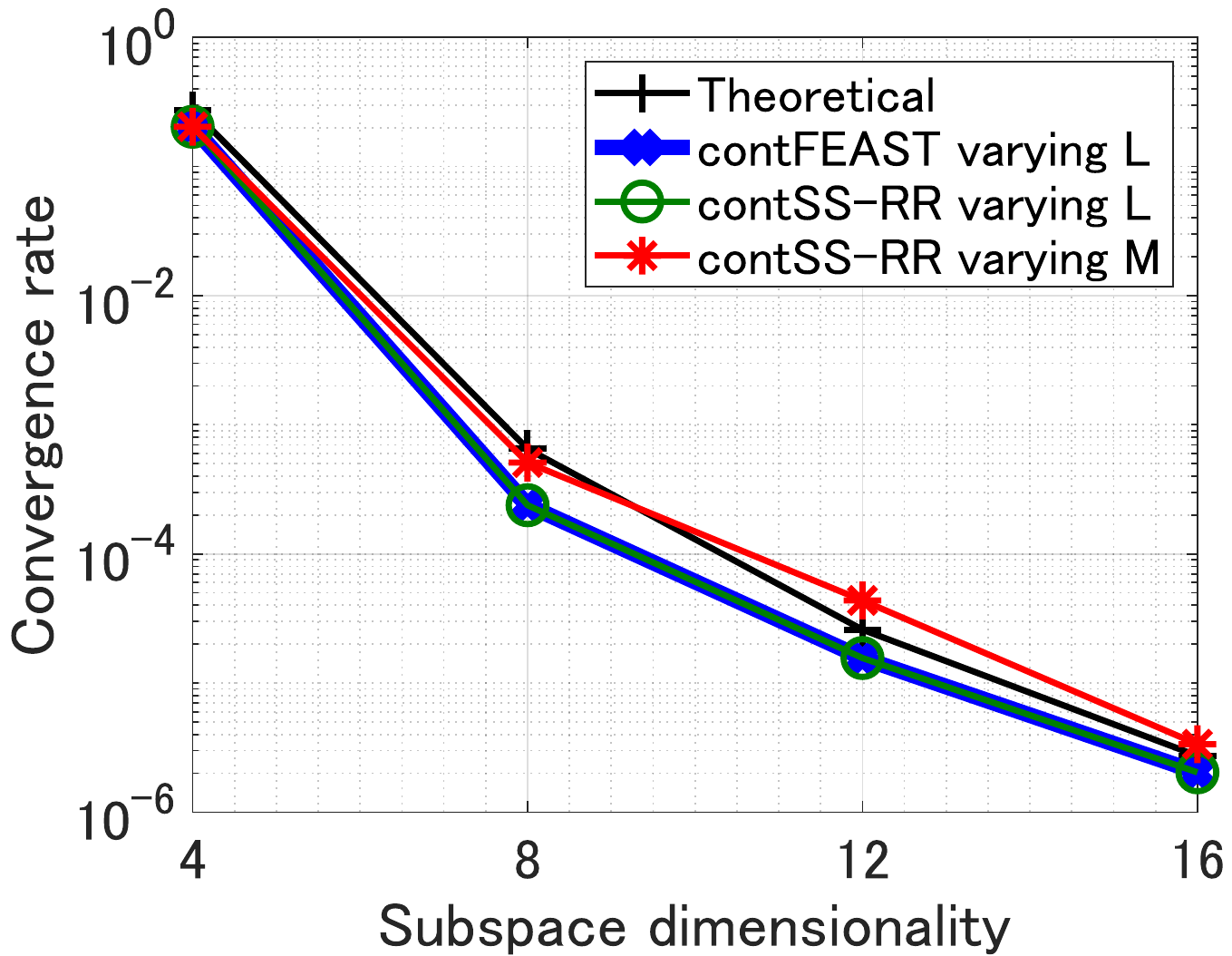}\\
		\subcaption{Convergence rate}
	\end{minipage}
	\caption{Convergence for the Laplace eigenvalue problem \eqref{eq:laplace}.}
	\label{fig:Laplace_iter}
\end{figure}
We evaluate the parameter dependence of contSS-RR with the subspace iteration technique and contFEAST on the convergence.
We computed the same four eigenpairs $\lambda_i \in [0,20]$ of \eqref{eq:laplace} using the same $(\gamma,\rho,\alpha) = (10,10,1)$ for the contour path as used in Section~4.1.
We evaluate the convergence of contSS-RR with $(L,N) = (4,4)$ varying $M = 1, 2, 3, 4$ and with $(M,N) = (1,4)$ varying $L=4, 8, 12, 16$, and contFEAST with $N=4$ varying $L=4, 8, 12, 16$.
\par
We show the residual history of each method in Fig.~\ref{fig:Laplace_iter}(a)--(c) regarding residual\break norm~$\| r_i \|_\mathcal{H} = \| \mathcal{A} \widehat{u}_i - \widehat{\lambda}_i \mathcal{B} \widehat{u}_i \|_\mathcal{H}$.
The convergences of contFEAST and contSS-RR with $M=1$ are almost identical and improve with an increase in $L$ (Fig.~\ref{fig:Laplace_iter}(a) and (b)).
We also observed that, in contSS-RR, increasing $M$ also improves the convergence to the same degree as increasing $L$ (Fig.~\ref{fig:Laplace_iter}(c)). 
\par
We also show in Fig.~\ref{fig:Laplace_iter}(d) the theoretical convergence rate obtained from Theorem~\ref{thm:error}, i.e., $\max_{\lambda_i \in \Omega} \lvert f_N(\lambda_{LM+1}) / f_N(\lambda_i) \rvert$, and the evaluated convergence rate of each method, i.e., the ratio of the residual norm between the first and second iterations.
Although, in contSS-RR, increasing $L$ indicates a slightly smaller convergence rate than increasing $M$, the both evaluated convergence rates are almost the same as the theoretical convergence rate obtained from Theorem~\ref{thm:error}.
\par
These results demonstrate that the proposed method with $M \geq 2$ achieves fast convergence even with a small value of $L$.
This contributes to the reduction in elapsed time, which will be shown in Section~4.3.
\subsection{Experiment III: performance for real-world problems}
Next, we evaluate the performances of the proposed methods without iteration ($\ell=1$) and compare them with those of contFEAST and the ``eigs'' function \cite{driscoll2008chebop} in Chebfun for the following six eigenvalue problems: two for computing real outermost eigenvalues, two for computing real interior eigenvalues and two for computing complex eigenvalues.
\begin{figure}[!t]
	\centering
	\begin{minipage}{0.49\hsize}
		\centering
		\includegraphics[width=0.8\textwidth]{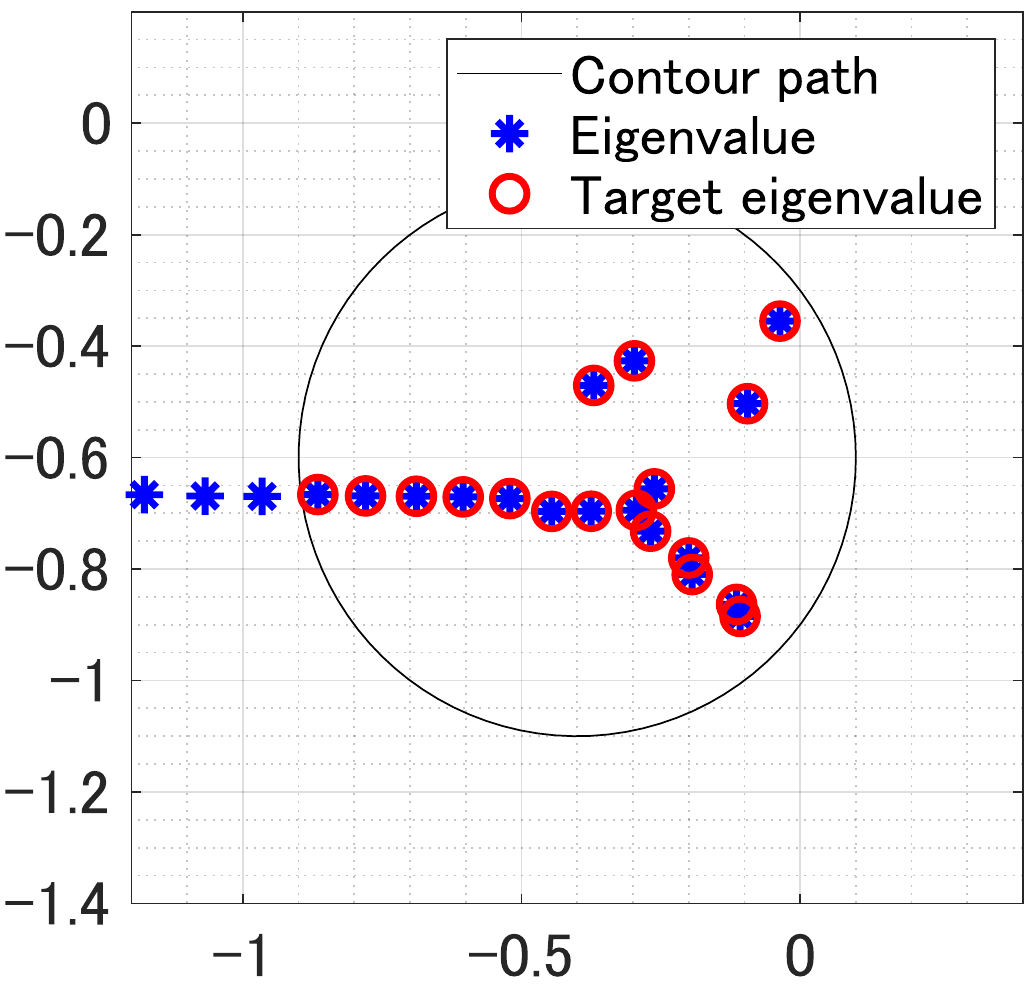}\\
		\subcaption{$Re=1000$}
	\end{minipage}
	\begin{minipage}{0.49\hsize}
		\centering
		\includegraphics[width=0.8\textwidth]{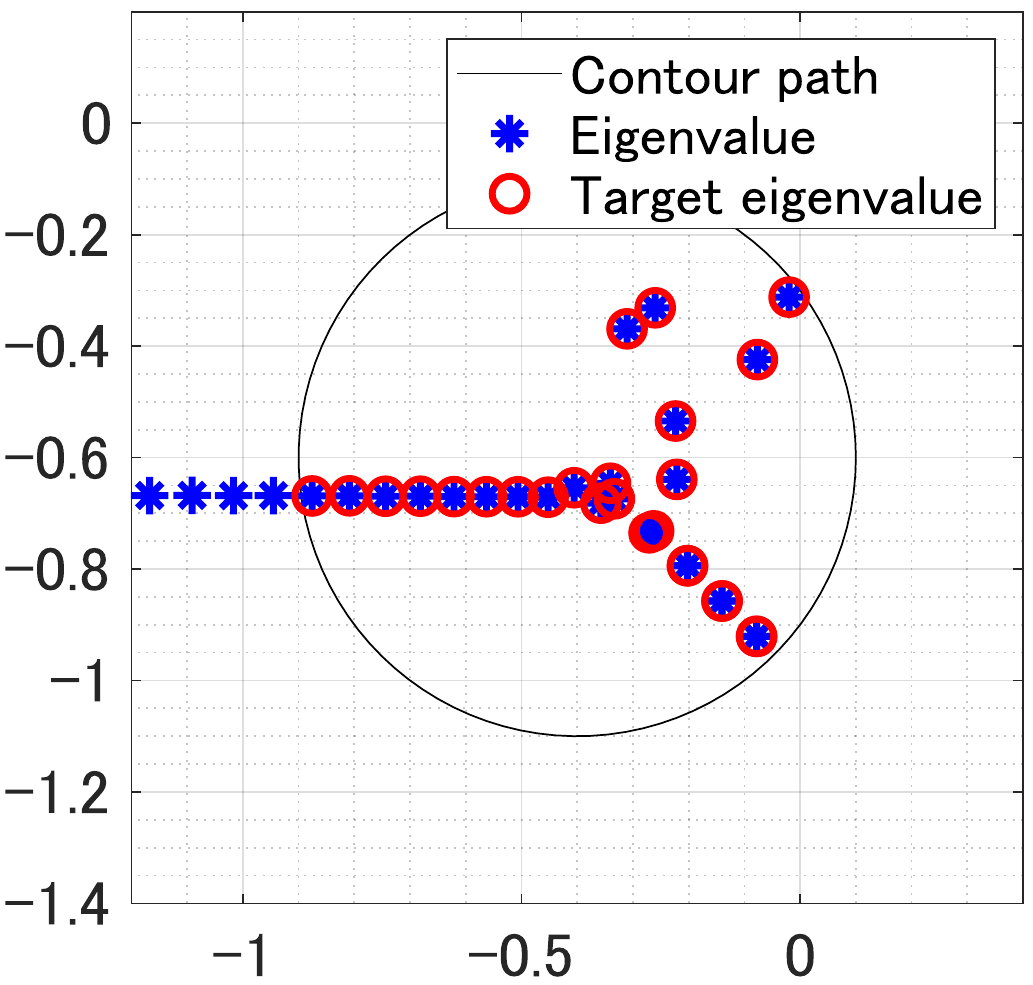}\\
		\subcaption{$Re=2000$}
	\end{minipage}
	\caption{Eigenvalues computed by the ``eigs'' function in Chebfun for the Orr--Sommerfeld eigenvalue problem.}
	\label{fig:OS_eig}
\end{figure}
\begin{itemize}
	\item \textbf{Real Outermost}: Mathieu eigenvalue problem \cite{Mathieu1868}:
	\begin{equation*}
		\left( - \frac{ \textrm{d}^2 }{ \textrm{d}x^2 } + 2 q \cos(2x) \right) u = \lambda u, \quad
		u(0) = u(\pi/2) = 0
	\end{equation*}
	with $q=2$.
	It has only real eigenvalues.
	We computed 15 eigenpairs corresponding to outermost eigenvalues $\lambda_i \in [0,1000]$.
	\item \textbf{Real Outermost}: Schr\"{o}dinger eigenvalue problem \cite[Chapter 6]{ExploringODEs2017}:
	\begin{equation*}
		\left( - \frac{\hbar}{2m} \frac{ \textrm{d}^2 }{ \textrm{d}x^2 } + V(x) \right) u = \lambda u, \quad
		u(-1) = u(1) = 0
	\end{equation*}
	with a double-well potential $V(x) = 1.5, x \in [-0.2, 0.3]$, where we set $\hbar/2m = 0.01$.
	It has only real eigenvalues.
	We computed 19 eigenpairs corresponding to outermost eigenvalues $\lambda_i \in [0,10]$.
	\item \textbf{Real Interior}: Bessel eigenvalue problem \cite{watson1995treatise}:
	\begin{equation*}
		\left( x^2 \frac{ \textrm{d}^2 }{ \textrm{d}x^2 } + x \frac{ \textrm{d} }{ \textrm{d}x} - \alpha^2 \right) u = - \lambda x^2 u, \quad
		u(0) = u(1) = 0
	\end{equation*}
	with $\alpha = 1$.
	It has only real eigenvalues.
	We computed 11 eigenpairs corresponding to interior eigenvalues $\lambda_i \in [500,3000]$.
	\item \textbf{Real Interior}: Sturm--Liouville-type eigenvalue problem:
	\begin{equation*}
		\left( - \frac{ \textrm{d}^2 }{ \textrm{d}x^2 } + x^2 \right) u = \lambda \cosh(x) u, \quad
		u(-1) = u(1) = 0,
	\end{equation*}
	which is used in \cite{horning2020feast}.
	It has only real eigenvalues.
	We computed 12 eigenpairs corresponding to interior eigenvalues $\lambda_i \in [200,1000]$.
	\item \textbf{Complex}: Orr--Sommerfeld eigenvalue problem \cite{schmid01:STS}:
	\begin{align*}
		&  \left\{ \frac{1}{ Re } \left(\frac{ \textrm{d}^2 }{ \textrm{d}x^2 } - \alpha^2\right)^2 - \textrm{i} \alpha \left[ U \left( \frac{ \textrm{d}^2}{ \textrm{d}x^2} - \alpha^2 \right) + U'' \right] \right\} u = \lambda \left( \frac{ \textrm{d}^2}{ \textrm{d}x^2} - \alpha^2 \right) u, \\
		&  u(-1) = u(1) = 0 
	\end{align*}
	with $\alpha = 1$ and $U = 1-x^2$.
	We solved two cases with $Re = 1000$ and $Re = 2000$. 
	They have complex eigenvalues.
	We computed 18 eigenpairs for $Re=1000$ and 28 eigenpairs for $Re=2000$ shown in Fig.~\ref{fig:OS_eig}.
\end{itemize}
Tables~\ref{table:path} and \ref{table:parameters} give the contour path and values of parameters for each problem.
	The ``eigs'' function in Chebfun with parameters $k$ and $\sigma$ computes $k$ closest eigenvalues to $\sigma$ and the corresponding eigenfunctions.
	We set the parameters $k$ and $\sigma$ to the number of input functions~$L$ of contFEAST in Table~\ref{table:parameters} and the center of contour path $\gamma$ in Table~\ref{table:path}, respectively.
\par
Figs.~\ref{fig:ex3_RE_res} and \ref{fig:ex3_CO_res} show the residual norms $\| r_i \|_\mathcal{H} = \| \mathcal{A} \widehat{u}_i - \widehat{\lambda}_i \mathcal{B} \widehat{u}_i \|_\mathcal{H}$ for each problem and Fig.~\ref{fig:ex3_time} shows the elapsed times for each problem.
In Fig.~\ref{fig:ex3_time}, ``Solve ODEs'', ``Orthonormalization'', ``Matrix Eig'', and ``MISC'' denote the elapsed times for solving ODEs \eqref{eq:ode}, orthonormalization of the column vectors of $\widehat{S}$, construction and solution of the matrix eigenvalue problem, and other parts including computation of the contour integral, respectively.
\begin{table}[t]
	\centering
		\caption{Contour path and the number of target eigenpairs.}
		\label{table:path}
		\begin{tabular}{lrrrc}  \toprule
			\multicolumn{1}{c}{Problem} & \multicolumn{3}{c}{Contour path} & \multicolumn{1}{c}{\# eigs}\\ 
			& \multicolumn{1}{c}{$\gamma$} & \multicolumn{1}{c}{$\rho$} & \multicolumn{1}{c}{$\alpha$} & \multicolumn{1}{c}{$m$} \\ \midrule
			Mathieu          & $500$ & $500$ & $0.1$ & $15$ \\ 
			Schr\"{o}dinger  & $5$   & $5$   & $0.1$ & $19$ \\
			Bessel           & $1750$ & $1250$ & $0.1$ & $11$ \\ 
			Sturm--Liouville & $600$ & $400$ & $0.1$ & $12$ \\ 
			Orr--Sommerfeld ($Re=1000$) & $-0.4 - 0.6 \textrm{i}$ & $0.5$ & $1.0$ & $18$ \\ 
			Orr--Sommerfeld ($Re=2000$) & $-0.4 - 0.6 \textrm{i}$ & $0.5$ & $1.0$ & $28$ \\ 
			\bottomrule
		\end{tabular}
\end{table}
\begin{table}[t]
	\centering
		\caption{Parameters.}
		\label{table:parameters}
		\begin{tabular}{lrrrrrr}  \toprule
			\multicolumn{1}{c}{Problem} & \multicolumn{6}{c}{Parameters}  \\ 
			& \multicolumn{3}{c}{contSS}  & \multicolumn{3}{c}{contFEAST} \\ \cmidrule(rl){2-4} \cmidrule(rl){2-4} \cmidrule(rl){5-7}
			& \multicolumn{1}{c}{$L$} & \multicolumn{1}{c}{$M$} & \multicolumn{1}{c}{$N$} & \multicolumn{1}{c}{$L$} & \multicolumn{1}{c}{$N$} & \multicolumn{1}{c}{$\ell$} \\ \midrule
			Mathieu          & $5$  & $8$ & $16$ & $20$ & $16$ & $1$--$3$ \\ 
			Schr\"{o}dinger  & $5$  & $8$ & $16$ & $20$ & $16$ & $1$--$3$ \\
			Bessel           & $5$  & $8$ & $16$ & $15$ & $16$ & $1$--$3$ \\ 
			Sturm--Liouville & $5$  & $8$ & $16$ & $15$  & $16$ & $1$--$3$ \\ 
			Orr--Sommerfeld ($Re=1000$) & $10$ & $8$ & $32$ & $20$ & $32$ & $1$--$3$ \\ 
			Orr--Sommerfeld ($Re=2000$) & $20$ & $8$ & $32$ & $40$ & $32$ & $1$--$3$ \\ 
			\bottomrule
		\end{tabular}
\end{table}
\par
First, we discuss the accuracy of the presented methods.
For the real outermost and interior problems (Fig.~\ref{fig:ex3_RE_res}), the residual norms of contFEAST decrease with more iterations reaching $\| r_i \|_\mathcal{H} \approx 10^{-10}$ at $\ell=3$ for the target eigenpairs.
ContSS-RR and contSS-CAA demonstrate almost the same high accuracy ($\| r_i \|_\mathcal{H} \approx 10^{-10}$) as contFEAST with $\ell=3$; on the other hand, contSS-Hankel shows lower accuracy than the others except for the Bessel eigenvalue problem.
The residual norms for the eigenvalues outside the target region tend to be large depending on the distance from the target region.
The experimental results exhibit a similar trend for both outermost and interior problems. 
\begin{figure}[!pt]
	\centering
	\begin{minipage}{0.99\hsize}
		\centering
		\includegraphics[width=0.4\textwidth]{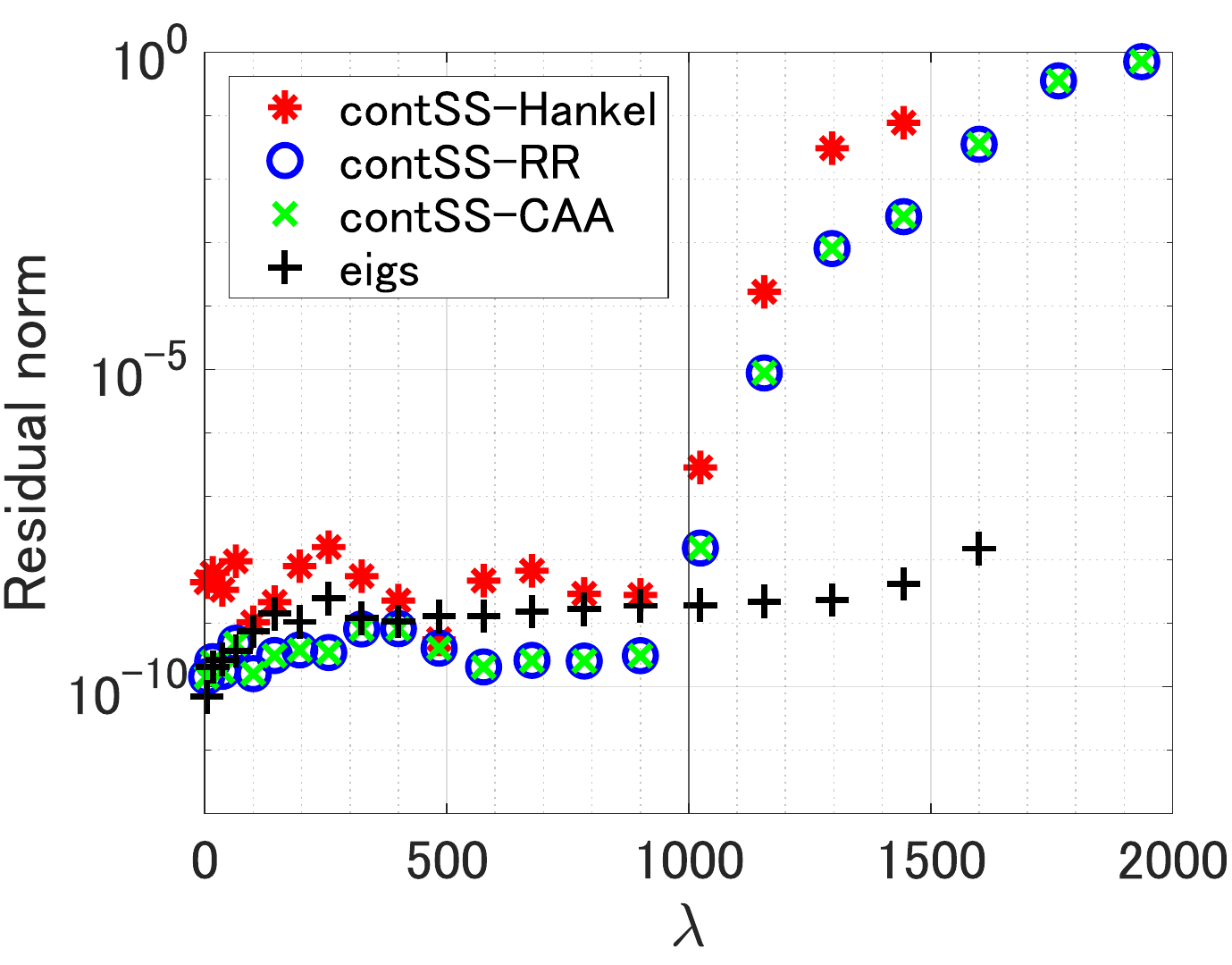}
		\hspace{0.08\textwidth}
		\includegraphics[width=0.4\textwidth]{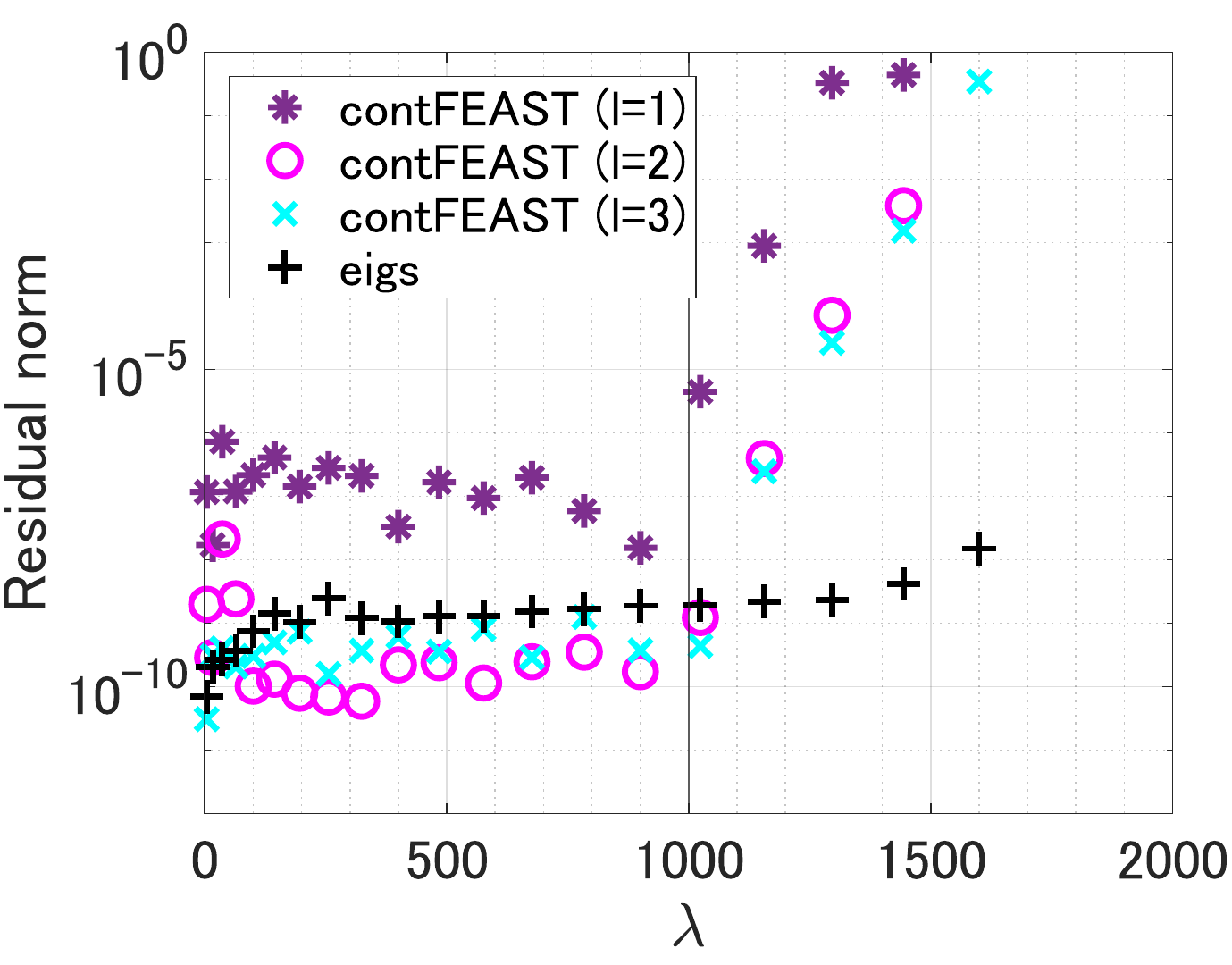}
		\subcaption{\textbf{Real {Outermost}:} Mathieu eigenvalue problem}
	\end{minipage}
	\begin{minipage}{0.99\hsize}
		\centering
		\includegraphics[width=0.4\textwidth]{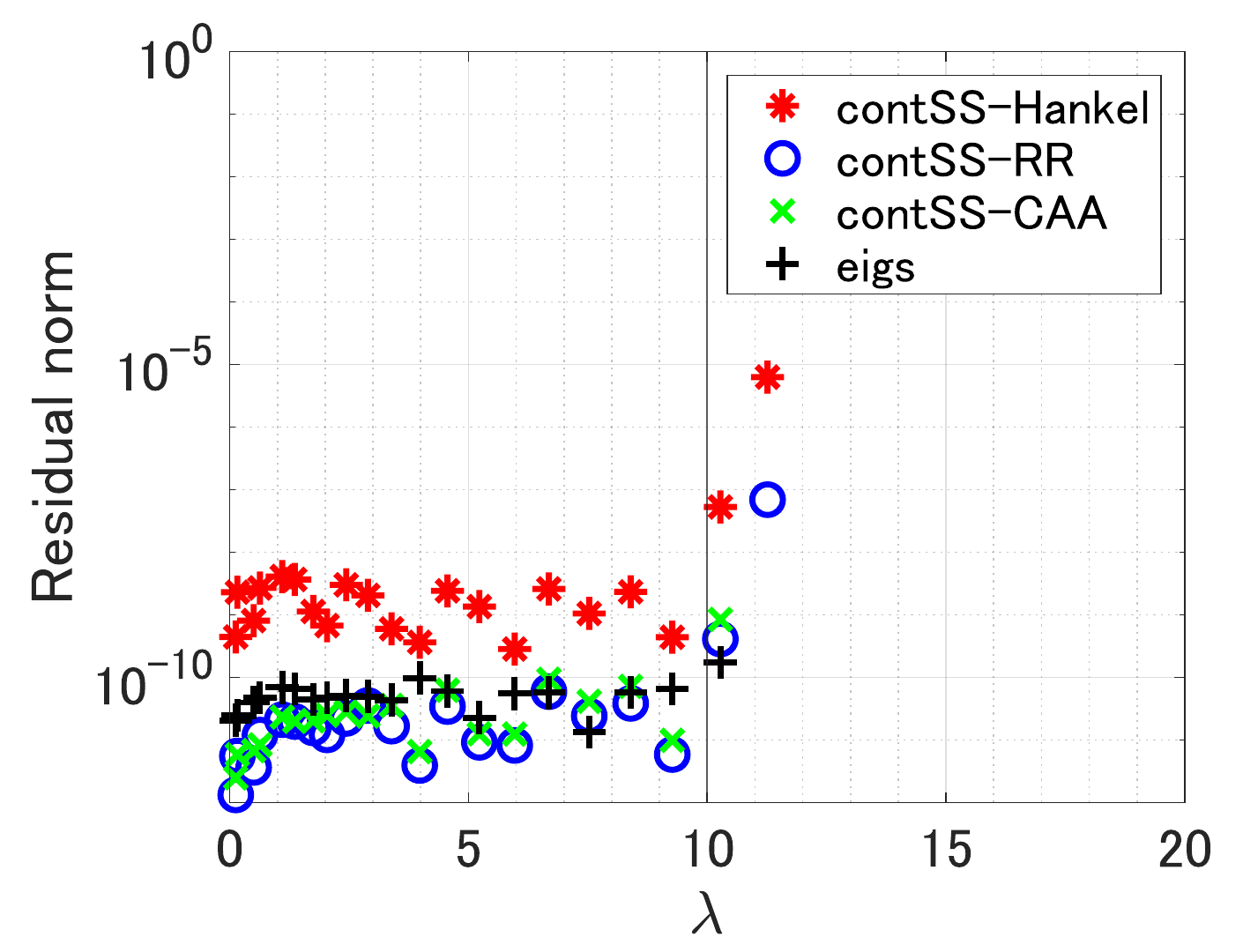}
		\hspace{0.08\textwidth}
		\includegraphics[width=0.4\textwidth]{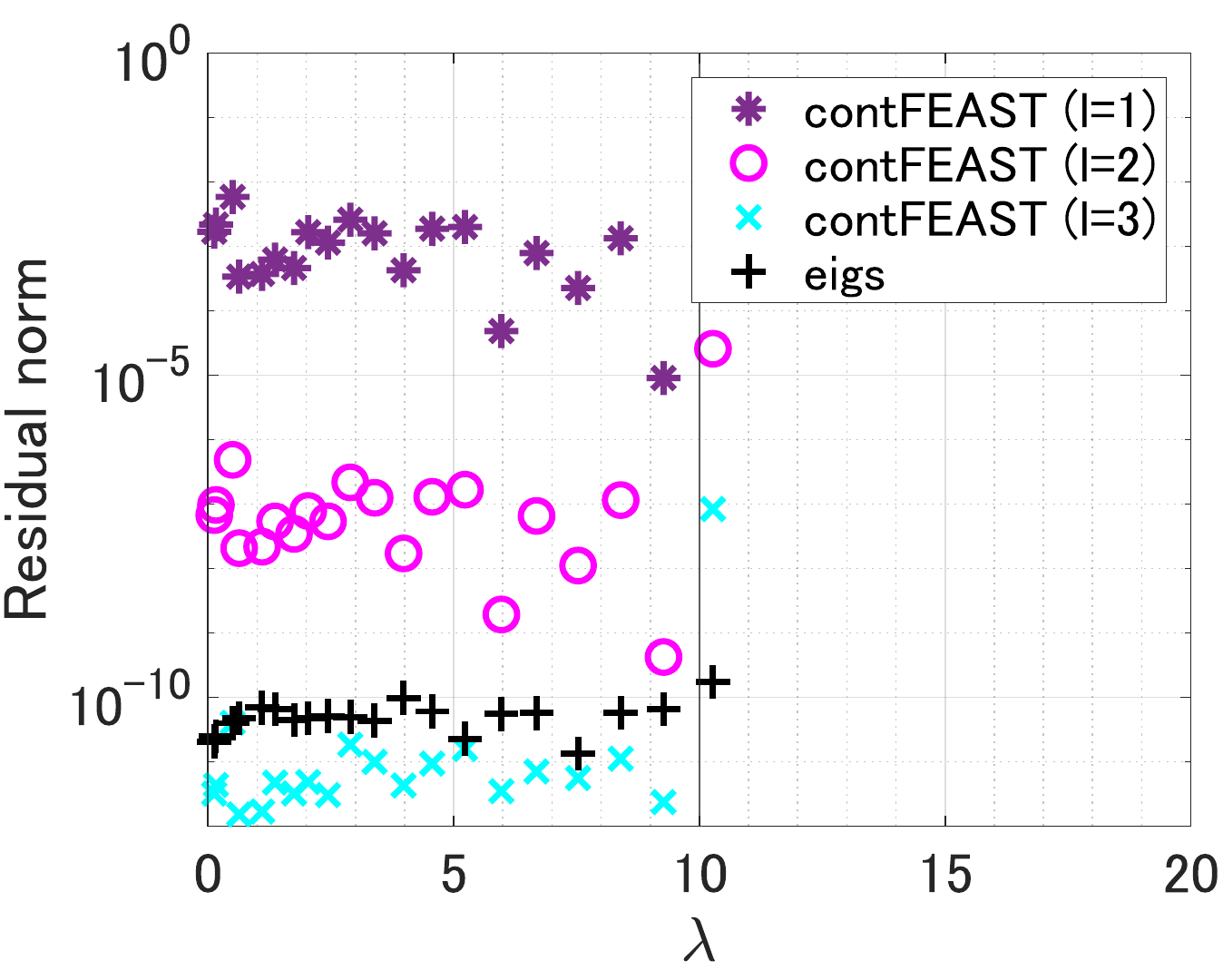}
		\subcaption{\textbf{Real {Outermost}:} Schr\"{o}dinger eigenvalue problem}
	\end{minipage}
	\begin{minipage}{0.99\hsize}
		\centering
		\includegraphics[width=0.4\textwidth]{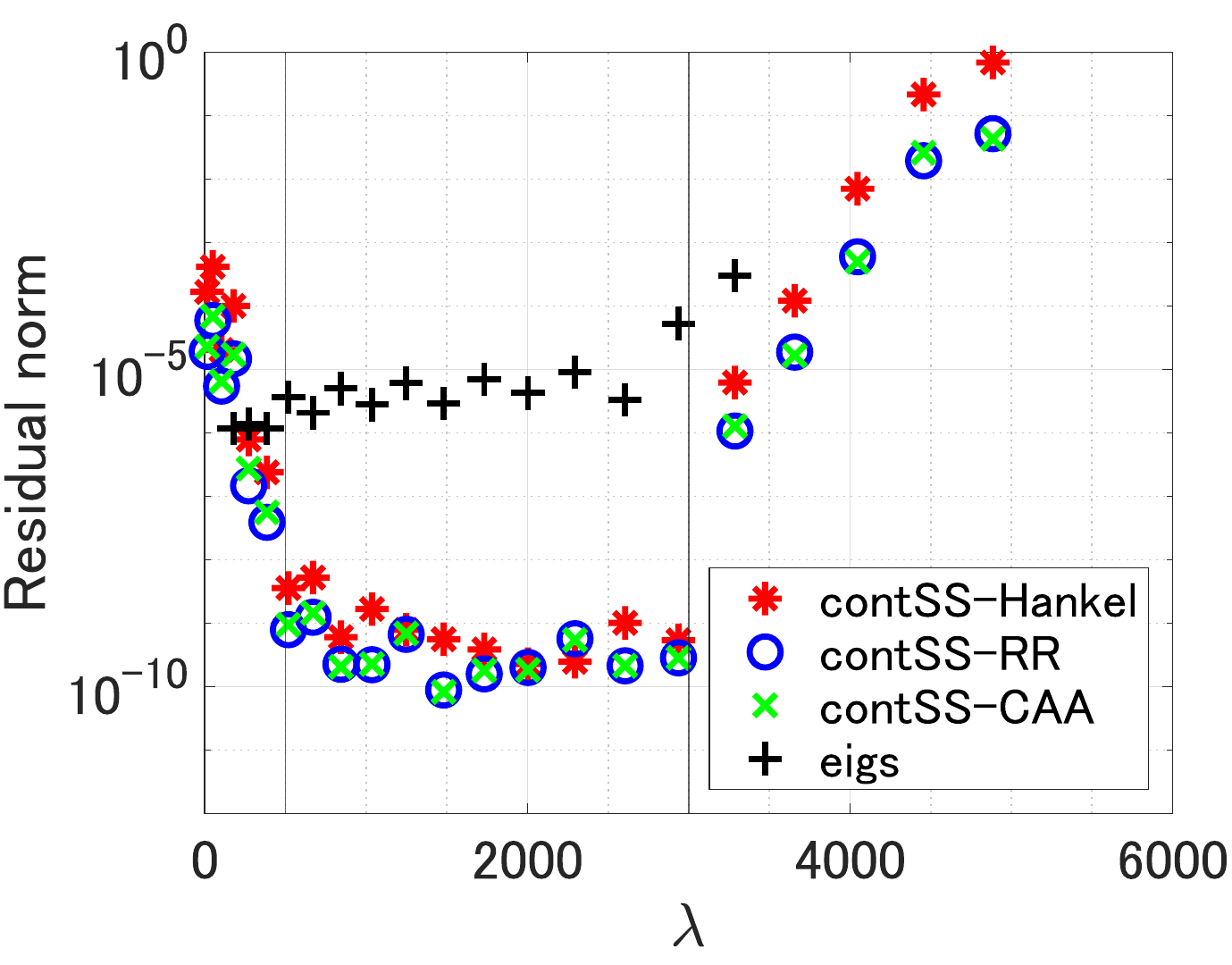}
		\hspace{0.08\textwidth}		
		\includegraphics[width=0.4\textwidth]{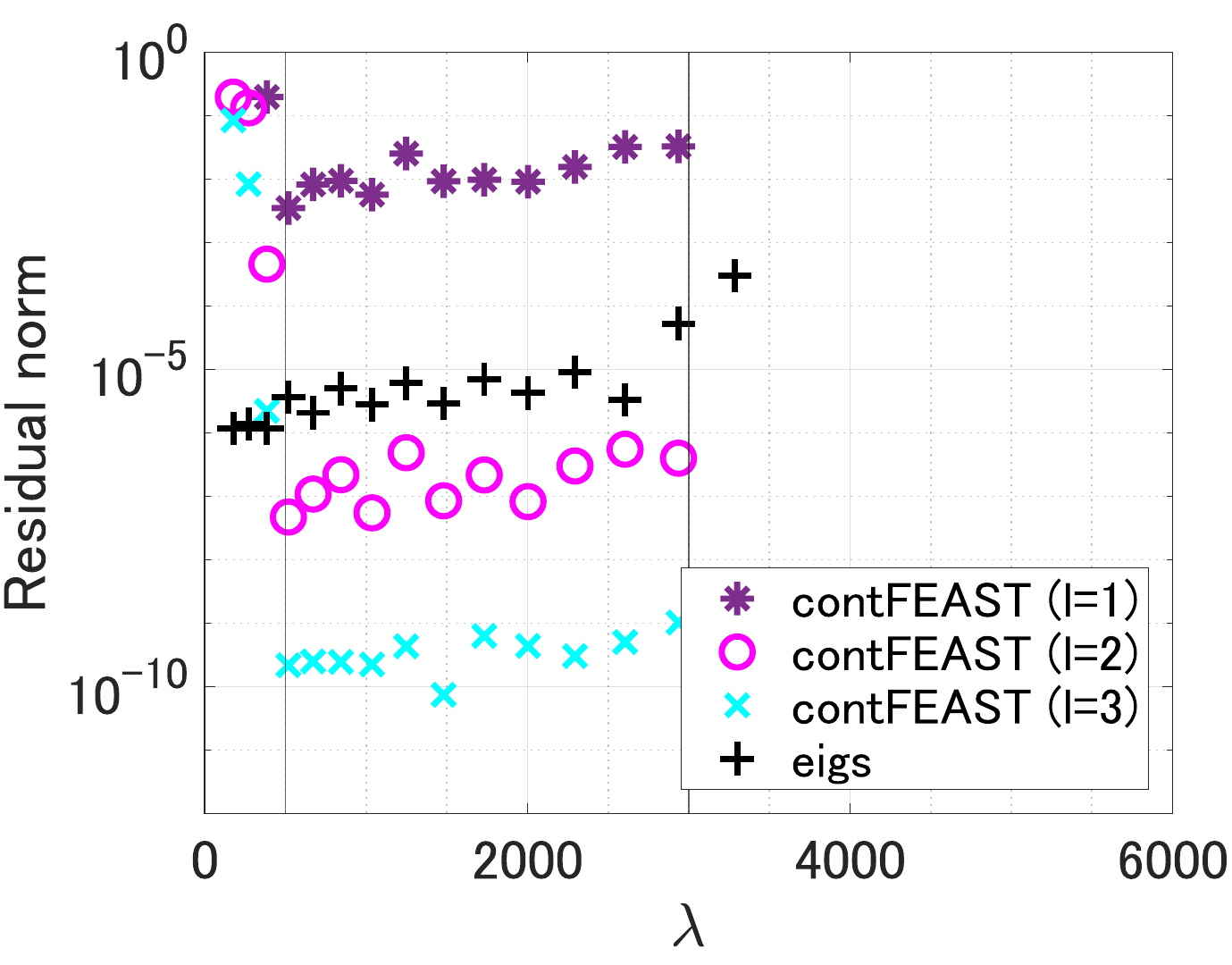}
		\subcaption{\textbf{Real Interior:} Bessel eigenvalue problem}
	\end{minipage}
	\begin{minipage}{0.99\hsize}
		\centering
		\includegraphics[width=0.4\textwidth]{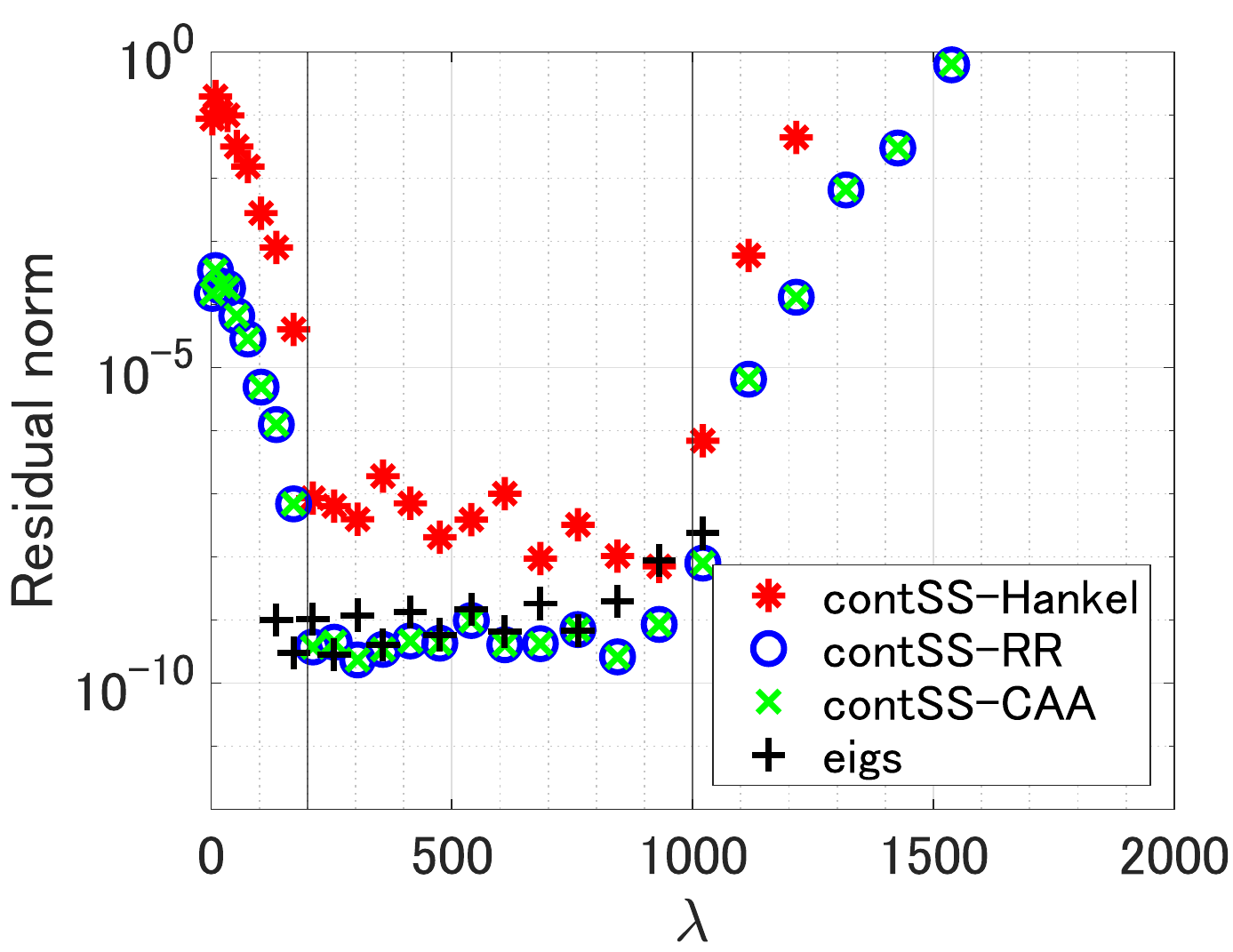}
		\hspace{0.08\textwidth}		
		\includegraphics[width=0.4\textwidth]{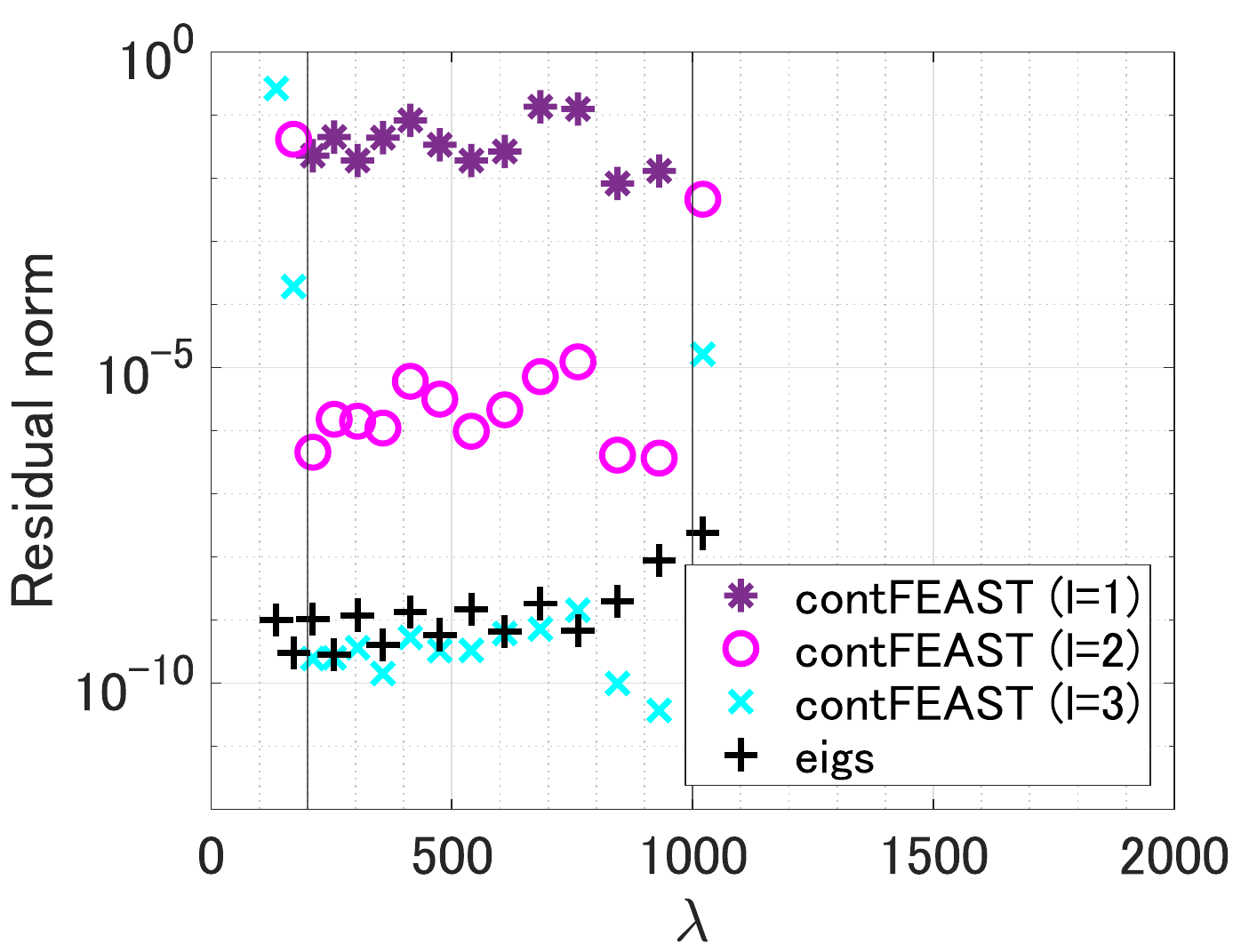}
		\subcaption{\textbf{Real Interior:} Sturm--Liouville-type eigenvalue problem}
	\end{minipage}
	\caption{Residual norm for real outermost and interior problems.}
	\label{fig:ex3_RE_res}
\end{figure}
\begin{figure}[!t]
	\centering
	\begin{minipage}{0.99\hsize}
		\centering
		\includegraphics[width=0.4\textwidth]{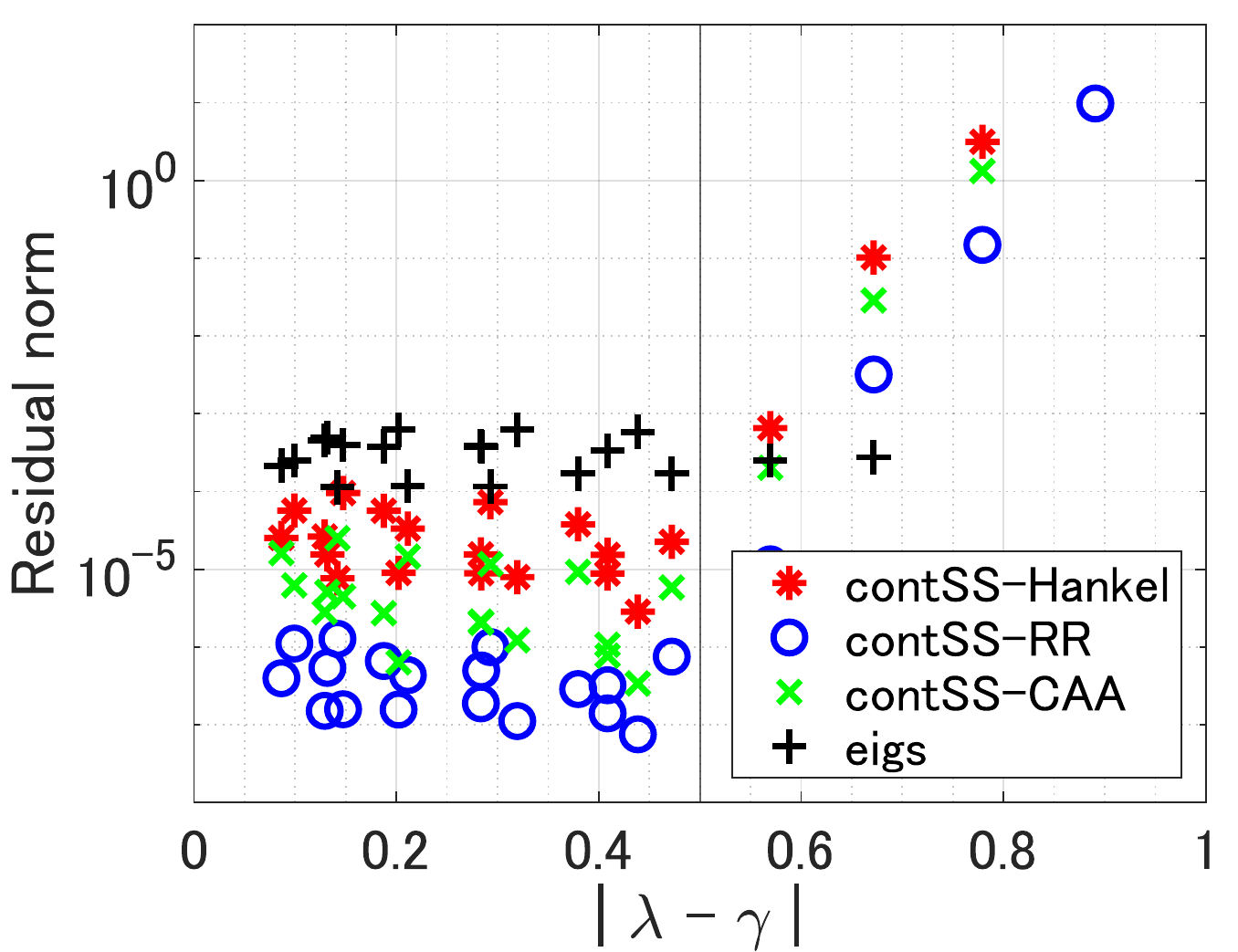}
		\hspace{0.08\textwidth}	
		\includegraphics[width=0.4\textwidth]{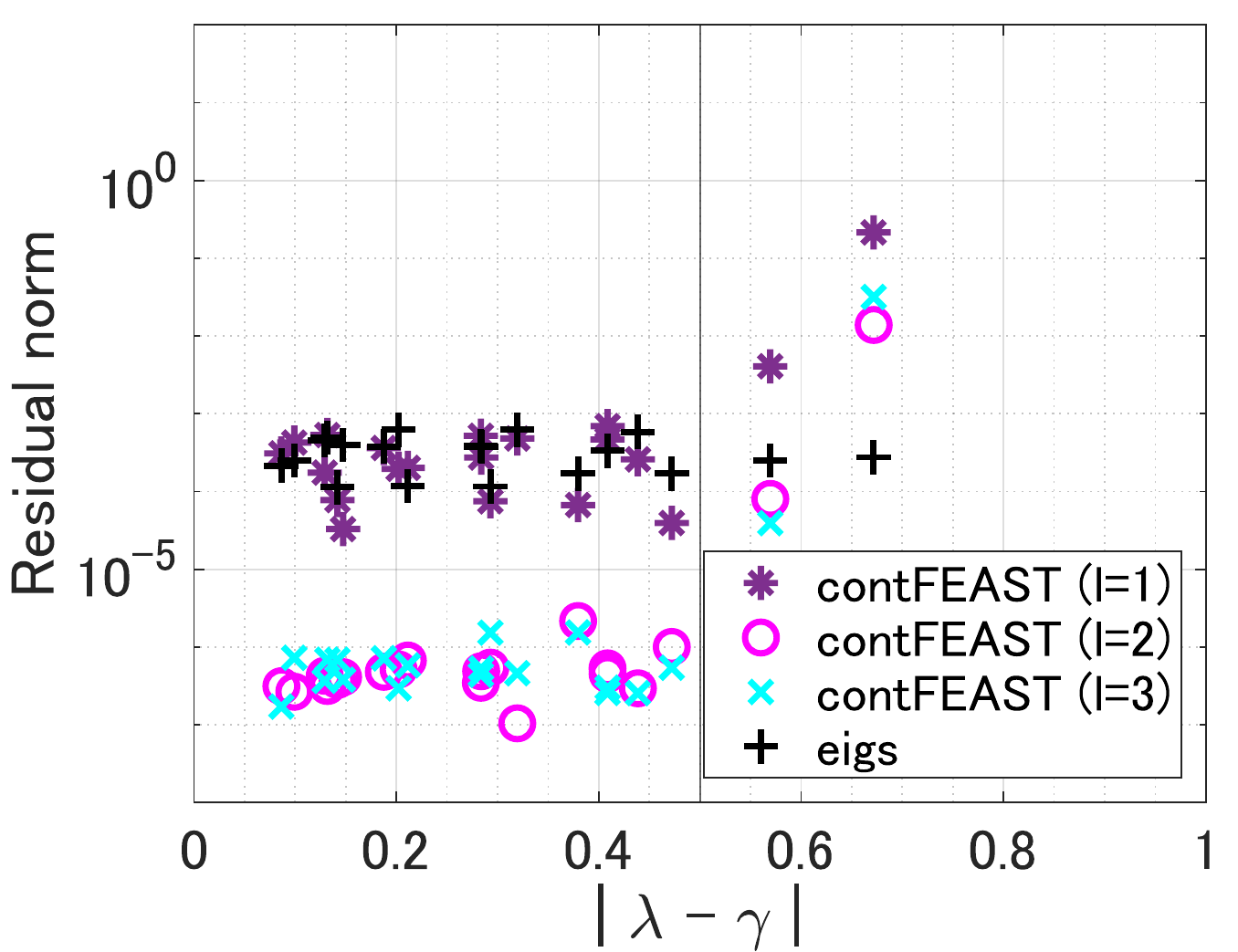}
		\subcaption{\textbf{Complex:} Orr--Sommerfeld eigenvalue problem ($Re = 1000$)}
	\end{minipage}
	\begin{minipage}{0.99\hsize}
		\centering
		\includegraphics[width=0.4\textwidth]{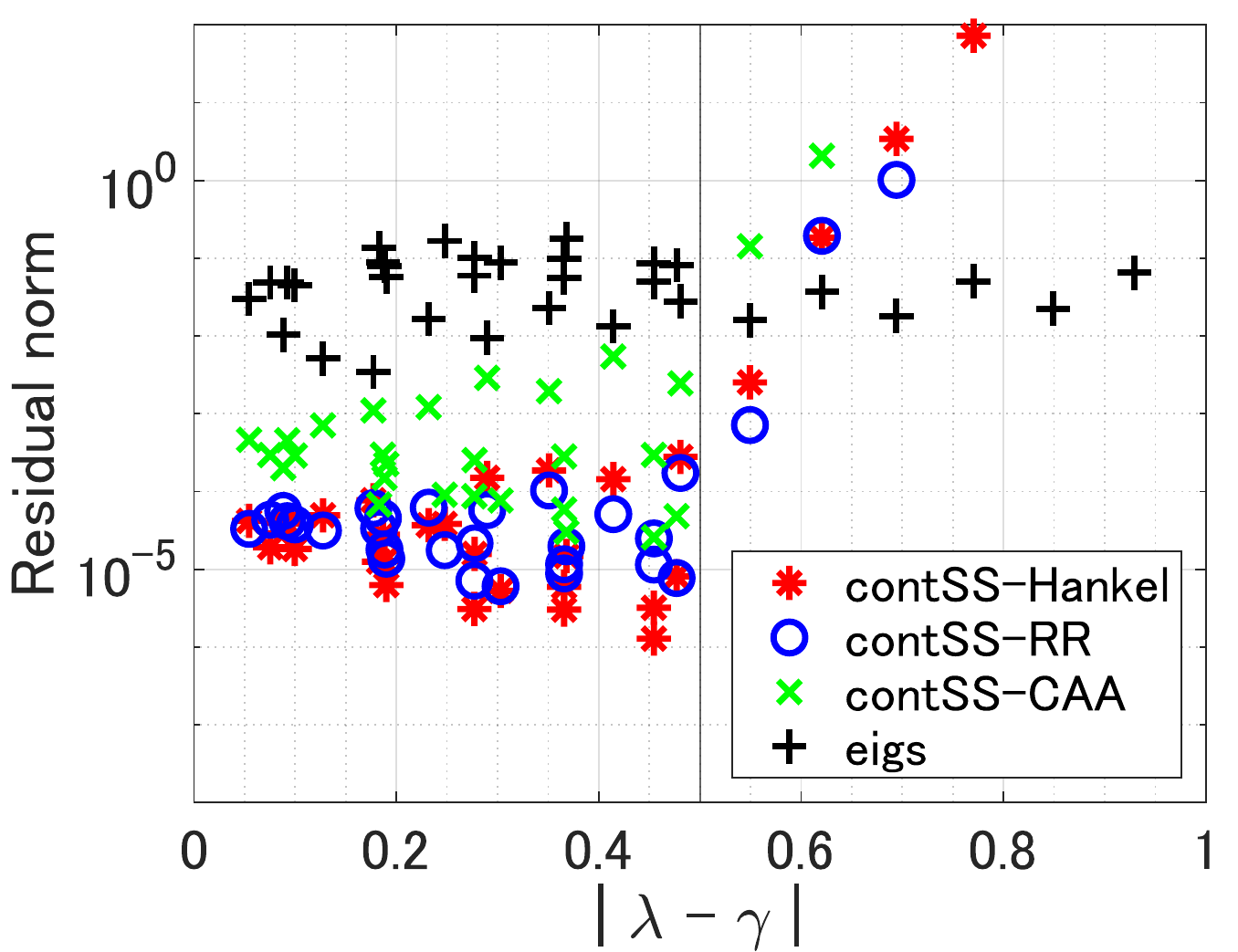}
		\hspace{0.08\textwidth}			
		\includegraphics[width=0.4\textwidth]{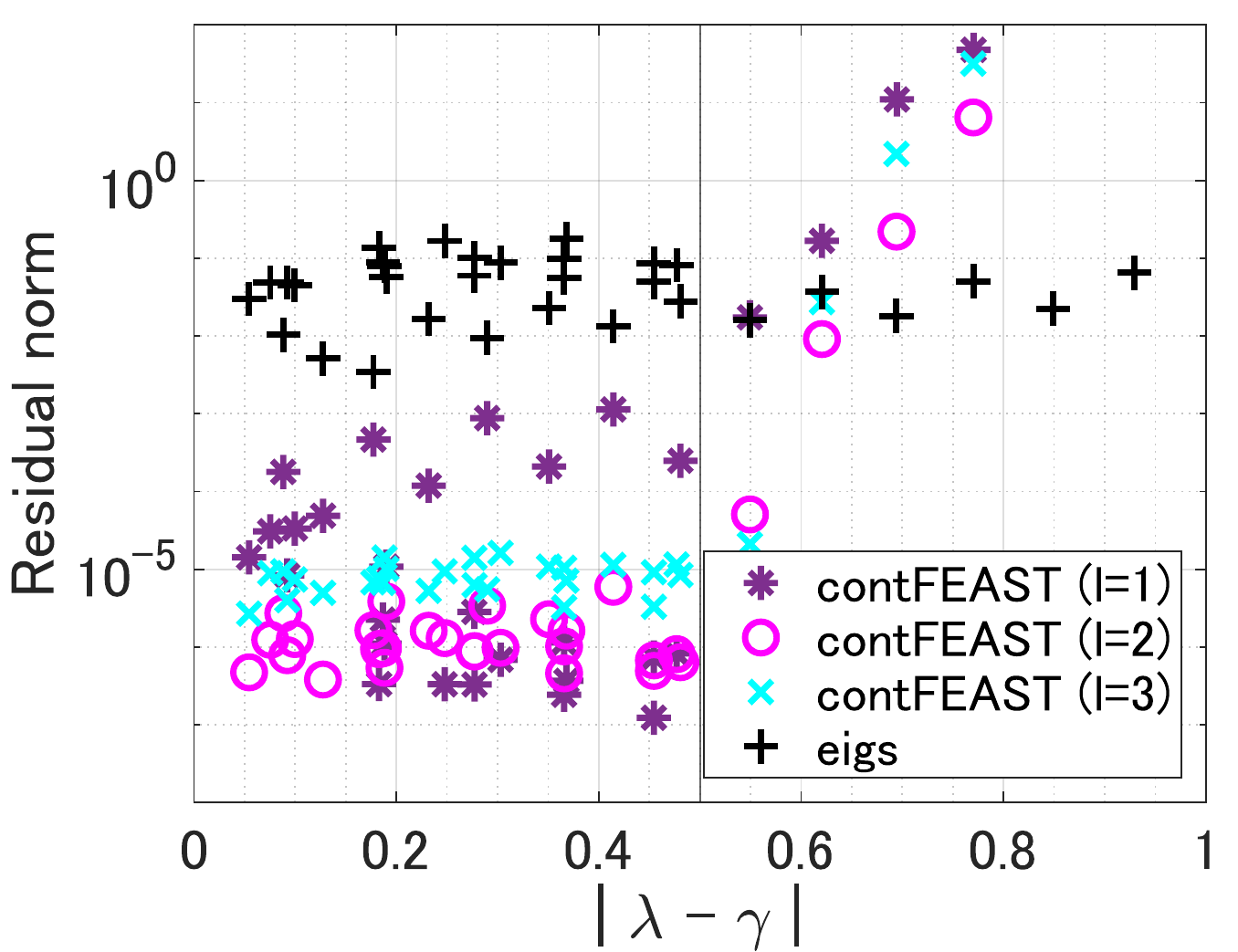}
		\subcaption{\textbf{Complex:} Orr--Sommerfeld eigenvalue problem ($Re = 2000$)}
	\end{minipage}
	\caption{Residual norm for complex problems.}
	\label{fig:ex3_CO_res}
\end{figure}
\begin{figure}[!t]
	\centering
	\begin{minipage}{0.49\hsize}
		\centering
		\includegraphics[width=0.8\textwidth]{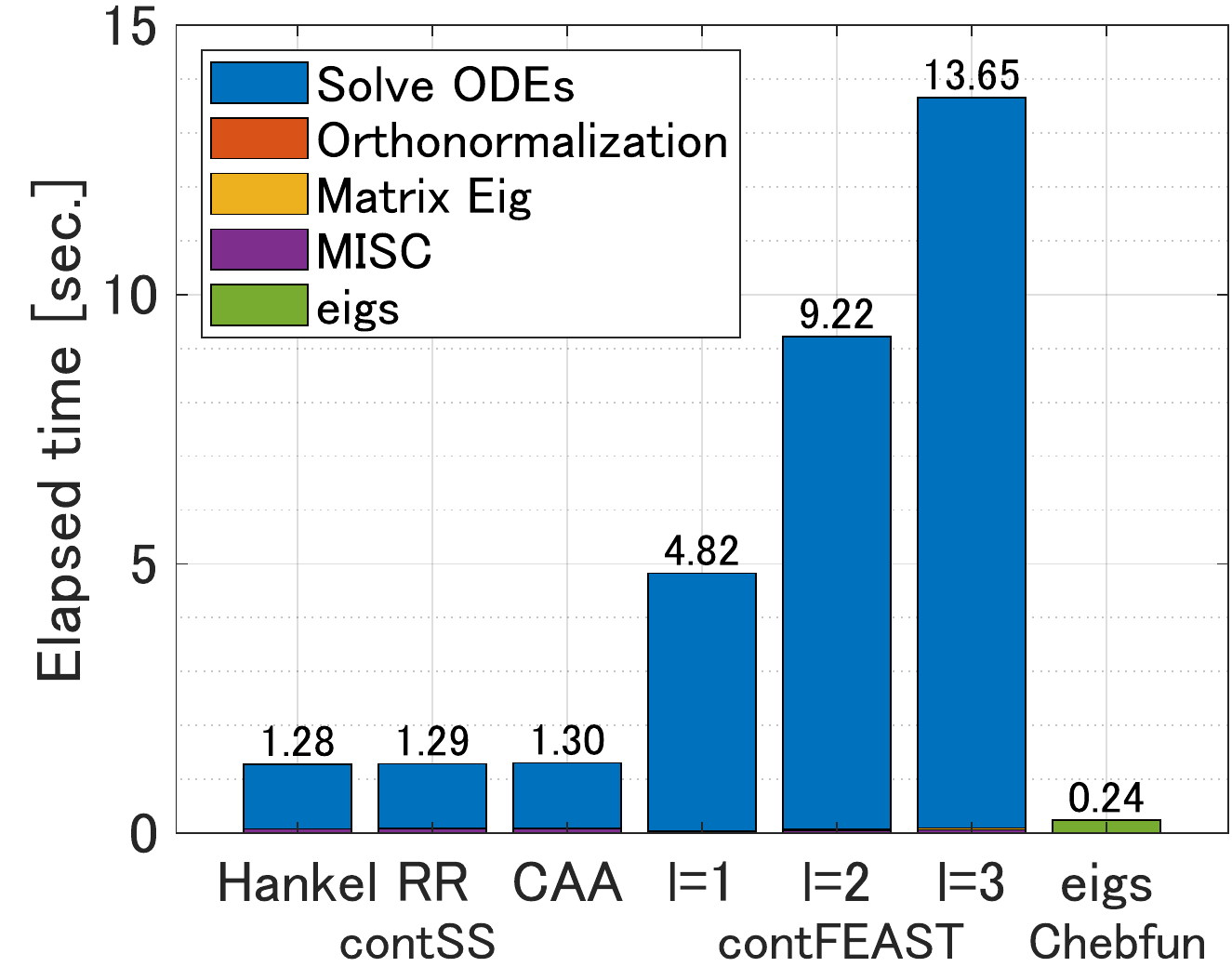}
		\subcaption{Mathieu eigenvalue problem}
	\end{minipage}
	\begin{minipage}{0.49\hsize}
		\centering
		\includegraphics[width=0.8\textwidth]{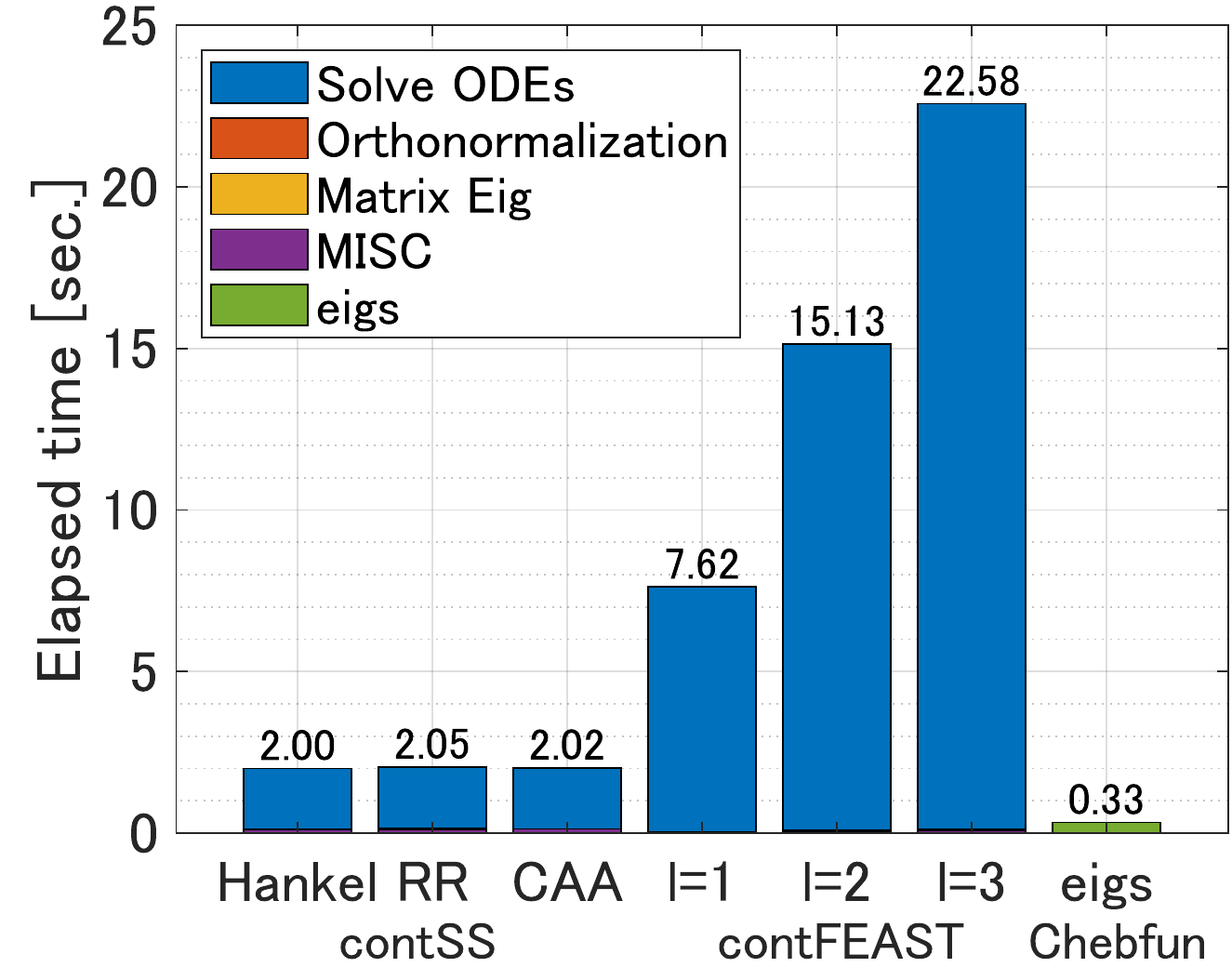}
		\subcaption{Schr\"{o}dinger eigenvalue problem}
	\end{minipage}
	\begin{minipage}{0.49\hsize}
		\centering
		\includegraphics[width=0.8\textwidth]{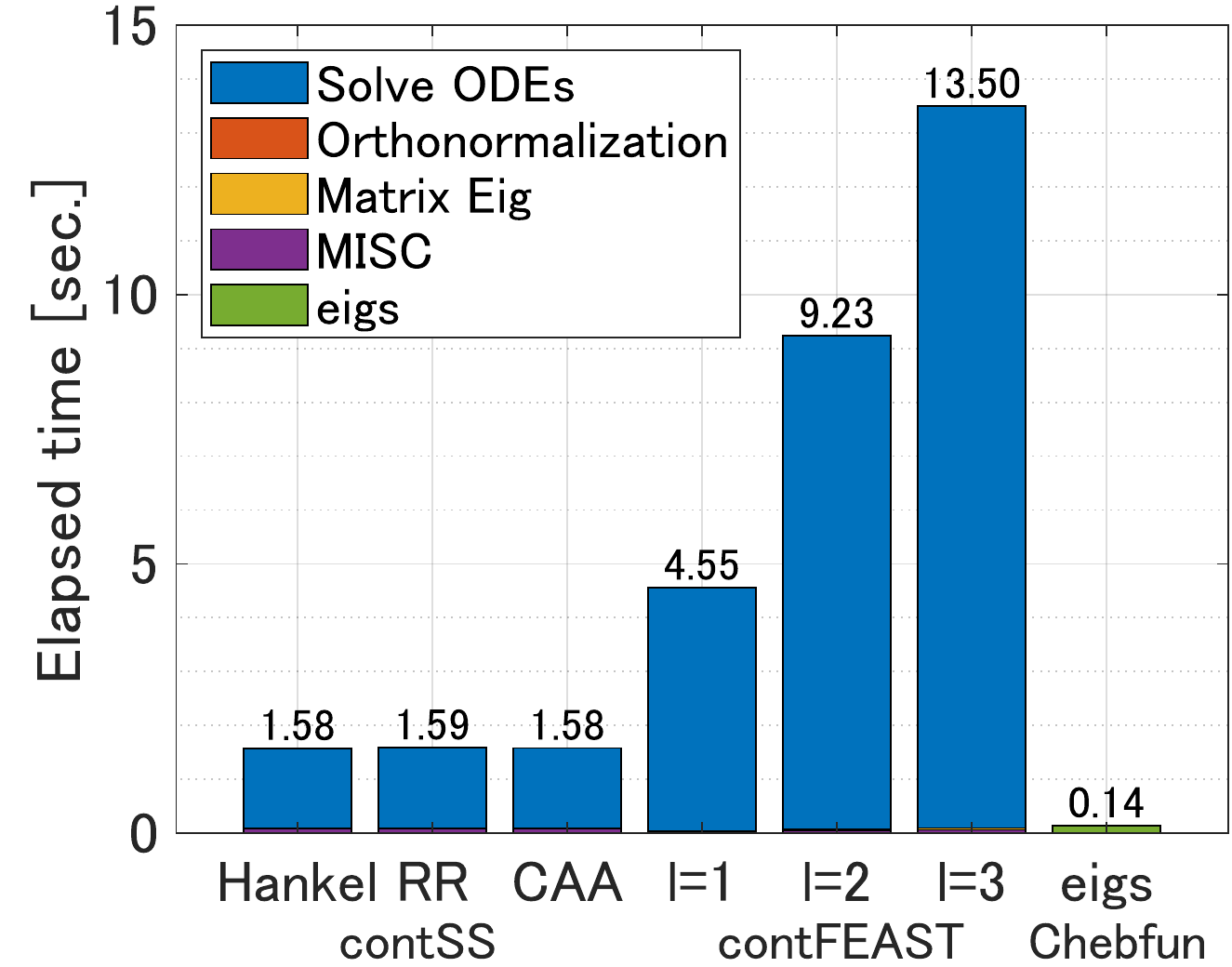}
		\subcaption{Bessel eigenvalue problem}
	\end{minipage}
	\begin{minipage}{0.49\hsize}
		\centering
		\includegraphics[width=0.8\textwidth]{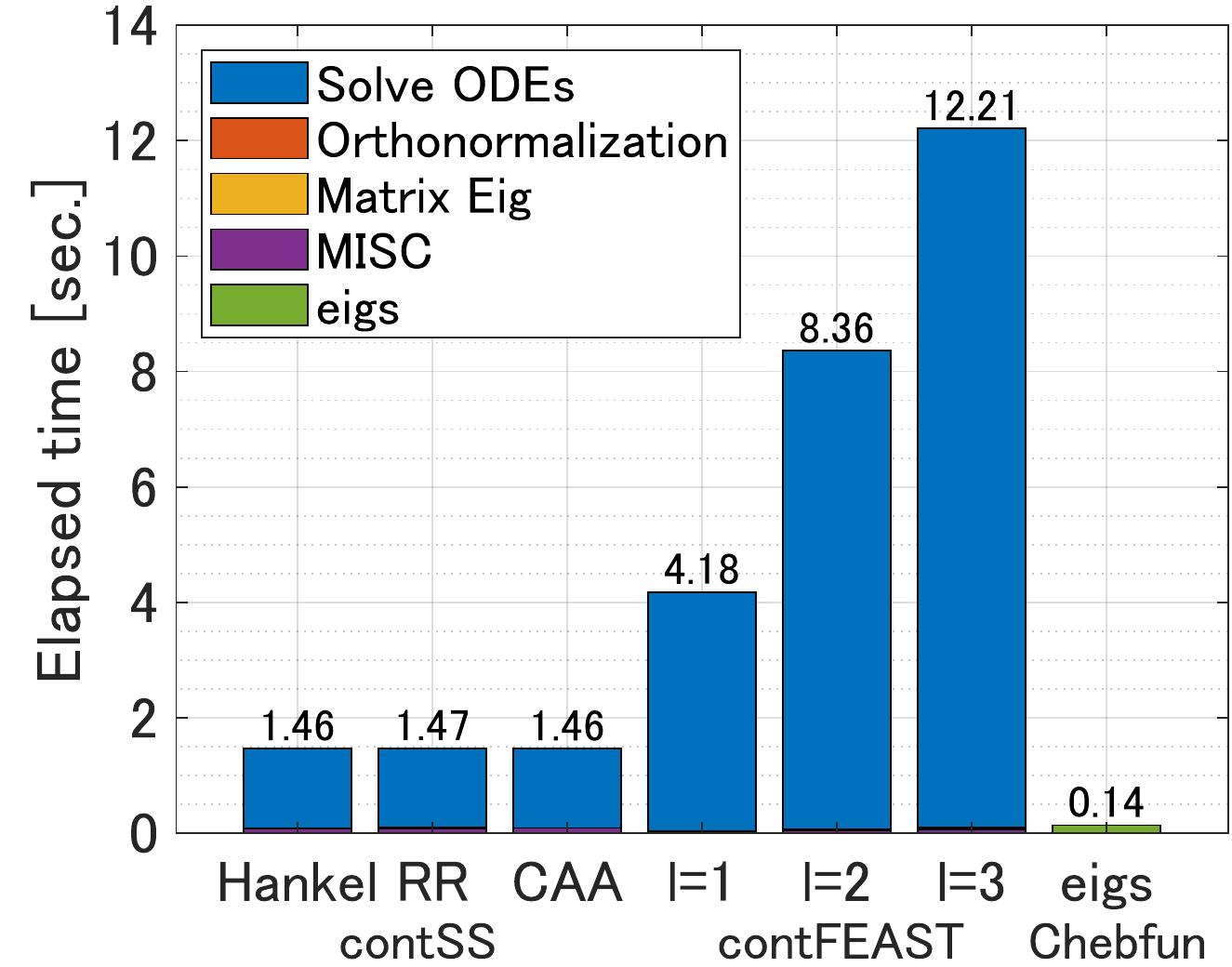}
		\subcaption{Sturm--Liouville-type eigenvalue problem}
	\end{minipage}
	\begin{minipage}{0.49\hsize}
		\centering
		\includegraphics[width=0.8\textwidth]{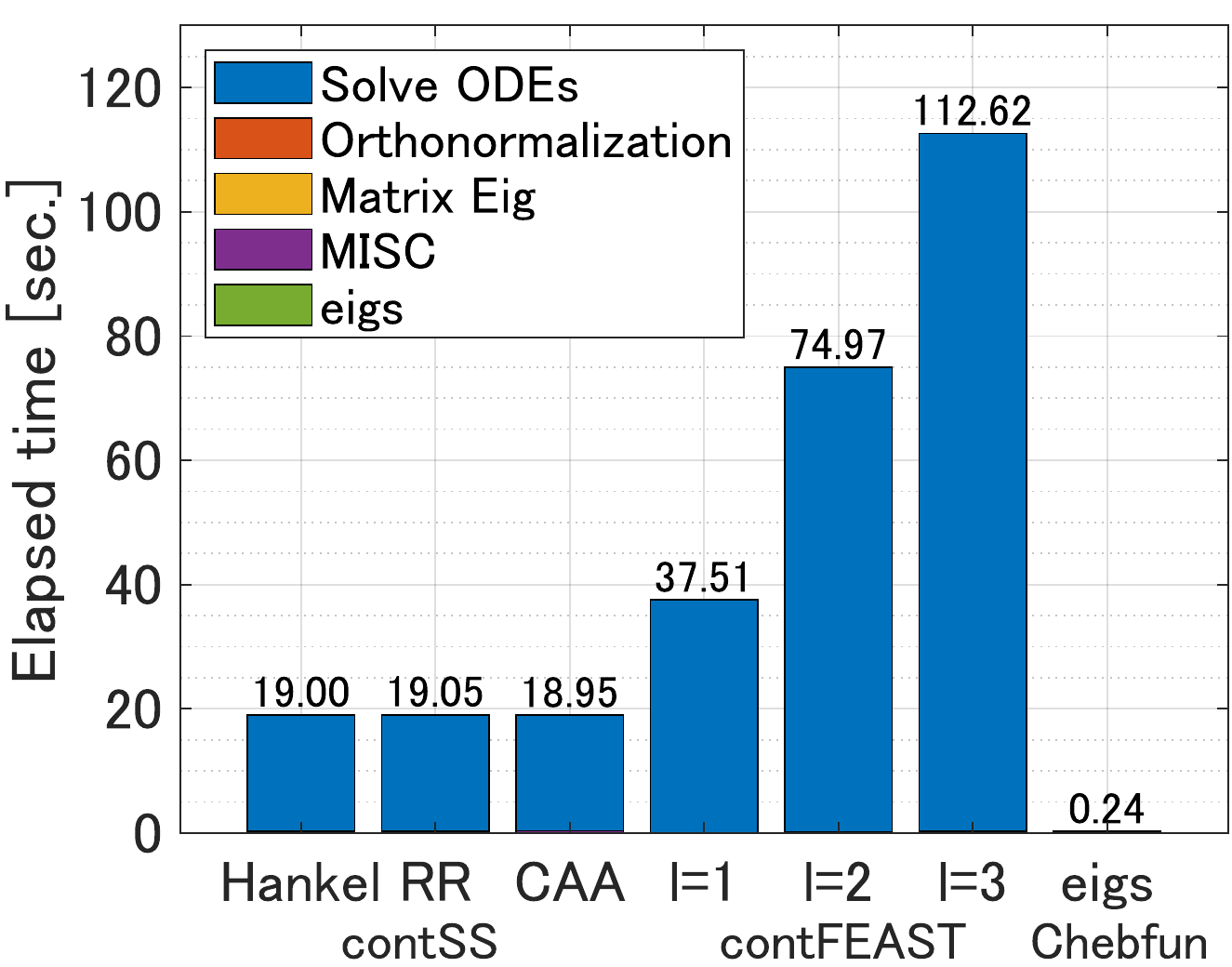}
		\subcaption{Orr--Sommerfeld eigenvalue problem ($Re=1000$)}
	\end{minipage}
	\begin{minipage}{0.49\hsize}
		\centering
		\includegraphics[width=0.8\textwidth]{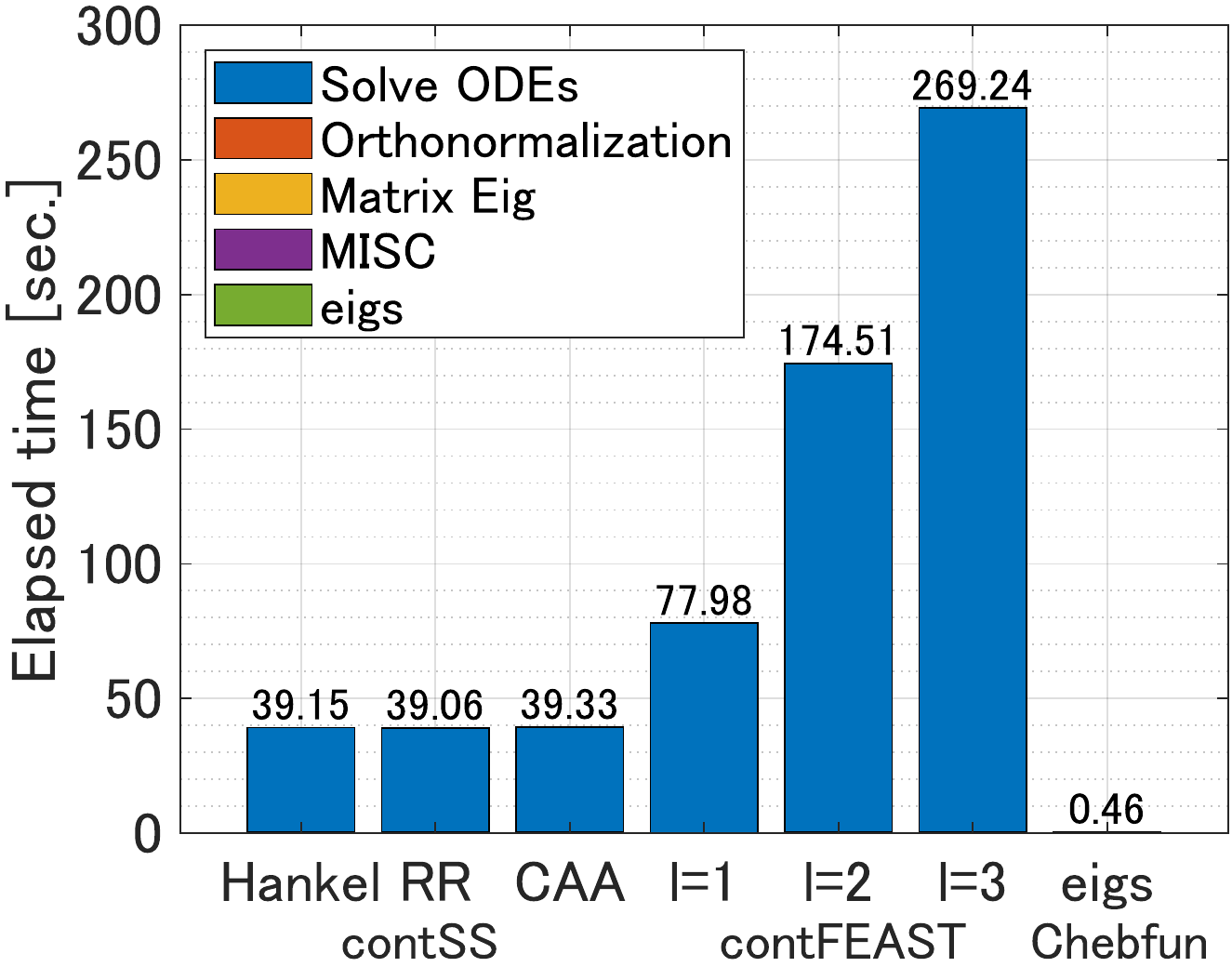}
		\subcaption{Orr--Sommerfeld eigenvalue problem ($Re=2000$)}
	\end{minipage}
	\caption{Elapsed time for each problem.}
	\label{fig:ex3_time}
\end{figure}
\par
For the complex problems (Fig.~\ref{fig:ex3_CO_res}), the residual norms of contFEAST stagnate at $\| r_i \|_\mathcal{H} \approx 10^{-7}$ for $Re=1000$ and $\| r_i \|_\mathcal{H} \approx 10^{-6}$ for $Re=2000$ in $\ell=2$.
ContSS-RR achieves almost the same accuracy as contFEAST; on the other hand, SS-Hankel and contSS-CAA are less accurate than contFEAST and contSS-RR.
\par
Next, we discuss the elapsed times of the methods (Fig.~\ref{fig:ex3_time}).
For the complex moment-based methods, most of the elapsed time is spent on solving the ODEs.
The total elapsed time of contFEAST increases in proportion to the number of iterations $\ell$.
Although contSS-RR and contSS-CAA account for larger portions of elapsed time for orthonormalization of the basis functions of $\mathscr{R}(\widehat{S})$ because they use a larger dimensional subspace (Section~\ref{sec:adv_contFEAST}), the proposed methods exhibit much less total elapsed times than contFEAST.
The proposed methods are over eight times faster than contFEAST with $\ell=3$ for real problems and over four times faster than contFEAST with $\ell=2$ for complex problems, while maintaining almost the same high accuracy.
\par
We also compare the performance of the proposed methods with that of the ``eigs'' function in Chebfun.
As shown in Fig.~\ref{fig:ex3_time}, the ``eigs'' function is much faster than the proposed methods and contFEAST.
On the other hand, Figs.~\ref{fig:ex3_RE_res} and \ref{fig:ex3_CO_res} show that the ``eigs'' function exhibits significant losses of accuracy in several cases ($\| r_i \|_\mathcal{H} \approx 10^{-5}$ for the Bessel eigenvalue problem, $\| r_i \|_\mathcal{H} \approx 10^{-4}$ for the Orr--Sommerfeld eigenvalue problems with $Re=1000$, and $\| r_i \|_\mathcal{H} \approx 10^{-2}$ for the Orr--Sommerfeld eigenvalue problems with $Re=2000$) and is unrobust in accuracy relative to the complex moment-based methods.
\subsection{Experiment IV: parallel performance}
	\label{sec:parallel}
	As demonstrated in Section~4.3, the most time-consuming part of the complex moment-based methods is the solutions of $LN$ ODEs \eqref{eq:ode}.
	Since these $LN$ ODEs can be solved independently, the methods are expected to have high parallel performance.
	\par
	Here, we estimated the strong scalability of the methods by using the following performance model.
	We assume that the elapsed time $T_\textrm{ODE}^{(j)}$ for solving ODEs \eqref{eq:ode} depends on the quadrature point $z_j$ but is independent of the right-hand side $\mathcal{B} v_i$.
	We also assume that the elapsed time~$T_\textrm{QP}$ for other computation at each quadrature point is independent of the quadrature point $z_j$.
	In addition, we let $T_\textrm{other}$ be the elapsed time for computation of other parts in each method, respectively.
	The $LN$ ODEs are solved in parallel by $P$ processes, computations at quadrature points are parallelized in $\min(P,N)$ processes, and other parts are computed in serial.
	\par
	Then, using the measured elapsed times $T_\textrm{ODE}^{(j)}, T_\textrm{QP}$, and $T_\textrm{other}$, we estimate the total elapsed time $T_\textrm{total}(P)$ of each method in $P$ processes as
	\begin{equation*}
		T_\textrm{total}(P) = \left\{
		\begin{array}{ll}
			\displaystyle
			\max_{p=1,2,\ldots,P} \ell \left( \sum_{j \in \mathscr{J}_p} L T_\textrm{ODE}^{(j)} + T_\textrm{QP} \right) + T_\textrm{other} & \quad (P \leq N), \\ \\
			\displaystyle
			\max_{j=1,2,\ldots,N} \ell \lceil LN/P \rceil T_\textrm{ODE}^{(j)} + T_\textrm{QP} + T_\textrm{other} & \quad (P > N),
		\end{array}
		\right.
	\end{equation*}
	where $\mathscr{J}_p$ is the index set of quadrature points equally assigned to each process $p$ and $\lceil \cdot \rceil$ denotes the ceiling function.
	\begin{figure}[t]
		\centering
		\begin{minipage}{0.49\hsize}
			\centering
			\includegraphics[width=0.8\textwidth]{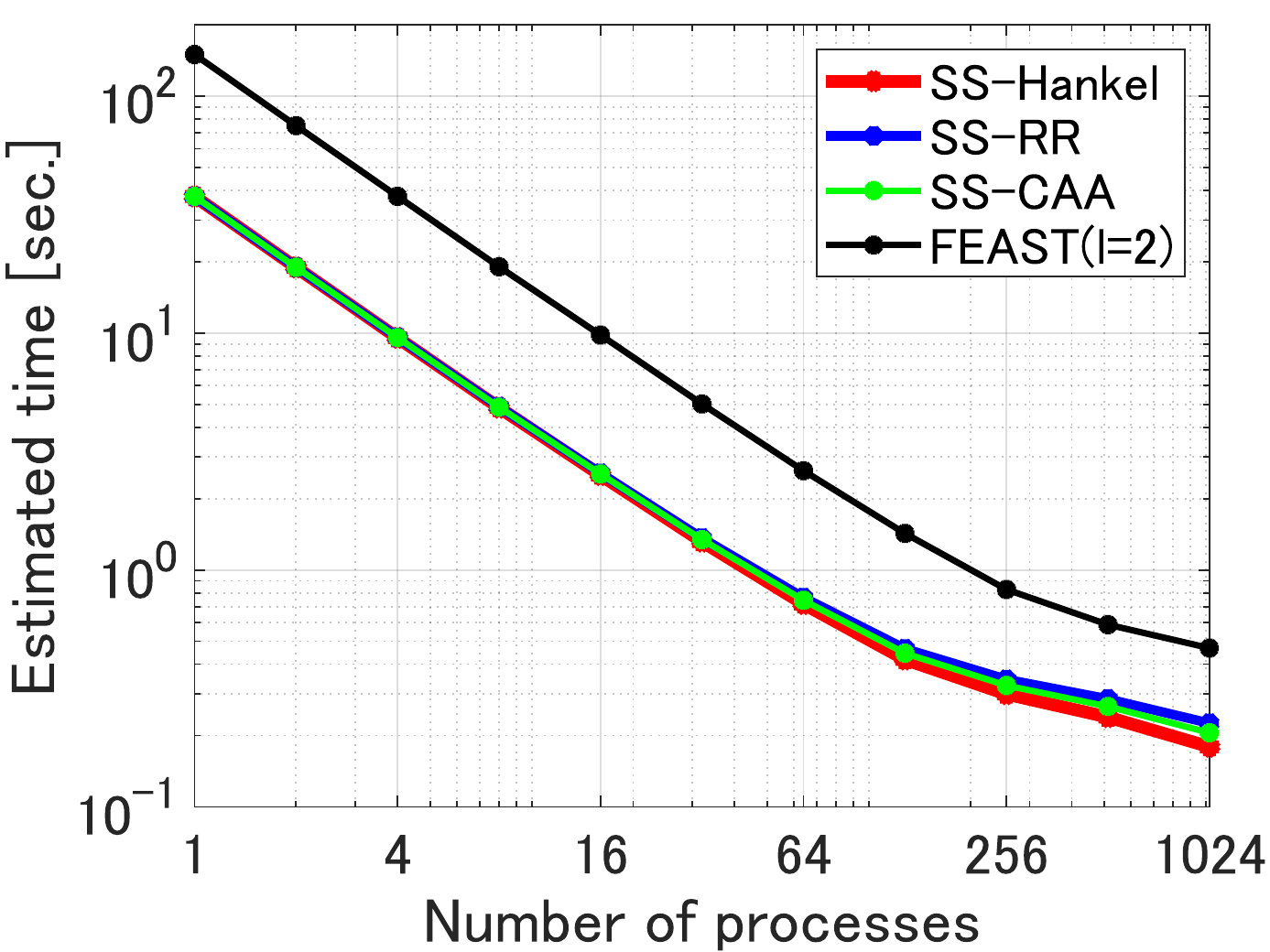}
			\subcaption{Estimated elapsed time}
		\end{minipage}
		\begin{minipage}{0.49\hsize}
			\centering
			\includegraphics[width=0.8\textwidth]{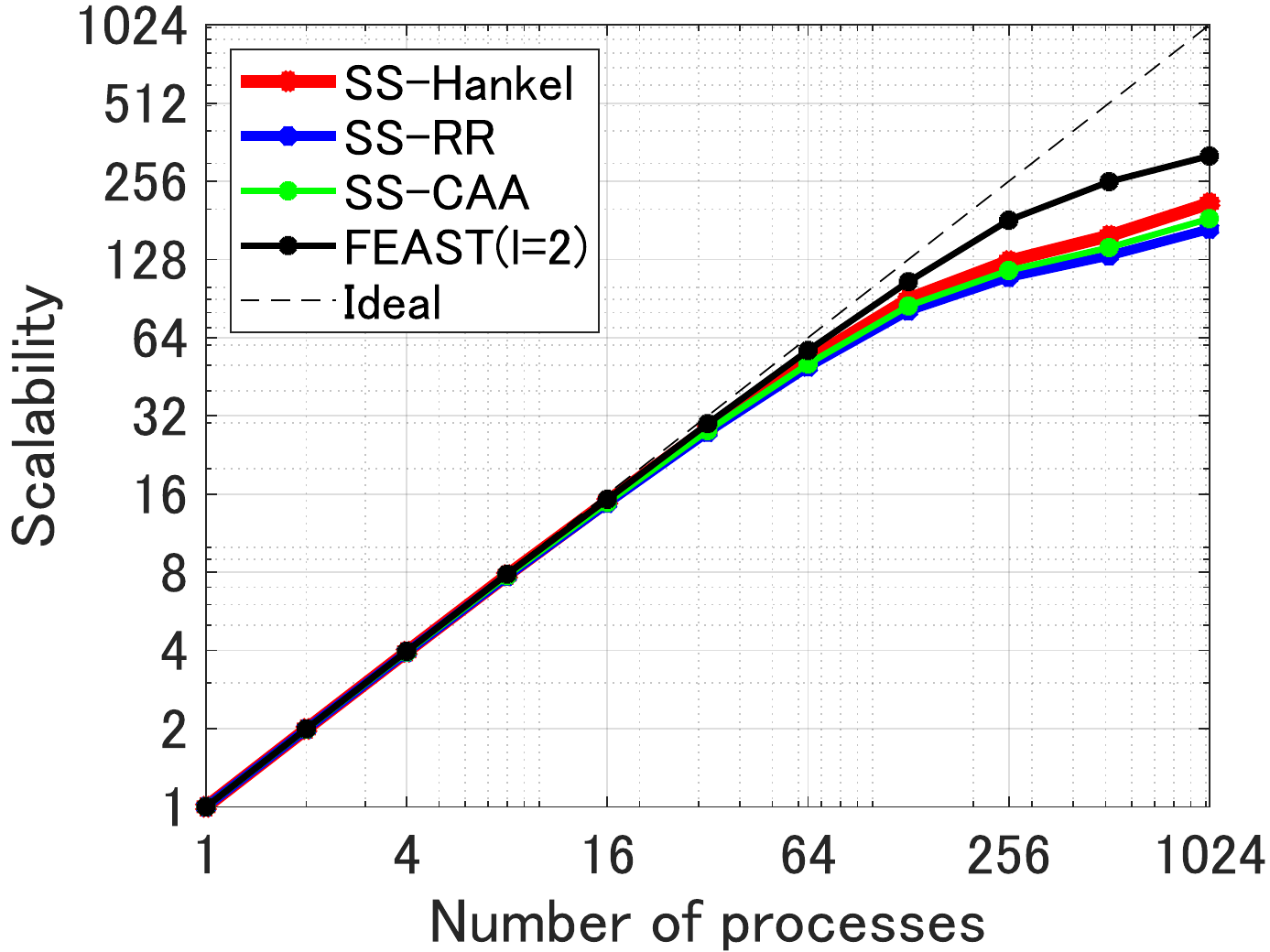}
			\subcaption{Scalability}
		\end{minipage}
		\caption{Estimated time and strong scalability for Orr--Sommerfeld eigenvalue problem with $Re=2000$.}
		\label{fig:scalability}
	\end{figure}
	\par
	We estimated the strong scalability of methods for solving the Orr--Sommerfeld eigenvalue problem with $Re=2000$.
	We used the same parameter values as in Section~4.3.
	Fig.~\ref{fig:scalability} shows the estimated time and strong scalability of methods.
	This result demonstrates that all the methods exhibit highly parallel performance.
	The proposed methods, especially contSS-Hankel, are much faster than contFEAST even with a large number of processes $P$, although contFEAST shows slightly better scalability than the proposed methods.
\subsection{Experiment V: performance for partial differential operators}
\label{sec:partial_differential}
The complex moment-based methods can be extended to partial differential operators in a straightforward manner in which $L$ PDEs are solved regarding each quadrature point.
\par
Here, we evaluate the performances of the proposed methods without iteration ($\ell=1$) and compare them with that of contFEAST for two real self-adjoint problems:
\begin{itemize}
	\item 2D Laplace eigenvalue problem:
	\begin{equation*}
		- \frac{\hbar}{2m} \left( \frac{ \partial^2 }{ \partial x^2 } + \frac{ \partial^2 }{ \partial y^2 } \right) u = \lambda u
	\end{equation*}
	in a domain $[0,\pi] \times [0,\pi]$ with zero Dirichlet boundary condition.
	The true eigenvalues are $i_x^2 + i_y^2$ with $i_x, i_y  \in \mathbb{Z}_+$.
	We computed 4 eigenpairs, counting multiplicity, corresponding to $\lambda_i \in [0, 9]$.
	Note that the target eigenvalues are $2, 5$, and $8$, where the eigenvalue $5$ has multiplicity $2$.
	\item 2D Schr\"{o}dinger eigenvalue problem:
	\begin{equation*}
		\left[ - \frac{\hbar}{2m} \left( \frac{ \partial^2 }{ \partial x^2 } + \frac{ \partial^2 }{ \partial y^2 } \right) + V(x,y) \right] u = \lambda u
	\end{equation*}
	\begin{figure}[!t]
		\centering
			\captionof{table}{True and obtained eigenvalues of the contSS-RR method for the 2D-Laplace eigenvalue problem.}
			\begin{tabular}{crc} 
				\toprule
				True eigenvalue & \multicolumn{1}{c}{Obtained eigenvalue} & Absolute error \\ \cmidrule{1-3}
				2.0 &  1.999999999999963 & $3.73 \times 10^{-14}$ \\
				5.0 &  4.999999999999198 & $8.02 \times 10^{-13}$ \\
				5.0 &  4.999999999999917 & $8.34 \times 10^{-14}$ \\
				8.0 &  7.999999999999917 & $8.34 \times 10^{-14}$ \\
				\bottomrule
			\end{tabular}
			\label{table:Laplace_2d_eig}
		\vspace{1em}
		\begin{subfigure}{\textwidth}
			\centering
			\hfill
			\includegraphics[width=0.4\linewidth]{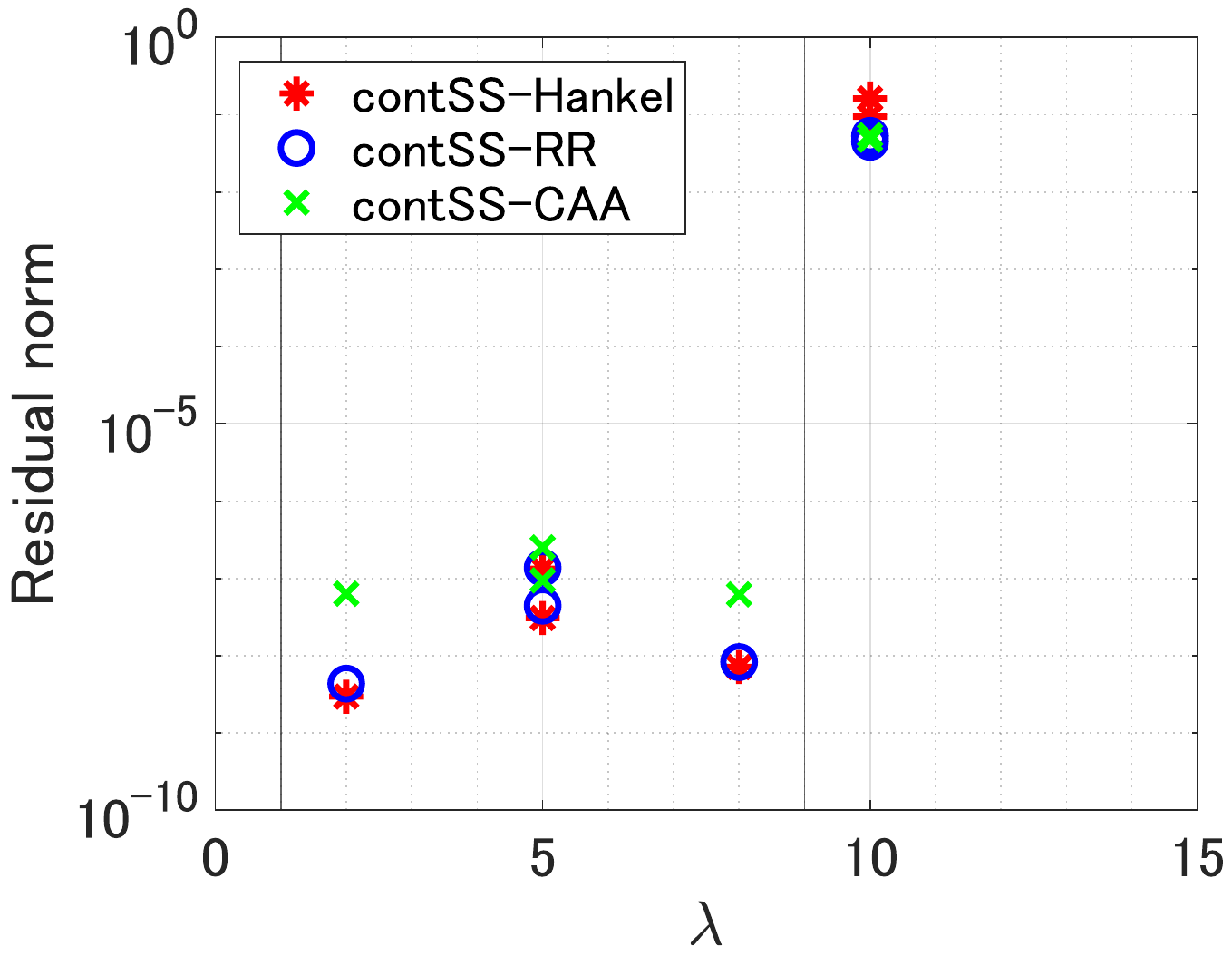}
			\hfill
			\hfill
			\includegraphics[width=0.4\linewidth]{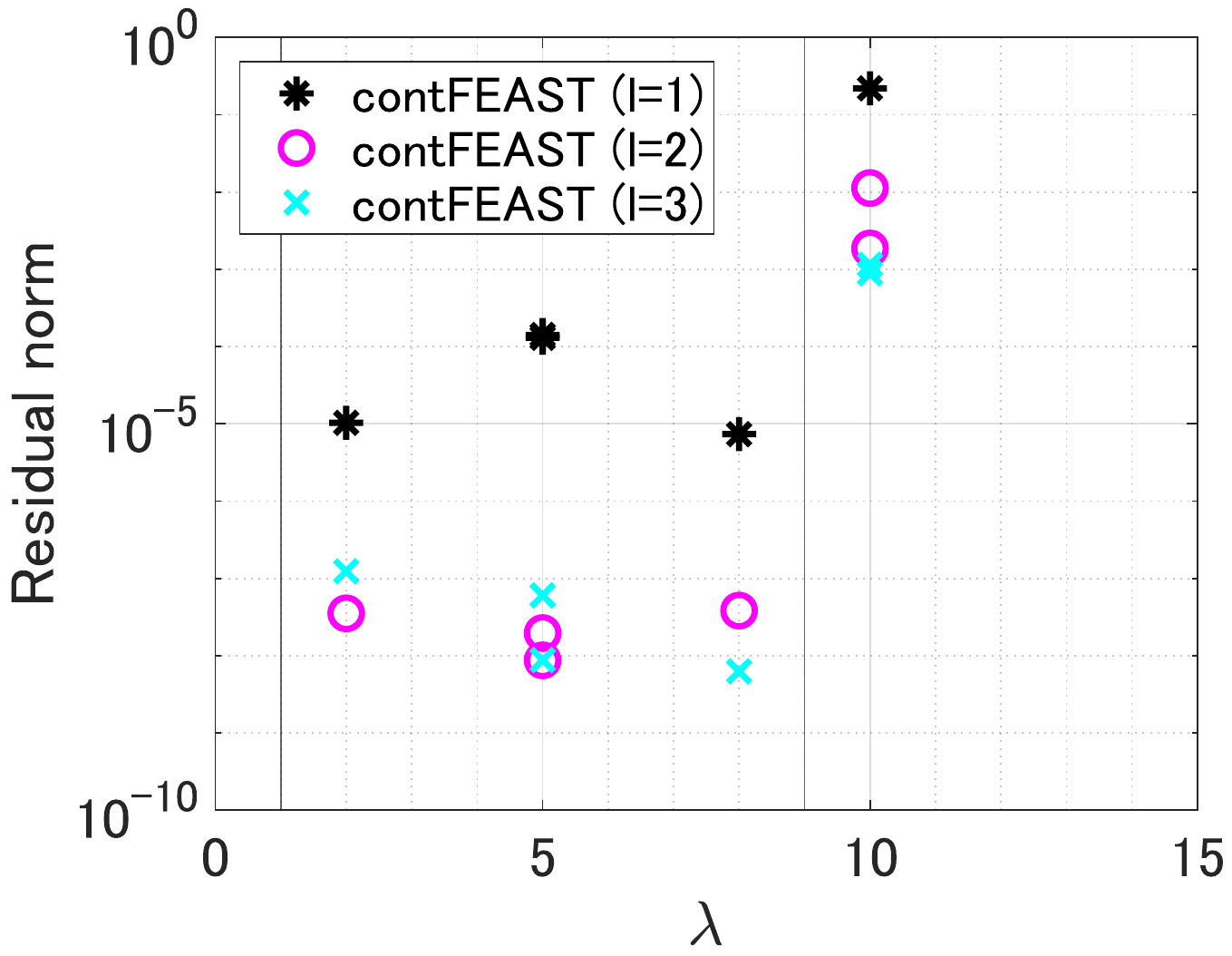}
			\hfill
			\caption{2D Laplace eigenvalue problem}
		\end{subfigure}
		\begin{subfigure}{\textwidth}
			\centering
			\hfill
			\includegraphics[width=0.4\linewidth]{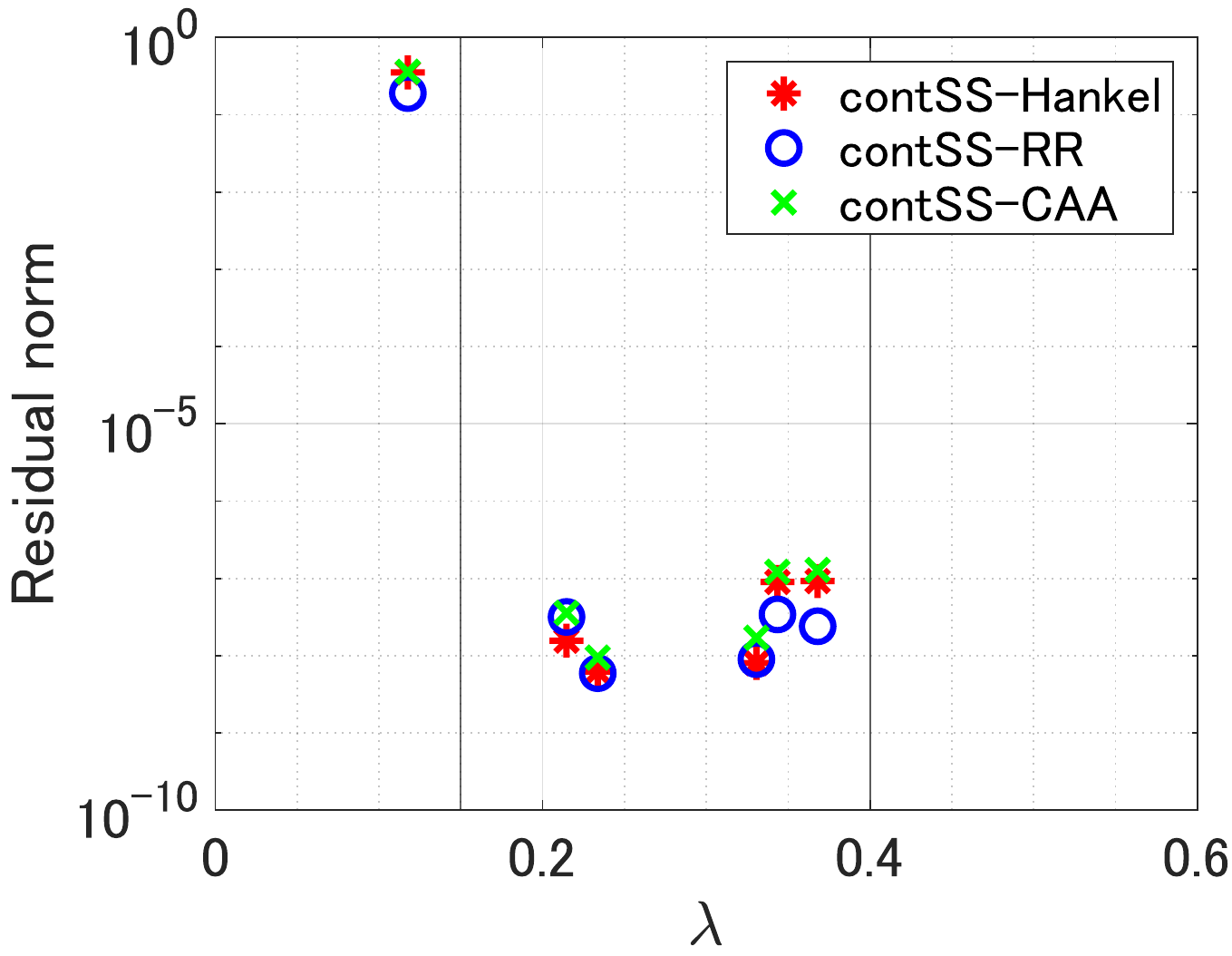}
			\hfill
			\hfill
			\includegraphics[width=0.4\linewidth]{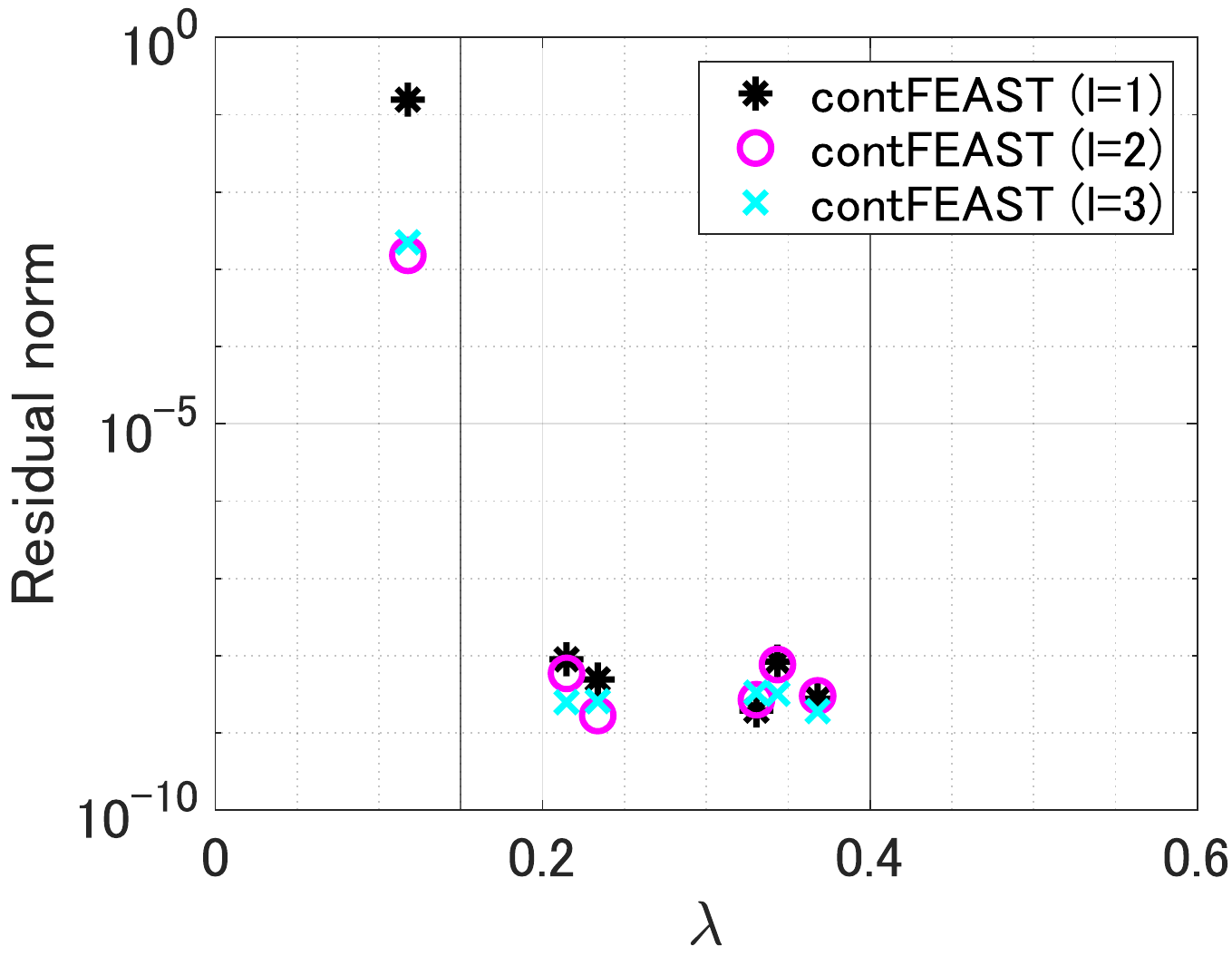}
			\hfill
			\caption{2D Schr\"{o}dinger eigenvalue problem}
		\end{subfigure}
		\caption{Residual norm for 2D problems.}
		\label{fig:ex4_res}
	\end{figure}
	in a domain $[-1,1] \times [-1,1]$ with a potential $V(x) = 0.1 (x+0.4)^2 + 0.1 (y-0.8)^2$ and zero Dirichlet boundary condition, where we set $\hbar/2m = 0.01$.
	We computed 5 eigenpairs corresponding to $\lambda_i \in [0.15,0.4]$.
\end{itemize}
For both problems, we set $(L,M,N)=(2,4,24)$ for the proposed methods and $(L,N) = (6,24)$ for contFEAST.
\par
Table~\ref{table:Laplace_2d_eig} gives the obtained eigenvalues of contSS-RR for the 2D Laplace eigenvalue problem.
In addition, residual norms $\| r_i \|_\mathcal{H} = \| \mathcal{A} \widehat{u}_i - \widehat{\lambda}_i \mathcal{B} \widehat{u}_i \|_\mathcal{H}$ for each problem are presented in Fig.~\ref{fig:ex4_res} and the elapsed times for each problem are presented in Fig.~\ref{fig:ex4_time}.
\par
We observed from Table~\ref{table:Laplace_2d_eig} and Fig.~\ref{fig:ex4_res} that, as in the case of ordinary differential operators, the proposed methods work well for solving DEPs with partial differential operators even for a non-simple case (2D-Laplacian eigenvalue problem).
In addition, the proposed methods exhibit much lower elapsed times than contFEAST; see Fig.~\ref{fig:ex4_time}, although the elapsed times for orthonormalization of the column vectors of $\widehat{S}$ and construction of the matrix eigenvalue problem are relatively larger than the cases of ordinary differential operators in Section 4.3.
\begin{figure}[t]
	\centering
	\begin{minipage}{0.49\hsize}
		\centering
		\includegraphics[width=0.8\textwidth]{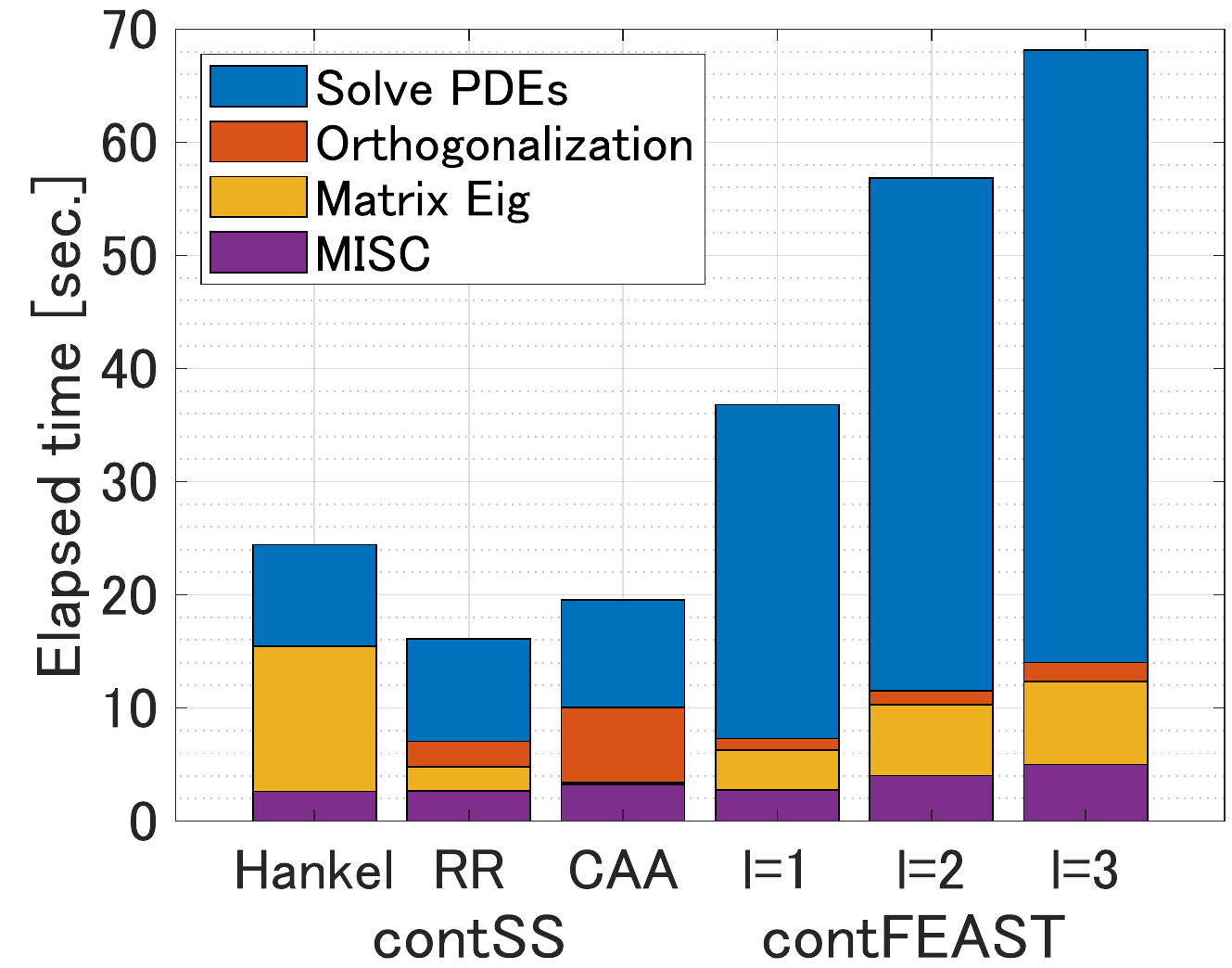}\\
		\subcaption{2D Laplace eigenvalue problem}
	\end{minipage}
	\begin{minipage}{0.49\hsize}
		\centering
		\includegraphics[width=0.8\textwidth]{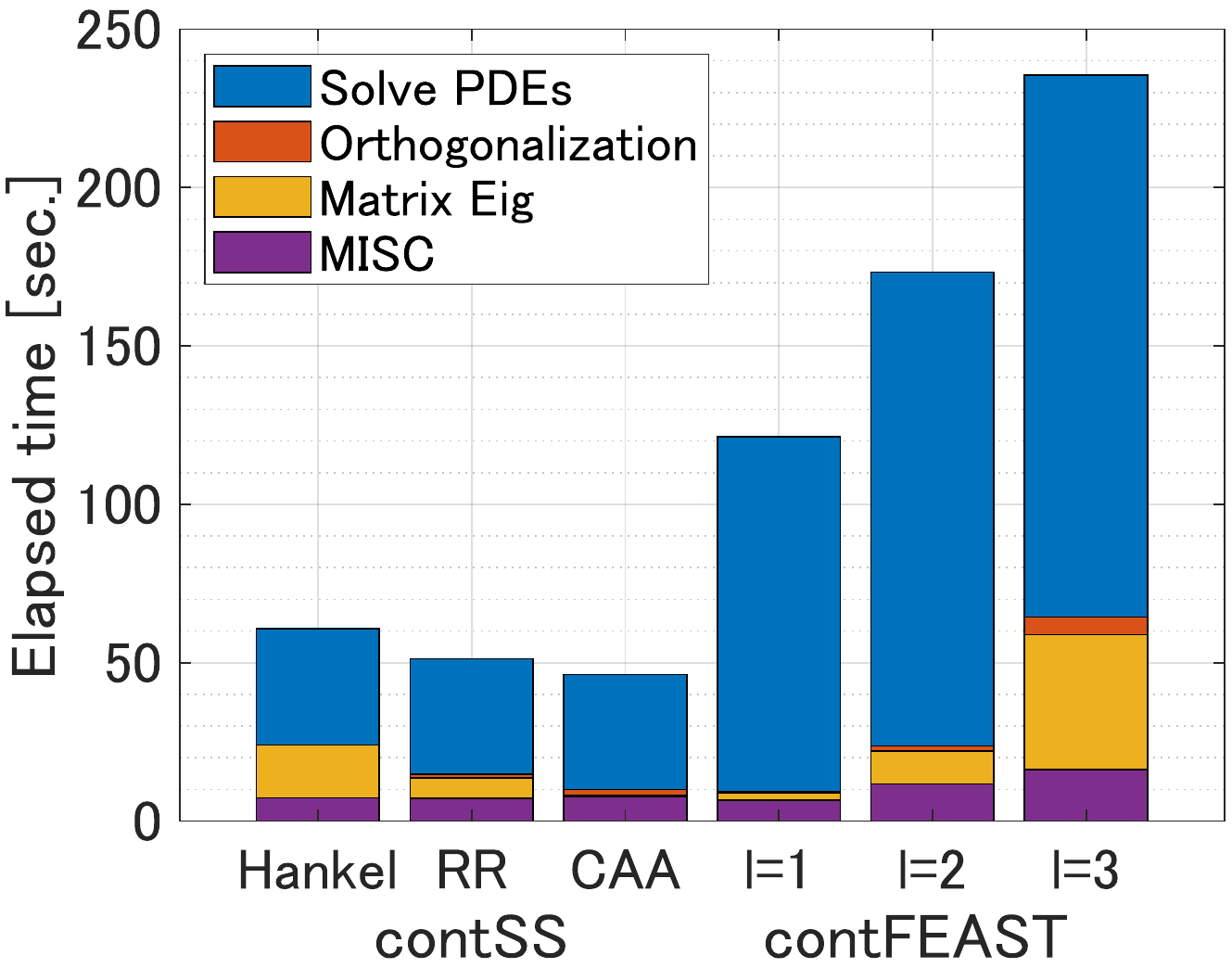}\\
		\subcaption{2D Schr\"{o}dinger eigenvalue problem}
	\end{minipage}
	\caption{Elapsed time for 2D problems.}
	\label{fig:ex4_time}
\end{figure}
\subsection{Summary of numerical experiments}
From the numerical experiments, we observed the following:
\begin{itemize}
	\item As well as contFEAST, the proposed methods in the ``solve-then-discretize'' paradigm exhibit a much higher accuracy than the ``discretize-then-solve'' approach for solving DEPs \eqref{eq:gep}.
	\item Using higher-order complex moments improves the accuracy as well as increasing the number of input functions $L$.
	\item Thanks to the higher-order complex moments, the proposed methods are over eight times faster for real problems and more than four times faster for complex problems compared with contFEAST while maintaining almost the same high accuracy.
\end{itemize}
\section{Conclusion}
\label{sec:conclusion}
In this paper, based on the ``solve-then-discretize'' paradigm, we propose operation analogues of the Sakurai--Sugiura's approach, contSS-Hankel, contSS-RR, and contSS-CAA, for DEPs \eqref{eq:gep}, without discretization of operators $\mathcal{A}$ and $\mathcal{B}$.
Theoretical and numerical results indicate that the proposed methods significantly reduce the number of ODEs to solve and elapsed time by using higher-order complex moments while maintaining almost the same high accuracy as contFEAST.
\par
As well as contFEAST, the proposed methods based on the ``solve-then-discretize'' paradigm exhibit much higher accuracy than methods based on the traditional ``discretize-then-solve'' paradigm.
This study successfully reduced the computational costs of contFEAST and is expected to promote the ``solve-then-discretize'' paradigm for solving differential eigenvalue problems and contribute to faster and more accurate solutions in real-world applications.
\par
This paper did not intend to investigate a practical parameter setting, rounding error analysis and parallel performance evaluation.
In future, we will develop the proposed methods and evaluate the parallel performance specifically for higher dimensional problems.
Furthermore, based on the concept in \cite{imakura2020verified}, we will rigorously evaluate the truncation error of the quadrature and numerical errors in the proposed methods and investigate a verified computation method based on the proposed methods for differential eigenvalue problems.
\section*{Acknowledgements}
This work was supported in part by the Japan Society for the Promotion of Science (JSPS), Grants-in-Aid for Scientific Research (Nos. JP18K13453, JP19KK0255, JP20K14356, and JP21H03451).

\bibliographystyle{siamplain}
\bibliography{mybibfile}

\begin{thebibliography}{10}

\bibitem{austin2015computing}
{\sc A.~P. Austin and L.~N. Trefethen}, {\em Computing eigenvalues of real
  symmetric matrices with rational filters in real arithmetic}, SIAM Journal on
  Scientific Computing, 37 (2015), pp.~A1365--A1387,
  \url{https://doi.org/10.1137/140984129}.

\bibitem{battles2004extension}
{\sc Z.~Battles and L.~N. Trefethen}, {\em An extension of {MATLAB} to
  continuous functions and operators}, SIAM Journal on Scientific Computing, 25
  (2004), pp.~1743--1770, \url{https://doi.org/10.1137/s1064827503430126}.

\bibitem{chatelin2012eigenvalues}
{\sc F.~Chatelin}, {\em Eigenvalues of Matrices: Revised Edition}, SIAM,
  Philadelphia, 2012, \url{https://doi.org/10.1137/1.9781611972467}.

\bibitem{driscoll2008chebop}
{\sc T.~A. Driscoll, F.~Bornemann, and L.~N. Trefethen}, {\em The chebop system
  for automatic solution of differential equations}, BIT Numerical Mathematics,
  48 (2008), pp.~701--723, \url{https://doi.org/10.1007/s10543-008-0198-4}.

\bibitem{driscoll2014chebfun}
{\sc T.~A. Driscoll, N.~Hale, and L.~N. Trefethen}, {\em Chebfun {G}uide},
  2014.

\bibitem{gilles2019continuous}
{\sc M.~A. Gilles and A.~Townsend}, {\em Continuous analogues of {Krylov}
  subspace methods for differential operators}, SIAM Journal on Numerical
  Analysis, 57 (2019), pp.~899--924, \url{https://doi.org/10.1137/18M1177810}.

\bibitem{guttel2015zolotarev}
{\sc S.~G\"{u}ttel, E.~Polizzi, P.~T.~P. Tang, and G.~Viaud}, {\em Zolotarev
  quadrature rules and load balancing for the {FEAST} eigensolver}, SIAM
  Journal on Scientific Computing, 37 (2015), pp.~A2100--A2122,
  \url{https://doi.org/10.1137/140980090}.

\bibitem{hoemmen2010communication}
{\sc M.~Hoemmen}, {\em Communication-avoiding {Krylov} subspace methods}, Tech.
  Report UCB/EECS-2010-37, University of California, Berkeley, 2010.

\bibitem{horning2020feast}
{\sc A.~Horning and A.~Townsend}, {\em {FEAST} for differential eigenvalue
  problems}, SIAM Journal on Numerical Analysis, 58 (2020), pp.~1239--1262,
  \url{https://doi.org/10.1137/19M1238708}.

\bibitem{huang2021efficient}
{\sc T.-M. Huang, W.~Liao, W.-W. Lin, and W.~Wang}, {\em An efficient contour
  integral based eigensolver for {3D} dispersive photonic crystal}, Journal of
  Computational and Applied Mathematics, 395 (2021), p.~113581,
  \url{https://doi.org/10.1016/j.cam.2021.113581}.

\bibitem{ikegami2010contour}
{\sc T.~Ikegami and T.~Sakurai}, {\em Contour integral eigensolver for
  {non-Hermitian} systems: a {Rayleigh--Ritz-type} approach}, Taiwanese Journal
  of Mathematics,  (2010), pp.~825--837,
  \url{https://doi.org/10.11650/twjm/1500405869}.

\bibitem{ikegami2010filter}
{\sc T.~Ikegami, T.~Sakurai, and U.~Nagashima}, {\em A filter diagonalization
  for generalized eigenvalue problems based on the {Sakurai--Sugiura}
  projection method}, Journal of Computational and Applied Mathematics, 233
  (2010), pp.~1927--1936, \url{https://doi.org/10.1016/j.cam.2009.09.029}.

\bibitem{imakura2014block}
{\sc A.~Imakura, L.~Du, and T.~Sakurai}, {\em A block {Arnoldi-type} contour
  integral spectral projection method for solving generalized eigenvalue
  problems}, Applied Mathematics Letters, 32 (2014), pp.~22--27,
  \url{https://doi.org/10.1016/j.aml.2014.02.007}.

\bibitem{imakura2016error}
{\sc A.~Imakura, L.~Du, and T.~Sakurai}, {\em Error bounds of {Rayleigh--Ritz}
  type contour integral-based eigensolver for solving generalized eigenvalue
  problems}, Numerical Algorithms, 71 (2016), pp.~103--120,
  \url{https://doi.org/10.1007/s11075-015-9987-4}.

\bibitem{imakura2016relationships}
{\sc A.~Imakura, L.~Du, and T.~Sakurai}, {\em Relationships among contour
  integral-based methods for solving generalized eigenvalue problems}, Japan
  Journal of Industrial and Applied Mathematics, 33 (2016), pp.~721--750,
  \url{https://doi.org/10.1007/s13160-016-0224-x}.

\bibitem{imakura2017structure}
{\sc A.~Imakura, Y.~Futamura, and T.~Sakurai}, {\em Structure-preserving
  technique in the block {SS--Hankel} method for solving {Hermitian}
  generalized eigenvalue problems}, in International Conference on Parallel
  Processing and Applied Mathematics, Cham, 2017, Springer, pp.~600--611,
  \url{https://doi.org/10.1007/978-3-319-78024-5_52}.

\bibitem{imakura2019complex}
{\sc A.~Imakura, M.~Matsuda, X.~Ye, and T.~Sakurai}, {\em Complex moment-based
  supervised eigenmap for dimensionality reduction}, in Proceedings of the AAAI
  Conference on Artificial Intelligence, 2019, pp.~3910--3918,
  \url{https://doi.org/10.1609/aaai.v33i01.33013910}.

\bibitem{imakura2020verified}
{\sc A.~Imakura, K.~Morikuni, and A.~Takayasu}, {\em Verified partial
  eigenvalue computations using contour integrals for {H}ermitian generalized
  eigenproblems}, Journal of Computational and Applied Mathematics, 369 (2020),
  p.~112543, \url{https://doi.org/10.1016/j.cam.2019.112543}.

\bibitem{imakura2017block}
{\sc A.~Imakura and T.~Sakurai}, {\em Block {Krylov-type} complex moment-based
  eigensolvers for solving generalized eigenvalue problems}, Numerical
  Algorithms, 75 (2017), pp.~413--433,
  \url{https://doi.org/10.1007/s11075-016-0241-5}.

\bibitem{iwase2017efficient}
{\sc S.~Iwase, Y.~Futamura, A.~Imakura, T.~Sakurai, and T.~Ono}, {\em Efficient
  and scalable calculation of complex band structure using {Sakurai--Sugiura}
  method}, in SC'17 Proceedings of the International Conference for High
  Performance Computing, Networking, Storage and Analysis, no.~40, 2017,
  pp.~1--12, \url{https://doi.org/10.1145/3126908.3126942}.

\bibitem{kato1995perturbation}
{\sc T.~Kato}, {\em Perturbation theory for linear operators (Second Edition)},
  vol.~132, Springer-Verlag, Berlin Heidelberg, 1995,
  \url{https://doi.org/10.1007/978-3-642-66282-9}.

\bibitem{kestyn2016pfeast}
{\sc J.~Kestyn, V.~Kalantzis, E.~Polizzi, and Y.~Saad}, {\em {PFEAST:} a high
  performance sparse eigenvalue solver using distributed-memory linear
  solvers}, in SC'16 Proceedings of the International Conference for High
  Performance Computing, Networking, Storage and Analysis, 2016, pp.~178--189,
  \url{https://doi.org/10.1109/SC.2016.15}.

\bibitem{kurz2020solving}
{\sc S.~Kurz, S.~Schöps, G.~Unger, and F.~Wolf}, {\em Solving {M}axwell's
  eigenvalue problem via isogeometric boundary elements and a contour integral
  method}, Mathematical Methods in the Applied Sciences, 44 (2021),
  pp.~10790--10803, \url{https://doi.org/10.1002/mma.7447}.

\bibitem{Mathieu1868}
{\sc E.~Mathieu}, {\em M\'{e}moire sur le mouvement vibratoire d'une membrane
  de forme elliptique}, Journal de Math\'{e}matiques Pures et Appliqu\'{e}es,
  13 (1868), pp.~137--203, \url{http://eudml.org/doc/234720}.

\bibitem{mohr2021full}
{\sc S.~Mohr, Y.~Nakatsukasa, and C.~Urz{\'u}a-Torres}, {\em Full operator
  preconditioning and the accuracy of solving linear systems}, arXiv preprint
  arXiv:2105.07963,  (2021), \url{https://doi.org/10.48550/arXiv.2105.07963}.

\bibitem{olver2014practical}
{\sc S.~Olver and A.~Townsend}, {\em A practical framework for
  infinite-dimensional linear algebra}, in 2014 First Workshop for High
  Performance Technical Computing in Dynamic Languages, 2014, pp.~57--62,
  \url{https://doi.org/10.1109/HPTCDL.2014.10}.

\bibitem{polizzi2009density}
{\sc E.~Polizzi}, {\em A density matrix-based algorithm for solving eigenvalue
  problems}, Physical Review B, 79 (2009), p.~115112,
  \url{https://doi.org/10.1103/physrevb.79.115112}.

\bibitem{sakurai2003projection}
{\sc T.~Sakurai and H.~Sugiura}, {\em A projection method for generalized
  eigenvalue problems using numerical integration}, Journal of Computational
  and Applied Mathematics, 159 (2003), pp.~119--128,
  \url{https://doi.org/10.1016/S0377-0427(03)00565-X}.

\bibitem{sakurai2007cirr}
{\sc T.~Sakurai and H.~Tadano}, {\em {CIRR}: a {Rayleigh--Ritz} type method
  with counter integral for generalized eigenvalue problems}, Hokkaido
  Mathematical Journal, 36 (2007), pp.~745--757,
  \url{https://doi.org/10.14492/hokmj/1272848031}.

\bibitem{schmid01:STS}
{\sc P.~J. Schmid and D.~S. Henningson}, {\em Stability and {T}ransition in
  {S}hear {F}lows}, Springer, New York, NY, 2001,
  \url{https://doi.org/10.1007/978-1-4613-0185-1}.

\bibitem{tang2014feast}
{\sc P.~T.~P. Tang and E.~Polizzi}, {\em {FEAST} as a subspace iteration
  eigensolver accelerated by approximate spectral projection}, SIAM Journal on
  Matrix Analysis and Applications, 35 (2014), pp.~354--390,
  \url{https://doi.org/10.1137/13090866X}.

\bibitem{townsend2015continuous}
{\sc A.~Townsend and L.~N. Trefethen}, {\em Continuous analogues of matrix
  factorizations}, Proceedings of the Royal Society A: Mathematical, Physical
  and Engineering Sciences, 471 (2015), p.~20140585,
  \url{https://doi.org/10.1098/rspa.2014.0585}.

\bibitem{trefethen2010householder}
{\sc L.~N. Trefethen}, {\em Householder triangularization of a quasimatrix},
  IMA Journal of Numerical Analysis, 30 (2010), pp.~887--897,
  \url{https://doi.org/10.1093/imanum/drp018}.

\bibitem{ExploringODEs2017}
{\sc L.~N. Trefethen, A.~Birkisson, and T.~A. Driscoll}, {\em Exploring ODEs},
  Society for Industrial and Applied Mathematics, Philadelphia, PA, 2017,
  \url{https://doi.org/10.1137/1.9781611975161}.

\bibitem{georgios2020comparison}
{\sc G.~Tzounas, I.~Dassios, M.~Liu, and F.~Milano}, {\em Comparison of
  numerical methods and open-source libraries for eigenvalue analysis of
  large-scale power systems}, Applied Sciences, 10 (2020),
  \url{https://doi.org/10.3390/app10217592}.

\bibitem{watson1995treatise}
{\sc G.~Watson}, {\em A Treatise on the Theory of Bessel Functions}, Cambridge
  Mathematical Library, Cambridge University Press, New York, NY, 1995.

\end{thebibliography}
\end{document}